\documentclass[reqno]{amsart}

\usepackage{tikz-cd}
\usepackage{fullpage,dsfont}
\usepackage{mathdots}
\usepackage{wasysym}

\usepackage{hyperref} 

\usepackage[nobysame,alphabetic,initials,msc-links]{amsrefs}
\usepackage{bm}

\usepackage{mathtools}
\usepackage{calc}

\DefineSimpleKey{bib}{how}
\renewcommand{\PrintDOI}[1]{%
  \href{http://dx.doi.org/#1}{{\tt DOI:#1}}%
}
\renewcommand{\eprint}[1]{#1}
\BibSpec{book}{%
    +{}  {\PrintPrimary}                {transition}
    +{.} { \PrintDate}                  {date}
    +{.} { \textit}                     {title}
    +{.} { }                            {part}
    +{:} { \textit}                     {subtitle}
    +{,} { \PrintEdition}               {edition}
    +{}  { \PrintEditorsB}              {editor}
    +{,} { \PrintTranslatorsC}          {translator}
    +{,} { \PrintContributions}         {contribution}
    +{,} { }                            {series}
    +{,} { \voltext}                    {volume}
    +{,} { }                            {publisher}
    +{,} { }                            {organization}
    +{,} { }                            {address}
    +{,} { }                            {status}
    +{,} { \PrintDOI}                   {doi}
    +{,} { \PrintISBNs}                 {isbn}
    +{}  { \parenthesize}               {language}
    +{}  { \PrintTranslation}           {translation}
    +{;} { \PrintReprint}               {reprint}
    +{.} { }                            {note}
    +{.} {}                             {transition}
    +{}  {\SentenceSpace \PrintReviews} {review}
}
\BibSpec{article}{%
    +{}  {\PrintAuthors}                {author}
    +{,} { \textit}                     {title}
    +{.} { }                            {part}
    +{:} { \textit}                     {subtitle}
    +{,} { \PrintContributions}         {contribution}
    +{.} { \PrintPartials}              {partial}
    +{,} { }                            {journal}
    +{}  { \textbf}                     {volume}
    +{}  { \PrintDatePV}                {date}
    +{,} { \issuetext}                  {number}
    +{,} { \eprintpages}                {pages}
    +{,} { }                            {status}
    +{,} { \PrintDOI}                   {doi}
    +{,} { \eprint}        {eprint}
    +{}  { \parenthesize}               {language}
    +{}  { \PrintTranslation}           {translation}
    +{;} { \PrintReprint}               {reprint}
    +{.} { }                            {note}
    +{.} {}                             {transition}
    +{}  {\SentenceSpace \PrintReviews} {review}
}
\BibSpec{collection.article}{%
    +{}  {\PrintAuthors}                {author}
    +{,} { \textit}                     {title}
    +{.} { }                            {part}
    +{:} { \textit}                     {subtitle}
    +{,} { \PrintContributions}         {contribution}
    +{,} { \PrintConference}            {conference}
    +{}  {\PrintBook}                   {book}
    +{,} { }                            {booktitle}
    +{,} { \PrintDateB}                 {date}
    +{,} { pp.~}                        {pages}
    +{,} { }                            {publisher}
    +{,} { }                            {organization}
    +{,} { }                            {address}
    +{,} { }                            {status}
    +{,} { \PrintDOI}                   {doi}
    +{,} { \eprint}        {eprint}
    +{}  { \parenthesize}               {language}
    +{}  { \PrintTranslation}           {translation}
    +{;} { \PrintReprint}               {reprint}
    +{.} { }                            {note}
    +{.} {}                             {transition}
    +{}  {\SentenceSpace \PrintReviews} {review}
}
\BibSpec{misc}{%
  +{}{\PrintAuthors}  {author}
  +{,}{ \textit}      {title}
  +{.}{ }             {how}
  +{}{ \parenthesize} {date}
  +{,} { available at \eprint}        {eprint}
  +{,}{ available at \url}{url}
  +{,}{ }             {note}
  +{.}{}              {transition}
}

\usepackage{parskip}
\makeatletter 
\def\thm@space@setup{%
 \thm@preskip=\parskip \thm@postskip=0pt
}
\def\th@remark{%
  \thm@headfont{\itshape}%
  \normalfont 
  \thm@preskip\parskip \thm@postskip=0pt
}
\makeatother

\usepackage{amssymb, amsfonts, amsxtra, amsmath}
\usepackage{mathrsfs}

\usepackage[all]{xy}

\usepackage{bbm}

\numberwithin{equation}{section}

\usepackage{chngcntr}

\newtheorem{Theorem}{Theorem}[section]
\newtheorem{Def}[Theorem]{Definition}
\newtheorem{Lem}[Theorem]{Lemma}
\newtheorem{Prop}[Theorem]{Proposition}
\newtheorem{Cor}[Theorem]{Corollary}

\newtheorem{Rem}[Theorem]{Remark}
\newtheorem{Rems}[Theorem]{Remarks}

\newtheorem{Exa}[Theorem]{Example}

\newcommand\bp{\begin{proof}}
\newcommand\ep{\end{proof}}

\mathchardef\mhyph="2D

\usepackage{physics}

\usepackage[normalem]{ulem}

\date{}

\newcommand{\C}{\mathbb{C}}
\newcommand{\N}{\mathbb{N}}
\newcommand{\R}{\mathbb{R}}
\newcommand{\Z}{\mathbb{Z}}

\newcommand{\id}{\mathrm{id}}
\newcommand{\End}{\mathrm{End}}

\newcommand{\opp}{\mathrm{op}}
\newcommand{\Hsp}{\mathcal{H}}
\newcommand{\Ksp}{\mathscr{K}}
\newcommand{\Ker}{\mathrm{Ker}}
\newcommand{\Diag}{\mathrm{Diag}}

\newcommand{\Ad}{\mathrm{Ad}}

\newcommand{\ad}{\mathrm{ad}}

\newcommand{\msA}{\mathscr{A}}
\newcommand{\msC}{\mathscr{C}}
\newcommand{\msD}{\mathscr{D}}
\newcommand{\msF}{\mathscr{F}}
\newcommand{\msI}{\mathscr{I}}
\newcommand{\msR}{\mathscr{R}}
\newcommand{\msS}{\mathscr{S}}
\newcommand{\msZ}{\mathscr{Z}}
\newcommand{\msP}{\mathscr{P}}
\newcommand{\msB}{\mathscr{B}}

\newcommand{\msJ}{\mathscr{J}}
\newcommand{\msU}{\mathscr{U}}

\newcommand{\mbX}{\mathbf{X}}
\newcommand{\mbT}{\mathbf{T}}

\newcommand{\mbY}{\mathbf{Y}}

\newcommand{\mbz}{\mathbf{z}}
\newcommand{\mbZ}{\mathbf{Z}}
\newcommand{\mbQ}{\mathbf{Q}}
\newcommand{\mbW}{\mathbf{W}}
\newcommand{\mbq}{\mathbf{q}}
\newcommand{\mbR}{\mathbf{R}}
\newcommand{\mbl}{\boldsymbol{\lambda}}
\newcommand{\mbr}{\mathbf{r}}
\newcommand{\mbn}{\mathbf{n}}
\newcommand{\mbd}{\boldsymbol{\delta}}

\newcommand{\mbS}{\mathbf{S}}

\newcommand{\mbD}{\mathbf{D}}

\newcommand{\mcA}{\mathcal{A}}
\newcommand{\mcL}{\mathcal{L}}

\newcommand{\mcO}{\mathcal{O}}
\newcommand{\mcF}{\mathcal{F}}
\newcommand{\mcE}{\mathcal{E}}
\newcommand{\mcN}{\mathcal{N}}

\newcommand{\wpi}{\widetilde{\pi}}

\newcommand{\wtheta}{\widetilde{\theta}}

\newcommand{\mfh}{\mathfrak{h}}

\newcommand{\mfu}{\mathfrak{u}}
\newcommand{\mfsu}{\mathfrak{su}}
\newcommand{\mfgl}{\mathfrak{gl}}
\newcommand{\mfn}{\mathfrak{n}}
\newcommand{\mfsl}{\mathfrak{sl}}
\newcommand{\mfs}{\mathfrak{s}}

\newcommand{\msRtimes}{\underset{\msR}{\otimes}}
\newcommand{\msktimes}{\otimes_{k}}
\newcommand{\kstartimes}{\otimes_{k_*}}
\newcommand{\mskctimes}{\otimes_{\kc}}

\newcommand{\bprod}{\mathop{{\prod}^{\mathrlap{b}}}}
\newcommand{\Det}{\mathrm{Det}}
\newcommand{\cop}{\mathrm{cop}}

\newcommand{\nc}{\mathrm{nc}}
\newcommand{\Spec}{\mathrm{Spec}}
\newcommand{\transp}{\mathrm{tr}}
\newcommand{\kc}{\overline{k}}
\newcommand{\ext}{\mathrm{ext}}

\newcommand{\Ind}{\mathrm{Ind}}
\newcommand{\Rep}{\mathrm{Rep}}

\begin{document}

\title{The field of quantum $GL(N,\C)$ in the C$^*$-algebraic setting}

\author{K. De Commer}
\address{Vakgroep wiskunde, Vrije Universiteit Brussel (VUB), B-1050 Brussels, Belgium}
\email{kenny.de.commer@vub.be, mflore@vub.ac.be}
\author{M. Flor\'{e}}
\date{}

\maketitle

\begin{abstract}
Given a unital $*$-algebra $\msA$ together with a suitable positive filtration of its set of irreducible bounded representations, one can construct a C$^*$-algebra $A_0$ with a dense two-sided ideal $A_c$ such that $\msA$ maps into the multiplier algebra of $A_c$. When the filtration is induced from a central element in $\msA$, we say that $\msA$ is an s*-algebra. We also introduce the notion of $\msR$-algebra relative to a commutative s$^*$-algebra $\msR$, and of Hopf $\msR$-algebra. We formulate conditions such that the completion of a Hopf $\msR$-algebra gives rise to a continuous field of Hopf C$^*$-algebras over the spectrum of $R_0$. We apply the general theory to the case of quantum $GL(N,\C)$ as constructed from the FRT-formalism. 
\end{abstract}

\section*{Introduction}

Let $GL(N,\C)$ be the group of invertible $N$-by-$N$ complex matrices. Viewing $GL(N,\C)$ as an affine real variety, we can associate to it the unital $*$-algebra $\mcO^{\R}(GL(N,\C))$ of regular functions, which is generated by the coordinate functions $X_{ij}$ and their complex conjugates $X_{ij}^*$ together with the inverse of the determinant $\Det(X)^{-1}$ and its adjoint $(\Det(X)^*)^{-1}$. Through the FRT-formalism \cite{FRT89,DSWZ92}, one can deform $\mcO^{\R}(GL(N,\C))$ as a Hopf $*$-algebra into a Hopf $\C[\mbq,\mbq^{-1}]$-algebra $\mcO^{\R}_\mbq(GL(N,\C))$ over the Laurent polynomials in $\mbq$. Our goal will be to complete this algebraic deformation into an analytic deformation by constructing a continuous field of Hopf C$^*$-algebras $C_{0,\mbq}(GL(N,\C))$ over the positive reals, having $C_0(GL(N,\C))$ as the fiber at $\mbq = 1$. 

The similar result for the unitary group $U(N)$, or rather $SU(N)$, was proven in \cite{Nag00}. In that case, the construction of the completion $C_{\mbq}(U(N))$ is immediate, since each element of $\mcO_\mbq(U(N))$ has uniformly bounded norm over all possible representations on Hilbert spaces, at least with values of $\mbq$ limited to a closed interval. The difficulty hence only consists in showing that the obtained field of C$^*$-algebras is continuous at $\mbq = 1$. For general semisimple compact Lie groups, continuity of the associated field of C$^*$-algebraic quantizations was shown in \cite[Theorem 1.2]{NT11}. 

For the case of $GL(N,\C)$, there is an extra initial complication since elements in $\mcO_\mbq^{\R}(GL(N,\C))$ do \emph{not} have uniformly bounded norms over all representations, so that it is a priori not even clear what the associated C$^*$-algebra should be. However, at least $\mcO_\mbq^{\R}(GL(N,\C))$ still has enough bounded representations to separate its elements. In itself this will not yet be sufficient to construct a C$^*$-algebraic completion, but one can also show that there exists a good positive filtration of the irreducible representations. This allows one to define an analogue of the C$^*$-algebra $C_{b,\mbq}(GL(N,\C))$ of uniformly bounded continuous functions and, inside of it, an analogue of the C$^*$-algebra $C_{0,\mbq}(GL(N,\C))$ of functions vanishing at infinity. We mention that such a theory of C$^*$-completions of unital $*$-algebras was studied in a more general context in \cite{Mey17}. 

To show that the irreducible representations of $\mcO_\mbq^{\R}(GL(N,\C))$ admit a good filtration, we make use of the well-known fact that there is a quantum trace of the absolute value squared $X^*X$ of the generating matrix $X$ of $\mcO_\mbq^{\R}(GL(N,\C))$ which lies in the center of $\mcO_\mbq^{\R}(GL(N,\C))$. This quantum trace is in fact the second leading term of a \emph{quantum Cayley-Hamilton identity} for $X^*X$. This is directly related to the corresponding quantum Cayley-Hamilton identity for the \emph{reflection equation algebra} \cite{P-S95,JW17}, of which the matrix coefficients of $X^*X$ generate a copy. The interplay with the reflection equation algebra will be needed to prove the continuity of our associated field of C$^*$-algebras, which is one of the harder results to achieve. We note that the reflection equation algebra also plays an important r\^{o}le in the formal quantization problem for coadjoint orbits of semisimple complex Lie groups with the Semenov-Tian-Shansky Poisson structure \cite{DM02,Mud07a}. 

We mention that in the case of $SL(N,\C)$, the dual problem of finding a continuous field of C$^*$-algebras $C_{\mbq}^*(SL(N,\C))$, quantizing the reduced group C$^*$-algebra of $SL(N,\C)$, was recently solved, in fact for any semi-simple complex Lie group, in \cite{MV18}, without however taking into account the quantum group structure. We also mention that a duality theory for fields of C$^*$-algebraic quantum groups is available through the use of continuous fields of multiplicative unitaries \cite[Section 4.2]{Bla96}, which could in principle be used to link the two approaches. This will however not be attempted in this paper.

The structure of the paper is as follows. 

In the first two sections, we develop some general theory. In the \emph{first section} we develop a theory of unital $*$-algebras endowed with a particular positive filtration on their set of irreducible bounded representations on Hilbert spaces. When the filtration is induced by an element in a matrix amplification of $\msA$, resp. by an element in the center of $\msA$, we call the resulting structure a \emph{weak s$^*$-algebra}, resp.~ \emph{s$^*$-algebra}. One can define for s$^*$-algebras a \emph{local completion}, producing a C$^*$-algebra $A_0$ of elements vanishing at infinity which we call the \emph{$C_0$-hull}. Part of the theory developed in this section is a special case of \cite[Section 7]{Mey17}, but we supplement the latter with some more working tools.

In the \emph{second section} we enhance s$^*$-algebras in various ways. We first consider a relative version of s$^*$-algebras with respect to a given \emph{commutative} s$^*$-algebra $\msR$. We call the resulting structure an \emph{$\msR$-algebra}. Its $C_0$-hull is a field of C$^*$-algebras over the spectrum of $\msR$. Next we study $\msR$-algebras endowed with a compatible coproduct, called \emph{Hopf $\msR$-algebra}. We show that the $C_0$-hull of a Hopf $\msR$-algebra gives rise to a field of C$^*$-bialgebras, and present particular conditions ensuring that this is actually a continuous field of Hopf C$^*$-algebras.

In the last two sections, we apply the general theory to the case of quantum $GL(N,\C)$ and its various associated $*$-algebras constructed through the FRT- and RE-formalism. In the \emph{third section}, we recall the main properties of the $*$-algebras we are interested in, paying particular attention to the reflection equation algebra and the associated quantum Cayley-Hamilton identity. We will also use the relation between the reflection equation algebra and the quantized enveloping algebra of the Lie algebra $\mfgl(N,\C)$ \cite{JL92}. Passing to the Borel part, this leads to a well-known quantum version of the \emph{Cholesky decomposition}.   In the \emph{fourth section}, we will prove our main theorem showing that $C_{0,\mbq}(GL(N,\C))$, as constructed from $\mcO_\mbq^{\R}(GL(N,\C))$, is a continuous field of Hopf C$^*$-algebras over the positive reals.

In an appendix, we clarify the connection of our theory with the notion of `quantum generator of a C$^*$-algebra' introduced by S.L. Woronowicz in \cite{Wor95}, and with the notion of C$^*$-hull introduced by R. Meyer in \cite{Mey17}.

\emph{Conventions}: 

We write $\prod$ for the algebraic direct product (of vector spaces, algebras,...) and $\bprod\;$ for the uniformly bounded direct product (of Banach spaces, Banach algebras, C$^*$-algebras,...) with respect to the supremum norm. For $V$ a subset of a Banach space $B$, we denote by $V^{\nc}$ the normclosure of $V$ in $B$. For $\Hsp$ a Hilbert space, we denote by $B(\Hsp)$ the C$^*$-algebra of bounded operators on $\Hsp$.

\emph{Acknowledgements}: The work of K. De Commer was partially supported by the FWO grant G.0251.15N and the grant H2020-MSCA-RISE-2015-691246-QUANTUM DYNAMICS.

\section{$C_0$-hulls and s$^*$-algebras}\label{Sec1}

Let $\msA$ be a unital $*$-algebra, that is, a unital $\C$-algebra equipped with an antilinear, antimultiplicative involution $*$. An element $a\in \msA$ is called \emph{self-adjoint} if $a^* = a$. The following definition introduces the notion of `closed spectral condition', which for the moment will be just a list of statements without content.

\begin{Def} Let $\msA$ be a unital $*$-algebra. We call \emph{closed spectral conditions} any family $\msS$ of statements of the following form: 
\[
\textrm{`The joint spectrum of the set of commuting selfadjoint}
\]
\[
\textrm{elements $Y_1,\ldots,Y_m \in M_n(\msA)$ lies in the closed set $S \subseteq \R^m$.'}
\]
We will abbreviate such a condition as $(Y_1,\ldots,Y_m;S)$. We call \emph{unital spectral $*$-algebra} any unital $*$-algebra together with a family of closed spectral conditions $\msS$.
\end{Def}

For $\pi$ a linear map on $\msA$ and $Y\in M_n(\msA)$, we will write $\pi(Y) = (\id\otimes \pi)(Y)$ for brevity. 

\begin{Def}\label{DefP} 
Let $\msA$ be a unital spectral $*$-algebra with closed spectral conditions $\msS$. We call \emph{admissible} representation of $\msA$ (with respect to $\msS$) any Hilbert space  $\Hsp_{\pi}$ endowed with a unital $*$-homomorphism $\pi: \msA \rightarrow B(\Hsp_{\pi})$ such that, for any $(Y_1,\ldots,Y_m;S) \in \msS$:
\[
\textrm{`The joint spectrum of the set of commuting selfadjoint}
\]
\[
\textrm{elements $\pi(Y_1),\ldots,\pi(Y_m) \in M_n(B(\Hsp_\pi))$ is a subset of $S \subseteq \R^m$.'}
\]
\end{Def}

Given $\msS$, we will more generally say that the joint spectrum of a set of commuting selfadjoint elements $(Y_1,\ldots,Y_m)$ in $M_n(\msA)$ lies in a set $S \subseteq \R^m$ if this is true for $(\pi(Y_1),\ldots,\pi(Y_m))$ for each admissible representation $\pi$. In particular, we write $Y\geq 0$ if $\pi(Y)$ is positive for each admissible representation $\pi$.

Let $\msA$ be a unital spectral $*$-algebra. We choose
\[
P^{\msA}_{\msS} = P^{\msA} = P = \{\pi\}
\] 
to be a maximal collection of mutually inequivalent \emph{topologically irreducible} non-zero admissible representations $\pi$ of $\msA$, that is, without proper \emph{closed} invariant subspaces. 

\begin{Def}
We call a unital spectral $*$-algebra $\msA$ \emph{C$^*$-faithful} if $P$ separates the elements in $\msA$. 
\end{Def}

By passing to a quotient, one can always assume that $\msA$ is C$^*$-faithful. Of course, this could turn $\msA$ into the zero algebra. Moreover, for concrete examples one wants that $\msA$ is C$^*$-faithful from the outset, so that no information is lost in considering the bounded representations. This excludes already a large class of unital $*$-algebras, such as for example the Weyl $*$-algebra generated by $x,x^*$ with $x^*x - xx^* = 1$, but C$^*$-faithfulness will be satisfied for the unital $*$-algebras we are ultimately interested in. 

In the following we will continue to work with a general unital spectral $*$-algebra. For $\pi$ an admissible representation of $\msA$ and $\pi' \in P$, we write $\pi' \preceq \pi$ if $\pi'$ factors through the normclosure $\pi(\msA)^{\nc}$ of $\pi(\msA)$. We write 
\[
P_{\pi} = \{\pi' \in P\mid \pi' \preceq \pi\}.
\]

Assume now that $P$ is endowed with a positive filtration $\msP$, i.e. there are given subsets $P_{\leq M}\subseteq P$ for $M>0$ with 
\[
P = \cup_{M>0} P_{\leq M} \qquad \textrm{and} \qquad P_{\leq M} \subseteq P_{\leq M'} \textrm{ for }M\leq M'.
\] 
We also write 
\[
P_{\geq M} = P\setminus \cup_{M'<M} P_{\leq M'}.
\]

\begin{Def}\label{DefCfiltr}
We call $\msP$ a \emph{C$^*$-filtration} if the following two conditions are satisfied:
\begin{itemize}
\item
For each $a\in \msA$ and $M>0$, we have $\sup_{\pi \in P_{\leq M}} \|\pi(a)\| <\infty$.

\item For each admissible representation $\pi$ there exists $M>0$ such that $P_{\pi} \subseteq P_{\leq M}$. 
\end{itemize}
\end{Def}

\begin{Def}
We say two filtrations $\msP$ and $\msP'$ of $P$ are \emph{equivalent}, $\msP \sim \msP'$, if for each $M>0$ there exists $M'\geq M$ such that $P_{\leq M} \subseteq P_{\leq M'}'$ and $P_{\leq M}' \subseteq P_{\leq M'}$. 
\end{Def}

Clearly $\msP$ is a C$^*$-filtration if and only if $\msP'$ is a $C^*$-filtration. In fact, we have the following.

\begin{Lem}
If a C$^*$-filtration exists, it is unique up to equivalence.
\end{Lem}
\begin{proof}
This follows immediately from the second condition in Definition \ref{DefCfiltr} upon considering representations of the form $\oplus_{\pi \in P_{\leq M}} \pi$, which are bounded by the first condition in Definition \ref{DefCfiltr}.
\end{proof}

\begin{Exa}\label{ExaMainExample}
Let $\msA$ be a unital spectral $*$-algebra, and let 
\[
X \in M_N(\msA) = M_N(\C)\otimes \msA,\qquad X =  (X_{ij})_{ij} = \sum_{ij} e_{ij} \otimes X_{ij}
\]
be a \emph{quantum generator} of $\msA$, that is, the entries $X_{ij}$ generate $\msA$ as a unital $*$-algebra. Then
\[
P_{\leq M}^X := \{\pi \in P \mid \|\pi(X)\|\leq M\}
\]
defines a C$^*$-filtration $\msP^X$.
\end{Exa}

In the following discussion, we fix a unital spectral $*$-algebra $\msA$ with C$^*$-filtration $\msP$. Then\footnote{We take the supremum in $[0,+\infty)$, so if $P_{\leq M} = \emptyset$ we have $\|a\|_M =0$.}
\[
\|a\|_M = \sup\{\|\pi(a)\|\mid\pi \in P_{\leq M}\}, \qquad a\in \msA
\] 
defines a C$^*$-seminorm on $\msA$. Let 
\[
\msI_{M} = \{a\in \msA\mid \|a\|_M=0\},
\] 
and write $A_{\leq M}$ for the completion of $\mathscr{A}/\msI_{M}$ with respect to the norm $\|-\|_M$. Then $A_{\leq M}$ is a unital C$^*$-algebra (possibly the zero C$^*$-algebra), equipped with a unital $*$-homomorphism 
\[
\theta_M: \msA \rightarrow A_{\leq M},
\] 
and each $\pi \in P_{\leq M}$ extends to a unital $*$-representation $\pi_M$ of $A_{\leq M}$ such that
\[
\xymatrix{ \msA \ar[r]^{\theta_M} \ar[rd]_{\pi} & A_{\leq M}  \ar[d]^{\wpi_M} \\ & B(\Hsp_{\pi})}
\] 
commutes. We obtain in this way a commuting triangle 
\[
\xymatrix{  \msA \ar[r]^{\theta_M} \ar[dr] & A_{\leq M}  \ar@{^{(}->}[d]\\  & \bprod_{\pi \in P_{\leq M}} B(\Hsp_{\pi})}
\] 
with the image of $A_{\leq M}$ the closure of the image of $\msA$ in $\bprod_{\pi \in P_{\leq M}} B(\Hsp_{\pi})$. In the following, we will identify $A_{\leq M}$ with its image inside $\bprod_{\pi \in P_{\leq M}} B(\Hsp_{\pi})$.

The following lemma extends the above discussion to arbitrary admissible representations.

\begin{Lem}\label{LemIrToRep} 
Let $\pi: \msA \rightarrow B(\Hsp_{\pi})$ be an admissible representation with $P_{\pi}\subseteq P_{\leq M}$. Then there exists a unique unital $*$-homomorphism $\pi_M: A_{\leq M} \rightarrow B(\Hsp_{\pi})$ such that the following diagram commutes:  
\[
\xymatrix{ \msA \ar[r]^{\theta_M} \ar[rd]_{\pi} & A_{\leq M}  \ar[d]^{\pi_M} \\ & B(\Hsp_{\pi})}.
\] 
\end{Lem} 

\begin{proof} 
Let $A_{\pi}$ be the closure of $\pi(\msA)$, and let $\msP_\pi$ be the set of pure states on $A_{\pi}$. For $\omega\in \msP_{\pi}$, let $(\Hsp_{\omega},\pi_{\omega})$ be the associated irreducible representation. Then $\pi_{\omega} \in P_{\leq M}$ by the second condition in Definition \ref{DefP}. Since 
\[
\|\pi(a)\| = \sup_{\omega\in \msP_{\pi}} |\omega(\pi(a))| \leq  \sup_{\omega\in \msP_{\pi}} \|\pi_{\omega}(a)\|  \leq \|a\|_M,\qquad a\in \msA,
\] 
the lemma follows.
\end{proof} 

Since $\|-\|_M \leq \|-\|_{M'}$ if $M\leq M'$, we obtain surjective C$^*$-homomorphisms 
\[
\wtheta_{M',M}:A_{\leq M'}\twoheadrightarrow A_{\leq M},\quad \theta_{M'}(a)\mapsto \theta_M(a),\qquad a\in \msA
\] 
such that the 
\[
\xymatrix{ \msA \ar[r]^{\theta_{M'}} \ar[rd]_{\theta_{M}} &  A_{\leq M'}  \ar[d]^{\wtheta_{M',M}} \\ & A_{\leq M}}
\]  
commute. These maps are obtained as restrictions of the natural projections of direct products 
\[
\xymatrix{A_{\leq M'} \ar@{->>}[d] \ar@{^{(}->}[r] &  \bprod_{\pi \in P_{\leq M'}} B(\Hsp_{\pi}) \ar@{->>}[d]  \\ A_{\leq M}  \ar@{^{(}->}[r] & \bprod_{\pi \in P_{\leq M}} B(\Hsp_{\pi}).} 
\]

Let $A_b$ be the projective limit C$^*$-algebra of the $A_{\leq M}$, 
\[
A_b = \{(x_M)_{M>0}\mid  x_M \in A_{\leq M}, \wtheta_{M',M}(x_{M'}) = x_M,\sup_M \|x_M\|<\infty\} \subseteq \bprod_{M>0} A_{\leq M}.
\] 
Then we obtain surjective $*$-homomorphisms 
\[
\wtheta_M: A_b \twoheadrightarrow A_{\leq M}.
\] 
In particular, the admissible representations of $\msA$ lead to representations $\pi_b: A_b \rightarrow B(\Hsp_{\pi})$ with $\pi_b = \pi_M \circ \wtheta_M$ for any $M$ large enough, making 
\[
\xymatrix{A_b \ar@{->>}[d] \ar@{^{(}->}[r] &  \bprod_{\pi \in P} B(\Hsp_{\pi}) \ar@{->>}[d]  \\ A_{\leq M}  \ar@{^{(}->}[r] & \bprod_{\pi \in P_{\leq M}} B(\Hsp_{\pi})} 
\]
commute.

Define for $M>0$
\[
A_{<M} = \cap_{\pi\in P_{\geq M}} \Ker(\pi_b) \subseteq A_b,
\] 
which are closed two-sided  ideals in $A_b$ with $A_{<M} \subseteq A_{<M'}$ for $M\leq M'$. Define 
\[
A_{c} = \cup_{M>0} A_{<M} =  \{x\in A_b\mid  \exists M>0, \forall \pi \in P_{\geq M}: \pi_b(x) = 0\},
\] 
which is a $*$-invariant two-sided ideal of $A_b$, and denote $A_0$ for the closure of $A_c$ inside $A_b$. Then $A_0$ is a C$^*$-algebra and a closed 2-sided ideal inside $A_b$.

\begin{Def}
We call $A_0$ the \emph{$C_0$-hull} of the spectral $*$-algebra $\msA$.
\end{Def}

Left and right multiplication give natural $*$-homomorphisms
\[
A_b \rightarrow M(A_0),\qquad A_b \rightarrow M(A_{<M}),
\]
where $M(-)$ denotes the multiplier C$^*$-algebra. The latter $*$-homomorphisms factor over $A_{\leq M}$, 
\[
\xymatrix{ A_b \ar[r] \ar[rd] & A_{\leq M}  \ar[d] \\ & M(A_{<M}).}
\] 

Let us write $\pi_0$ for the restriction of some $\pi_b$ to $A_0$. By the following lemma, all irreducible representations of $A_0$ arise from irreducible representations of $\msA$. 

\begin{Lem}\label{LemRepInd} 
Let $\pi'$ be a non-zero irreducible representation of $A_0$ on a Hilbert space $\Hsp$. Then $\pi' \cong \pi_0$ for a unique $\pi \in P$.
\end{Lem} 
\begin{proof} 
Since the $A_{<M}$ form an increasing net of closed $2$-sided ideals with union dense in $A_0$, already $\pi'_{\mid A_{<M}}$ must be irreducible non-zero (and hence non-degenerate) for some $M>0$. It follows that the extension $\wpi'$ of $\pi'$ to $A_b$ via $A_b \rightarrow  M(A_0)$ factors over $A_{\leq M}$, 
\[
\wpi': A_b \rightarrow A_{\leq M} \rightarrow M(A_{<M}) \rightarrow B(\Hsp),
\] 
as a unital representation. Write $\wpi_M'$ for the resulting representation of $A_{\leq M}$. Then we can in particular define a representation $\pi$ of $\msA$ by 
\[
\pi: \msA \rightarrow A_{\leq M} \overset{\wpi_M'}{\rightarrow} B(\Hsp),
\] 
which is irreducible since $\theta_M(\msA)$ is dense in $A_{\leq M}$ and $\pi'$ is already irreducible on $A_0$. We have that $\pi$ is admissible since it factors through $A_{\leq M}$. Hence $\pi$ is equivalent to an element in $P$. On the other hand, we clearly have $\wpi_M' = \pi_M$, hence $\wpi' = \pi_b$ and $\pi' = \pi_0$. The construction of $\pi$ is independent of the choice of $M$ above. 

To prove uniqueness, assume that $\rho \in P_{\leq M}$ with $\rho_0 \cong \pi'$. Necessarily $\rho_0$ is non-degenerate, and its extension to $A_b$ must coincide with $\rho_b$. On the other hand, this extension is equivalent to $\wpi'$. Since $\rho_b$ factors over $A_{\leq M}$, the same is true for $\wpi'$, and the above reasoning gives that $\rho$ is equivalent to $\pi$. 
\end{proof}

In general, $A_b$ can be small and even zero if $\msA$ is not C$^*$-faithful. But even when $\msA$ is C$^*$-faithful the C$^*$-algebra $A_0$ can be small or even zero. To ensure that $A_0$ is large enough, we introduce the following property.

\begin{Def} 
We say that a C$^*$-filtration $\msP$ satisfies the \emph{(compact) Tietze extension property} if for any $0<M$  the projection
\[
(\wtheta_M)_{\mid A_c}: A_c \rightarrow A_{\leq M}
\] 
is surjective. 
\end{Def} 

\begin{Lem}
Assume $\msP$ has the (compact) Tietze extension property. Then for any $0<M$ there exists $M'>M$ for which the projection
\[
(\wtheta_M)_{\mid A_{<M'}}: A_{<M'} \rightarrow A_{\leq M}
\] 
is surjective. 
\end{Lem} 
\begin{proof}
Pick $e\in A_{<M'} \subseteq  A_c$ with $\wtheta_M(e) = 1$. Then $A_be  \subseteq A_{<M'}$ while $\wtheta_M(A_be) = A_{\leq M}$. 
\end{proof}

It follows in particular that if the Tietze extension property holds, we can choose an increasing net of positive contractions $e_M \in A_c$ such that $\wtheta_M(e_M) = 1 \in A_{\leq M}$. We call the $e_M$ \emph{associated local units}. It is again clear that this property does not depend on the chosen C$^*$-filtration up to equivalence.

\begin{Lem}\label{LemRepSam} 
If $\msP$ satisfies the Tietze extension property, then $\{\pi_0 \mid \pi \in P\}$ is a maximal set of mutually inequivalent irreducible non-degenerate representations of $A_0$. 
\end{Lem} 
\begin{proof} 
If $\pi\in P$, say $\pi \in P_{\leq M}$, it follows from the Tietze extension property that $\pi_0(A_0) = \pi_M(A_{\leq M})$, where the latter contains $\pi(\msA)$ as a dense subset. Hence $\pi_0$ is irreducible. The lemma now follows directly from Lemma \ref{LemRepInd}. 
\end{proof}

\begin{Lem}\label{LemEqMulti}
Assume the Tietze extension property holds. Then $A_b = M(A_0)$. 
\end{Lem}
\begin{proof} 
By Lemma \ref{LemRepSam} , it follows that we have an embedding $A_b \subseteq M(A_0)$ induced by
\[
\xymatrix{ A_0 \ar@{^{(}->}[r] \ar@{^{(}->}[d] & A_b  \ar@{-->}[d]  \ar@{^{(}->}[ld] \\ \bprod_{\pi \in P} B(\Hsp_{\pi}) &  \ar@{_{(}->}[l] M(A_0).}
\] 
Let us show that in fact $A_b = M(A_0)$. For this, it is enough to show that the restriction of $x\in M(A_0)$ to $\bprod_{\pi \in P_{\leq M}} B(\Hsp_{\pi})$ lands in $A_{\leq M}$. However, pick any associated local unit $e \in A_c$ restricting to $1$ in $A_{\leq M}$. Then the projections of $x$ and $xe \in A_0$ into $\bprod_{\pi \in P_{\leq M}} B(\Hsp_{\pi})$ coincide, and the projection of $xe$ lies in $A_{\leq M}$. 
\end{proof}

We can now also approximate elements in $A_0$ with elements in $\msA$. For $C$ a pre-C$^*$-algebra, i.e. a dense $*$-subalgebra of a C$^*$-algebra, denote by $M(C)$ the unital $*$-algebra of multipliers of $C$ \cite[Appendix]{VDae94}. Then we have a natural $*$-homomorphism 
\[
\theta:\msA \rightarrow M(A_c),
\]
 using the fact that each $A_{<M} \subseteq \prod_{\pi \in P} B(\Hsp_{\pi})$ only has non-zero components in $\prod_{\pi \in P_{\leq M}} B(\Hsp_{\pi})$. 

\begin{Lem}\label{LemTietDens}
Assume the Tietze extension property holds, and let $\{e_M\}$ be associated local units. Then the sum of all subspaces $e_M\theta(\msA)$ is dense in $A_0$. 
\end{Lem}
\begin{proof}
Let $\msF$ be the sum of all $e_M\theta(\msA) \subseteq A_c \subseteq A_0$.  Assume $e_M \in A_{<M'}$. Then the linear map 
\[
\msA \rightarrow A_{<M'},\quad a \mapsto e_M\theta(a) 
\]
is $\|-\|_{M'}$-bounded, and hence factors through a bounded map on $A_{\leq M'}$. Since $\msA$ is dense in $A_{\leq M'}$ with respect to $\|-\|_{M'}$, it follows that $\msF$ contains all $e_Mx$ with $x\in A_b$. In particular, $\msF$ contains all $x = e_Mx$ with $x\in A_{<M}$. Hence $\msF$ contains $A_c$, and thus is dense in $A_0$.
\end{proof}

We will be interested in unital spectral $*$-algebras with C$^*$-filtration which are very closely related to the simple ones in Example \ref{ExaMainExample}. However, as we want some more flexibility, we introduce the following notions which relax the generating condition. When $X_1,\ldots,X_n$ are matrices over $\msA$, we write $X_1\oplus \ldots  \oplus X_n$ for the associated block diagonal matrix over $\msA$.

\begin{Def}
Let $\msA$ be a unital spectral $*$-algebra. We call $X \in M_{K,N}(\msA)$ a \emph{quantum control} if 
\[
P_{\leq M}^X := \{\pi \in P \mid \|\pi(X)\|\leq M\}
\]
is a C$^*$-filtration $\msP^X$ of $P$. 

We call \emph{central control} any quantum control $C_1\oplus \ldots \oplus C_n$ with each $C_i \in \msZ(\msA)$, the center of $\msA$.  

We call two rectangular matrices $X,Y$ over $\msA$ \emph{bounded equivalent} if there exist $m,M>0$ such that for each $\pi \in P$ 
\[
m\|\pi(Y)\|\leq \|\pi(X)\|\leq M\|\pi(Y)\|.
\]
\end{Def}

The following lemma collects some easy properties of quantum and central controls.

\begin{Lem}\label{LemLem}
Let $\msA$ be a unital spectral $*$-algebra. 
\begin{enumerate}
\item If $X$ is a quantum control and $Y$ is bounded equivalent with $Y$, then $Y$ is a quantum control. 
\item The matrix $X\in M_N(\msA)$ is a quantum control if and only if $X^*X$ is a quantum control. We may hence always assume that a quantum control is positive.
\item If $\oplus C_i$ is a central control, then $\sum C_i^*C_i$ is a central control. We may hence always take a central control which is a single positive element $C\in \msZ(\msA)$.
\item If $Q\in M_N(\msA)$ is invertible, $X\in M_N(\msA)$ and $Q\oplus Q^{-1} \oplus X$ a quantum control, then $Q\oplus Q^{-1} \oplus QX$ is a quantum control.  
\end{enumerate}
\end{Lem}

\begin{Def} 
Let $\msA$ be a unital spectral $*$-algebra. We call $\msA$ a \emph{weak s$^*$-algebra} if there exists a quantum control for $\msA$. We call $\msA$ an \emph{s$^*$-algebra} if there exists a central control for $\msA$. 
\end{Def}

Although it would be natural to assume also C$^*$-faithfulness, it will be technically more convenient not to require this. In any case, one can always pass to $\msA/\msI$ for $\msI =\cap_{\pi \in P} \Ker(\pi)$. 

\begin{Prop}\label{PropCentTiet}
Any s$^*$-algebra $\msA$ has the Tietze extension property. 
\end{Prop}
\begin{proof}
Let $C \in \msZ(\msA)$ be a positive central control. We may assume $\msP = \msP^{C}$. Note that there exists a unique contraction $r\in A_b$ with $\pi_b(r) = (1+\pi(C))^{-1}$ for all $\pi \in P$. Since however $C$ is central, $\pi(C)$ is a scalar $c_{\pi}\geq 0$ for each $\pi \in P$, and 
\[
\pi \in P_{\leq M} \quad \iff \quad c_{\pi} \leq M.
\]
If we now choose $f_M \in C_c((0,1])$ such that $f_M(x) = 1$ for $x \geq (1+M)^{-1}$, we see that $e_M = f_M(r)$ lies in $A_c$ and satisfies $\pi_M(e_M)= 1$. Hence the Tietze extension property holds. 
\end{proof}

In the following, we will simply write $(1+C)^{-1}$ for the element $r$ in the proof above. Note that by this same proof, we see that $(1+C)^{-1} \in A_0$. In case $\msA$ is an s$^*$-algebra with positive central control $C$, we will call the $e_M$ as in the proof above the \emph{central local units associated with $C$}. 

We will be interested to know when a weak s$^*$-algebra is actually an s$^*$-algebra. We will deduce this by means of the existence of a well-behaved trace. We first recall the following Pimsner-Popa type estimate, see for example \cite[Proposition 2.6.2]{Wat90}. We will again use shorthand notation: for $X$ an $N$-by-$N$-matrix over a vector space, we write
\[
\Tr(X) := (\Tr\otimes \id)X = \sum_i X_{ii}.
\]

\begin{Lem}\label{LemPimPop}
Let $B_0$ be a C$^*$-algebra, and $X\in M_N(B_0)$. Embed $B_0 \subseteq M_N(B_0)$ as constant diagonal matrices. Then
\[
\frac{1}{N}X^*X \leq \Tr(X^*X).
\]
\end{Lem} 

\begin{Prop}\label{PropDensr} 
 Let $\msA$ be a weak s$^*$-algebra with positive quantum control $X\in M_N(\msA)$ and assume $\Tr(X) \in \msZ(\msA)$. Then $\msA$ is an s$^*$-algebra with central control $\Tr(X)$.
\end{Prop}
\begin{proof}  
It is sufficient to show that $X$ and $\Tr(X)$ are bounded equivalent. But for any $\pi \in P$ we have 
\[
\|\pi(\Tr(X))\| = \|\Tr(\pi(X))\| \leq N \|X\|
\] 
since $\|\Tr\| = N$, while 
\[
\|\pi(X)\|\leq  N \|\Tr(\pi(X))\| = N\|\pi(\Tr(X))\|
\] 
by Lemma \ref{LemPimPop}. 
\end{proof}

\begin{Exa} 
Let $\msA$ be a weak s$^*$-algebra without closed spectral conditions and with quantum control $X$. If $P = P_{\leq M}$ for some $M$, i.e.~ $\|\pi(X)\|$ is uniformly bounded over all $\pi \in P$, then $A_b = A_0$ coincides with the usual universal C$^*$-algebraic envelope of $\msA$. Clearly $\msA$ is an s$^*$-algebra with central control $1$. 
\end{Exa}

\begin{Exa}\label{ExaComm} 
Let $\msA$ be a weak s$^*$-algebra with $\msA$ commutative. If $X$ is a positive quantum control, then $\Tr(X)$ is central and hence $\msA$ is an s$^*$-algebra. We may thus take $X = C \in \msA$.  

Now all irreducible representations are one-dimensional, and we may identify $P$ with the set of (non-zero) admissible $*$-characters $\{\chi\}$ on $\msA$. We claim that $P$ is a locally compact space with respect to pointwise convergence on $\msA$, and that moreover we have a homeomorphism $P \cong \Spec(A_0)$ by means of $\chi \mapsto \widetilde{\chi}$. 

Indeed, the above map is a well-defined set-theoretic bijection by Lemma \ref{LemRepSam}. To see that it sends opens to opens, it is sufficient to see that for any $a\in \msA$ the image of 
\[
U_{M,a,\epsilon}(\chi) = \{\omega \in P \mid \omega(C)<M,|\chi(a)-\omega(a)|< \epsilon\}
\]
is open. However, choosing any $x \in A_0$ such that $\wtheta_{M+\epsilon}(x) = \theta_{M+\epsilon}(a)$, it follows immediately that the image of $U_{M,a,\epsilon}(\chi)$ is the open 
\[
U_{M,x,\epsilon}'(\chi_0) = \{\omega\in \Spec(A_0) \mid \omega((1+C)^{-1})>(1+M)^{-1},|\chi_0(x)-\omega(x)|< \epsilon\}.
\]
Conversely, since $\theta_M(\msA)$ is dense in $A_{\leq M}$, it is not hard to see that the above $U_{M,x,\epsilon}'(\chi_0)$ form in fact a subbasis for the topology on $\Spec(A_0)$.  
\end{Exa} 

We now turn to the notion of morphism between unital spectral $*$-algebras.

\begin{Def} 
Let $\msA,\msB$ be unital spectral $*$-algebras with closed spectral conditions $\msS_\msA,\msS_\msB$. A unital $*$-homomorphism $\alpha: \msA \rightarrow \msB$ will be called \emph{admissible} if  $\pi\circ \alpha$ is $\msA$-admissible for all $\msB$-admissible representations $\pi$. 
\end{Def} 

Note that we only need to check $\msA$-admissibility of $\pi \circ \alpha$ for \emph{irreducible} $\msB$-admissible representations $\pi$.

\begin{Prop}
Let $\msA$ and $\msB$ be unital spectral $*$-algebras with C$^*$-filtration $\msP^{\msA},\msP^{\msB}$. 
Let $\alpha: \msA \rightarrow \msB$ be an admissible unital $*$-homomorphism. Then there exists
\[
F: \R^+ \rightarrow \R^+, \quad M\mapsto F(M)
\] 
satisfying 
\begin{equation}\label{EqEstMorAb}
\pi \in P_{\leq M}^{\msB} \Rightarrow \left(P^{\msA}_{\pi\circ \alpha} \subseteq P_{\leq F(M)}^{\msA} \right).
\end{equation} 
\end{Prop}
\begin{proof}
Let $\wpi_M = \oplus_{\pi \in P_{\leq M}^{\msB}}$. Then $\wpi_M \circ \alpha$ is a bounded admissible representation of $\msA$.  By the second condition in Definition \ref{DefCfiltr}, there exists $M' = F(M)$ such that $P_{\wpi_M \circ \alpha}^{\msA} \subseteq P_{\leq M'}^{\msA}$, hence $P_{\pi \circ \alpha}^{\msA}\subseteq P_{\leq M'}^{\msA}$ for all $\pi \in P_{\leq M}^{\msB}$. 
\end{proof}

Clearly any admissible $^*$-morphism $\alpha$ extends to unital C$^*$-algebra morphisms 
\[
\alpha_{N,M}: A_{\leq N} \rightarrow B_{\leq M},\qquad \theta_N^{\msA}(a) \mapsto \theta_M^{\msB}(\alpha(a)),\qquad a\in \msA, N \geq F(M).
\] 
As these are compatible with the projective system, we obtain a unital $*$-homomorphism 
\[
\alpha_b: A_b \rightarrow B_b,
\] 
and hence a map 
\[
\alpha_0: A_0  \hookrightarrow A_b \rightarrow B_b  \rightarrow M(B_0).
\] 

\begin{Def} 
We call $\alpha_0: A_0 \rightarrow M(B_0)$ the \emph{induced C$^*$-algebra morphism} from $\alpha: \msA\rightarrow \msB$.
\end{Def}

In general, the map $\alpha_0$ need not be non-degenerate. This will be the case however if $\msA$ has the Tietze extension property.

\begin{Lem}\label{LemNonDeg}
Let $\alpha: \msA \rightarrow \msB$ be an admissible morphism, and assume $\msA$ has the Tietze extension property. Then $\alpha_0: A_0 \rightarrow M(B_0)$ is non-degenerate.
\end{Lem} 
\begin{proof} 
Fix $M>0$, and let $N>0$ be large enough so that we can define $\alpha_{N,M}: A_{\leq N}\rightarrow B_{\leq M}$. Let $\{e_{K}\} \subseteq A_c$ be a system of associated local units. Then clearly $\alpha_{N,M}(\wtheta_{N}(e_N)) =1$. Hence $\alpha_0(A_c)B_{<M} = B_{<M}$, and $\alpha_0(A_0)B_0$ dense in $B_0$.  
\end{proof}

The following definition introduces a stronger notion for an admissible $*$-morphism.

\begin{Def}
We call an admissible $*$-morphism $\alpha: \msA \rightarrow \msB$ \emph{proper} if there exists $G: \R^+ \rightarrow \R^+$ such that 
\begin{equation}\label{EqEstProp}
\pi \in P_{\geq G(M)}^{\msB} \Rightarrow (P_{\pi \circ \alpha}^{\msA} \subseteq P_{\geq M}^{\msA}).
\end{equation}
\end{Def}

If $\alpha$ is proper we clearly have that the induced C$^*$-algebra morphism $\alpha_0$ maps $A_0$ into $B_0$, 
\[
\alpha_0: A_0 \rightarrow B_0.
\] 

Note that if $\alpha: \msA \rightarrow \msB$ is a surjective admissible unital $*$-homomorphism between unital spectral $*$-algebras, a C$^*$-filtration on $P^{\msA}$ induces a C$^*$-filtration on $P^{\msB}$ by 
\begin{equation}\label{EqNewFilt}
P_{\leq M}^{\msB} = \{\pi \in P^{\msB}\mid \pi \circ \alpha \in P_{\leq M}^{\msA}\}. 
\end{equation} 
Then automatically $\alpha$ is proper, and $\msB$ is a weak s$^*$-algebra or s$^*$-algebra whenever $\msA$ is.  

\begin{Prop}\label{PropSun} 
Let $\msA$ be an s$^*$-algebra, and $\alpha: \msA \rightarrow \msB$ a surjective, proper, admissible $^*$-morphism. Then $\alpha_0:A_0 \rightarrow B_0$ is surjective.
\end{Prop} 
\begin{proof}
Fix a C$^*$-filtration $\msP^{\msA}$ for $\msA$, and let $\msP^\msB$ be the C$^*$-filtration for $\msB$ defined by \eqref{EqNewFilt}. Then the maps $\theta_M^{\msB} \circ \alpha$ induce surjective C$^*$-homomorphisms
\[
\alpha_{M}: A_{\leq M}\rightarrow B_{\leq M}.
\]

Take now $y,y'\in B_{<M}$, and choose $M'$ such that $A_{<M'} \twoheadrightarrow A_{\leq M}$. Pick $x\in A_{<M'}$ such that 
\[
\alpha_{M}(\wtheta_M^{\msA}(x)) = \wtheta_M^{\msB}(y).
\] 
Note that $\alpha_0(x) \in B_{<M'}$, so we can not conclude that $\alpha_0(x) = y$ as we do not know the behaviour of $\alpha_0(x)$ on the range from $M$ to $M'$. Choose however $M''$ such that $A_{<M''} \twoheadrightarrow A_{\leq M'}$ and pick $x'\in A_{<M''}$ such that 
\[
\alpha_{M'}(\wtheta_{M'}^{\msA}(x')) = \wtheta_{M'}^{\msB}(y').
\] 
Then it is clear that $xx' \in A_c$ with $\alpha_0(xx') = yy'$. Since $B_{<M}B_{<M} = B_{<M}$, this shows that $\alpha_0(A_c) = B_c$ and hence $\alpha_0$ surjective.
\end{proof}

\section{$\msR$-algebras, $\msR$-bialgebras and Hopf $\msR$-algebras}

In this section, we build a relative version of the theory in the previous section, called (weak) $\msR$-algebra. 

Fix $k$ a commutative $\C$-algebra. When $k$ is equipped with a particular $*$-structure, we write $k_*$. When we moreover have a fixed s$^*$-algebra-structure on $k_*$ with respect to some closed spectral conditions, we write $\msR$. We will  assume that $\msR$ is C$^*$-faithful. We write $\Theta = \Spec(\msR)$ for the associated locally compact space, so the $C_0$-hull of $\msR$ becomes $R_0 = C_0(\Theta)$. Note that we have an injective unital $*$-homomorphism 
\[
\msR \rightarrow C(\Theta) = M(C_c(\Theta)),
\] 
the $*$-algebra of all continuous functions on $\Theta$. In particular, we can interpret $f(x)$ for $f\in \msR$ and $x \in \Theta$.

\begin{Def}  
We call \emph{$k_*$-algebra} any unital $*$-algebra $\msA$ with a unital $*$-homomorphism $\iota: k_* \rightarrow \msZ(\msA)$, where $\msZ(\msA)$ is the center of $\msA$. We call $(\msA,\iota)$ a \emph{weak $\msR$-algebra}, resp.~ \emph{$\msR$-algebra}, if $\msA$ is a weak s$^*$-algebra, resp.~ an s$^*$-algebra, such that $\iota:  \msR \rightarrow \msA$ is admissible.
\end{Def} 

We have for a general weak $\msR$-algebra a $*$-homomorphism 
\[
\iota_0:R_0 =C_0(\Theta) \rightarrow  A_b  \rightarrow M(A_0)
\] 
with image in $\msZ(M(A_0))$. By Lemma \ref{LemNonDeg} this map is non-degenerate if $\msA$ is an $\msR$-algebra. It follows that $A_0$ is a $C_0(\Theta)$-algebra, i.e.~ a C$^*$-algebra with a fixed non-degenerate $*$-homomorphism $\iota_0: C_0(\Theta) \rightarrow \msZ(M(A_0))$. Here, we do not assume necessarily that $\iota_0$ is injective. 

Let $\msA$ be an $\msR$-algebra. For $x \in \Theta$ we define 
\[
\msJ_x = \{\iota(f) - f(x)\mid f\in k_*\}\subseteq \msA,\qquad \msA_x = \msA/(\msA \msJ_x).
\] 
We write the quotient map as 
\[
\pi_x: \msA \rightarrow \msA_x,\quad a \mapsto a_x.
\]
Note that $\iota$ defines a partition of the set of irreducible admissible representations
\[
P^{\msA} = \sqcup_{x\in \Theta} P_x^{\msA},
\] 
where $\pi \in P_x^{\msA}$ if $\pi(\iota(f)) = f(x)\id_{\Hsp_{\pi}}$ for all $f\in k_*$. For $x\in \Theta$ we will write

On the other hand, for $A_0$ a general $C_0(\Theta)$-algebra, we can consider the fibers $(A_{0})_x =A_{0,x} = A_0/A_0J_x$ with 
\[
J_x = \{\iota_0(f)-f(x)\mid f\in C_0(\Theta)\} \subseteq M(A_0).
\] 
Write $a_x$ for the image of $a\in A_0$ inside $A_{0,x}$, and 
\[
\pi_x^0: A_0 \rightarrow A_{0,x},\quad a\mapsto a_x.
\] 
for the corresponding quotient map. 

\begin{Def}
Let $A_0$ be a $C_0(\Theta)$-algebra. We call $x \mapsto A_{0,x}$ the \emph{field of C$^*$-algebras} associated to $A_0$. 
\end{Def}

Such a field is automatically upper semicontinuous for the norm, in the sense that for each $a\in A_0$ fixed, the map $x\mapsto \|a_x\|$ is upper semicontinuous \cite[Proposition 1.2]{Rie89}. 

\begin{Def} 
Let $A_0$ be a $C_0(\Theta)$-algebra. We call $x\mapsto A_{0,x}$ a \emph{continuous field} of C$^*$-algebras if for all $a\in A_0$ the map $x\mapsto \|a_x\|$ is continuous.
\end{Def}

Assume now that $A_0$ is the C$_0$-hull of an $\msR$-algebra $\msA$.  By the discussion preceding Proposition \ref{PropSun}, each $\msA_x$ for $x\in \Theta$ inherits an s$^*$-algebra structure from $\msA$, and by properness we obtain surjective $*$-morphisms 
\[
A_{\leq M} \rightarrow A_{x,\leq M},\qquad A_c \rightarrow A_{x,c}, \qquad A_0 \rightarrow A_{x,0}.
\] 
This last map clearly factors through a surjective $*$-homomorphism
\[
\rho_x:A_{0,x}\rightarrow A_{x,0}.
\]

\begin{Lem} 
The map $\rho_x: A_{0,x} \rightarrow A_{x,0}$ is an isomorphism of C$^*$-algebras.
\end{Lem} 
\begin{proof}
We have already shown that the map is surjective. The injectivity follows immediately from the fact that the irreducible representations of $A_{x,0}$ are the $\pi_0$ with $\pi \in P^{\msA_x}$, that the  irreducible representations of $A_{0,x}$ are the $\pi_0$ with $\pi \in P_x^{\msA}$, and the equality $P^{\msA_x} = P^{\msA}_x$.  
\end{proof}

The $C_0(\Theta)$-algebra $A_0$ hence defines an upper semi-continuous field of C$^*$-algebras $x\mapsto A_{x,0}$. In the following, we will always identify $A_{x,0}$ and $A_{0,x}$ through $\rho_x$. 

\begin{Def}
We call a (weak) $\msR$-algebra $\msA$ \emph{strongly C$^*$-faithful} if $\msA$ and all $\msA_x$ are C$^*$-faithful.

We call an $\msR$-algebra $\msA$ \emph{continuous} if the associated field of C$^*$-algebras $x\mapsto A_{x,0}$ is continuous. 
\end{Def}

The following notion of morphism between $\msR$-algebras is obvious. 

\begin{Def}
We call \emph{$\msR$-morphism} between (weak) $\msR$-algebras $(\msA,\iota_{\msA})$ and $(\msB,\iota_{\msB})$ any admissible $^*$-morphism $\alpha: \msA \rightarrow \msB$ such that 
\[
\alpha\circ \iota_{\msA}  = \iota_{\msB}.
\] 
\end{Def}

Let $\msA,\msB$ be $\msR$-algebras, and let $\alpha: \msA \rightarrow \msB$ be an $\msR$-morphism. Then clearly $\alpha$ descends to admissible $*$-morphisms $\alpha_x: \msA_x \rightarrow \msB_x$. On the other hand, the completion $\alpha_0: A_0 \rightarrow M(B_0)$ is $C_0(\Theta)$-linear, and hence can be localized to non-degenerate C$^*$-homomorphisms $\alpha_{0,x}: A_{0,x} \rightarrow M(B_{0,x})$. The following lemma is not hard to verify.

\begin{Lem} 
Let $\alpha: \msA \rightarrow \msB$ be an $\msR$-morphism between $\msR$-algebras $\msA,\msB$. Then $\alpha_{x,0} = \alpha_{0,x}$.
\end{Lem} 

We now want to put coproducts on $\msR$-algebras. Since this requires working with completed balanced tensor products on the C$^*$-level, we want to have a good regularity condition avoiding the need to use different completions. The most natural condition to consider is nuclearity, but we will settle for something stronger which is easier to check representation-theoretically, namely being type $I$. This notion can easily be transported to the setting of unital spectral $*$-algebras.

\begin{Def} 
We call a unital spectral $*$-algebra $\msA$ \emph{type $I$} if for all $\pi \in P^{\msA}$ the normclosure of $\pi(\msA)$ contains a compact operator.
\end{Def} 

A (weak) s$^*$-algebra or $\msR$-algebra will be called type $I$ if its underlying spectral $^*$-algebra  is type $I$.

The following proposition follows immediately from Lemma \ref{LemRepInd} and its proof, using the Tietze extension property for the converse direction.

\begin{Prop}\label{PropTypI} 
Assume that $\msA$ is a type $I$ weak s$^*$-algebra. Then $A_0$ and every $A_{\leq M},A_{<M}$ is a type $I$ C$^*$-algebra, and in particular is nuclear. Conversely, if $\msA$ is an s$^*$-algebra and $A_0$ type $I$, then $\msA$ is type $I$. 
\end{Prop} 

We now introduce the notion of relative tensor product of weak $\msR$-algebras.

\begin{Def} 
Let $\msA,\msB$ be  weak $\msR$-algebras. The \emph{$\msR$-relative (or $\msR$-balanced) tensor product} $\msA\msRtimes \msB$ of $\msA$ and $\msB$ is the $k_*$-algebra $\msA\kstartimes \msB$ endowed with the closed spectral conditions
\[
\msS_{\msA, \msB} := \{(y_i\otimes 1;S) \mid (y_i;S) \in \msS_\msA\} \cup \{(1\otimes z_i;S) \mid (z_i;S) \in \msS_\msB\}
\] 
and, writing $P^{\msA,\msB}$ for the set of irreducible admissible representations of $\msA \kstartimes \msB$, the filtration 
\[
P^{\msA,\msB}_{\leq M} = \{\pi \mid P_{\pi_{\mid \msA}} \subseteq P^{\msA}_{\leq M} \textrm{ and }P_{\pi_{\mid \msB}}\subseteq  P^{\msB}_{\leq M}\}.
\]
\end{Def}

\begin{Lem}
The filtration $\msP^{\msA,\msB}$ is a C$^*$-filtration, and $\msA\msRtimes \msB$ is a weak s$^*$-algebra. If $\msA,\msB$ are s$^*$-algebras, then $\msA\msRtimes \msB$ is an s$^*$-algebra.
\end{Lem}
\begin{proof}
If $a\in \msA$, $b\in \msB$ and $\pi \in P^{\msA,\msB}_{\leq M}$, then 
\[
\|\pi(a\otimes b)\|\leq \|\pi(a)\| \|\pi(b)\| \leq \|a\|_{M}\|b\|_{M},
\]
so the first condition in Definition \ref{DefCfiltr} is satisfied. If $\pi$ is any admissible representation of $\msA \kstartimes \msB$, then there exists $M>0$ such that $P_{\pi_{\mid \msA}} \subseteq P^{\msA}_{\leq M}$ and $P_{\pi_{\mid \msB}}\subseteq  P^{\msB}_{\leq M}$, and clearly $P_{\pi}^{\msA,\msB} \subseteq P^{\msA,\msB}_{\leq M}$. Hence $\msP^{\msA,\msB}$ is a C$^*$-filtration.

If now $X,Y$ are respective quantum controls for $\msA$ and $\msB$, it follows immediately that $(X\otimes 1) \oplus (1\otimes Y)$ is a quantum control for $\msA\msRtimes \msB$, since an admissible representation $\pi$ of $\msA$ has $P_{\pi} \subseteq P^{\msA}_{\leq M}$ if and only if $\|\pi(X)\|\leq M$ (and similarly for $Y$ with respect to $\msB$). 

Clearly $(X\otimes 1) \oplus (1\otimes Y) = \begin{pmatrix} X\otimes 1 & 0 \\ 0 &  1\otimes Y \end{pmatrix}$ is bounded equivalent to $ \begin{pmatrix} X \otimes 1 \\  1 \otimes Y \end{pmatrix}$, hence if $X,Y$ are central then $X^*X\otimes 1+ 1\otimes Y^*Y$ is a central control for $\msA\msRtimes \msB$.
\end{proof}

Let now  $\msA,\msB$ be $\msR$-algebras, and momentarily write $\msC = \msA\msRtimes \msB$. Let us write $\underset{C_0(\Theta)}{\otimes}$ for the maximal balanced tensor product of $C_0(\Theta)$-algebras. By definition of $\msP^{\msA,\msB}$ we obtain a family of compatible  $*$-isomorphisms
\[
C_{\leq M} \cong A_{\leq M} \underset{C_0(\Theta)}{\otimes} B_{\leq M}, \qquad C_b = \underset{\longleftarrow}{\lim}^b \, C_{\leq M}.
\]
We obtain a natural map 
\[
\rho: A_0 \underset{C_0(\Theta)}{\otimes} B_0 \rightarrow A_b\underset{C_0(\Theta)}{\otimes} B_b  \rightarrow C_b = M(C_0).
\]

However, it seems unclear if this map will have range in $C_0$, as we have insufficient information on the irreducible representations for the factors $\msA,\msB$ appearing in a general $\pi \in P_{\geq M}^{\msA,\msB}$. We therefore restrict to $\msA,\msB$ type $I$ from here on.

\begin{Prop}
Let $\msA,\msB$ be type $I$ $\msR$-algebras. Then $\msC = \msA\msRtimes \msB$ is a type $I$ $\msR$-algebra, and $\rho$ gives an isomorphism $A_0 \underset{C_0(\Theta)}{\otimes} B_0 \cong C_0$. 
\end{Prop} 
\begin{proof}
Since $A_{\leq M}$ and $B_{\leq M}$ are type $I$, any $\pi \in P^{\msA,\msB}$ is of the form $\pi_1\otimes \pi_2$ where $\pi_1 \in P^{\msA}_x$, $\pi_2\in P^{\msB}_x$ for some $x\in \Theta$. Moreover, we see that $\pi \in P^{\msA,\msB}_{\leq M}$ if and only if $\pi_1\in P^{\msA}_{\leq M}$ and $\pi_2\in P^{\msB}_{\leq M}$. It follows immediately that 
\[
\rho(A_{<M} \underset{C_0(\Theta)}{\otimes} B_{<M}) \subseteq C_{<M},  
\]
and hence $\rho$ maps  $A_0 \underset{C_0(\Theta)}{\otimes} B_0$ into $C_0$. 

To see that $\rho$ is injective, it is sufficient to show that the restrictions
\[
A_{<M} \underset{C_0(\Theta)}{\otimes} B_{<M} \rightarrow C_{<M} 
\]
are injective. But this follows from the fact that 
\[
A_{<M} \underset{C_0(\Theta)}{\otimes} B_{<M} \rightarrow  C_{\leq M} = A_{\leq M} \underset{C_0(\Theta)}{\otimes} B_{\leq M}
\]
is injective (this does not use the type $I$-property, only the fact that $A_{<M}$ can be seen as an ideal in $A_{\leq M}$, and similarly $B_{<M}$ in $B_{\leq M}$). 

The surjectivity of $\rho$ follows immediately from the existence of local units as in the proof of Proposition \ref{PropSun}.
\end{proof} 

Of course, the results above can be obtained for any finite number of tensorands. 

\begin{Rem}\label{RemTypI}
Note that the conclusion $A_0 \underset{C_0(\Theta)}{\otimes} B_0 \cong C_0$ is still valid if we only assume that $\msA$ is type $I$. The same holds for multiple tensor products as long as all factors but one are type $I$.
\end{Rem}

Also the following observation is straightforwardly proved. 

\begin{Lem} 
Let $\msA,\msB,\msC,\msD$ be type $I$ $\msR$-algebras, and assume we are given $\msR$-morphisms 
\[
\alpha: \msA \rightarrow \msC,\quad \beta: \msB\rightarrow \msD.
\] 
Then the balanced tensor product 
\[
\alpha\otimes \beta: \msA\msRtimes \msB \rightarrow \msC\msRtimes \msD
\] 
is an $\msR$-morphism.
\end{Lem}

We can write the associated induced C$^*$-morphism as 
\[
\alpha_0 \otimes \beta_0:= (\alpha\otimes \beta)_0: A_0\underset{C_0(\Theta)}{\otimes} B_0 \rightarrow M(C_0\underset{C_0(\Theta)}{\otimes} D_0),
\] 
as it clearly corresponds to the universal tensor product of the morphisms $\alpha_0$ and $\beta_0$.

We can now introduce the notion of (weak) $\msR$-bialgebra. 

\begin{Def} We call \emph{$k_*$-bialgebra} any $k_*$-algebra $\msA$ together with $k_*$-morphisms 
\[
\Delta:\msA \rightarrow \msA \kstartimes \msA,\qquad \varepsilon: \msA \rightarrow k_*
\] 
satisfying the usual definitions of a bialgebra. We call Hopf $k_*$-algebra any $k_*$-bialgebra admitting an antipode $S: \msA \rightarrow \msA$.

We call \emph{(weak) $\msR$-bialgebra}, resp.~ \emph{(weak) $\msR$-Hopf algebra}, any $\msA$ as above which is a (weak) $\msR$-algebra and for which $\Delta,\varepsilon$ are $\msR$-morphisms, with $\varepsilon$ proper.
\end{Def} 

Note that in the above definition, the antipode is not required to commute with the $*$-structure, nor to be an $\msR$-morphism. 

If $(\msA,\Delta)$ is a type $I$ $\msR$-bialgebra, it follows by the general constructions discussed above that we obtain $C_0(\Theta)$-compatible non-degenerate $*$-homomorphisms 
\[
\Delta_0: A_0 \rightarrow M(A_0\underset{C_0(\Theta)}{\otimes} A_0),\qquad \varepsilon_0: A_0 \rightarrow C_0(\Theta),
\] 
where $\varepsilon_0$ lands in $C_0(\Theta)$ since $\varepsilon$ is proper and surjective. We clearly have the coassociativity condition 
\[
(\Delta_0\otimes \id)\Delta_0 = (\id\otimes \Delta_0)\Delta_0,
\] 
hence $(C_0(\Theta),A_0,\Delta_0)$ can be seen as a field of C$^*$-bialgebras over $\Theta$ in the sense of the following definition.\footnote{We slightly change the terminology of \cite{Bla96} to accord more with the Hopf algebraic terminology. We also diverge from the terminology used in \cite{VV01}.}

\begin{Def}\cite[D\'{e}finition 4.1]{Bla96}
Let $A_0$ be a type $I$ (or nuclear) $C_0(\Theta)$-algebra. We call \emph{$C_0(\Theta)$-bialgebra} structure on $A_0$, or call $A_0$ a \emph{field of C$^*$-bialgebras over $\Theta$}, if $A_0$ is endowed with  a non-degenerate coassociative $C_0(\Theta)$-linear $*$-homomorphism
\[
\Delta_0: A_0 \rightarrow M(A_0\underset{C_0(\Theta)}{\otimes} A_0).
\] 
We call $(A_0,\Delta_0)$ a \emph{Hopf $C_0(\Theta)$-algebra}, or \emph{field of Hopf C$^*$-algebras over $\Theta$}, if the following density conditions are satisfied:
\[
\Delta_0(A_0)(1\otimes A_0)\textrm{ and }(A_0\otimes 1)\Delta_0(A_0)\textrm{ are dense in }A_0\underset{C_0(\Theta)}{\otimes} A_0.
\]
\end{Def} 

Unfortunately, it is not clear if in general the completion $(A_0,\Delta_0)$ of a type $I$ Hopf $\msR$-algebra $(\msA,\Delta)$ will be a field of Hopf C$^*$-algebras. To be able to show this, we need a quantum control compatible with the comultiplication. 

\begin{Prop}\label{PropFieldHopf} 
Let $(\msA,\Delta)$ be a Hopf $\msR$-algebra, and assume that $Y\oplus X$ is a quantum control for $\msA$ with $Y\in \msR$ and $X\in M_N(\msA)$ with  
\[
(\id\otimes \Delta)(X)  = X_{12}X_{13},\qquad (\id\otimes \varepsilon)(X) = I_N.
\] 
Assume moreover that there exist invertible $Q, \widetilde{Q} \in M_N(k_*)$ such that 
\[
\Tr(Q^*X^*XQ) \in \msZ(\msA),\qquad \Tr(\widetilde{Q}XX^*\widetilde{Q}^*) \in \msZ(\msA).
\]
Then $(A_0,\Delta_0)$ is a field of Hopf C$^*$-algebras.
\end{Prop}
\begin{proof}
We are going to show that the condition $\Tr(Q^*X^*XQ) \in \msZ(\msA)$ implies that 
\[
(A_0\otimes 1)\Delta_0(A_0) \textrm{ is dense in } A_{0} \underset{C_0(\Theta)}{\otimes} A_{0}.
\]
Since $X^*$ satisfies the same conditions as $X$ with respect to $\Delta^{\opp}$, the proposition will follow. 

Put $Z = Q^*X^*XQ$ and $C = \Tr(Z) \in \msZ(\msA)$. It is clear that $Y \oplus Z$ is still a quantum control (cf. Lemma \ref{LemLem}), and by Lemma \ref{PropDensr} $Y \oplus C$ is a central control. We may also assume that $Y$ is positive. In the following, we will use the filtration associated to this central control.

For $\pi,\pi' \in P^{\msA}$ with $\pi_{\mid \msR} = \pi_{\mid \msR}'$, let us write $\pi * \pi' = (\pi\otimes \pi')\Delta$, which is again an admissible representation. We claim that for each $M>0$, there exists $M'>0$ such that $\pi\in P_{\leq M}^{\msA}$ and $\pi' \in P_{\geq M'}^{\msA}$ implies $P_{\pi * \pi'}^{\msA}\subseteq P_{\geq M}^{\msA}$.  

To see this, note first that $X$ is invertible with inverse $X^{-1} = (S(X_{ij}))_{ij}$, where $S$ is the antipode of $\msA$. For $M$ fixed, put 
\[
K_M = \sup_{\pi \in P_{\leq M}} \|\pi(X^{-1})\|.
\]
Put $M' = MK_M^2$. Pick $\pi,\pi'\in P^{\msA}$ which agree on $\msR$, with $\pi \in P_{\leq M}$, and let $Q_{\pi} = Q_{\pi'}= \pi(Q) \in M_N(\C)$. Further put $\pi(C) = C_{\pi} \in \R^+$. Assume that $\pi' \in P_{\geq M'}$. Since $\pi$ and $\pi'$ agree on $\msR$, we necessarily have $C_{\pi'} \geq M'$.  On the other hand, since $\|\pi(X^{-1})\|\leq K_M$ we have $K_M^{-2} \leq \pi(X^*X)$. It follows that 
\begin{eqnarray*}
(\pi*\pi')(C) &=& (\Tr(Q_{\pi}^* - Q_{\pi})\otimes \id\otimes \id)(\pi'(X)_{13}^* \pi(X^*X)_{12} \pi'(X)_{13}) \\ & \geq & K_M^{-2} (\Tr(Q_{\pi}^* - Q_{\pi})\otimes \id\otimes \id)(\pi'(X^*X)_{13}) \\ &\geq & M' K_M^{-2} = M.
\end{eqnarray*}
Hence $P_{\pi * \pi'}^{\msA}\subseteq P_{\geq M}^{\msA}$. This proves the claim. 

From the above, it follows immediately that 
\[
(A_{<M}\otimes 1)\Delta_0(A_{<M})\subseteq A_{<M'} \underset{C_0(\Theta)}{\otimes} A_{<M'},
\]
with $M'$ as above in terms of $M$, and hence
\begin{equation}\label{EqIncl}
(A_0\otimes 1)\Delta_0(A_0) \subseteq  A_{0} \underset{C_0(\Theta)}{\otimes} A_{0}.
\end{equation}
To see that the left hand side is actually a dense subspace, note that 
\[
(\msA\otimes 1)\Delta(\msA) = \msA \msRtimes \msA
\] 
by the existence of the antipode. Take $a,b\in \msZ(A_{<M})$. Then from the above observation, we deduce 
\[
(\msA b \otimes 1) \Delta_0(\msA a)= (\msA \msRtimes \msA)(b\otimes 1)\Delta_0(a)
\]
(where we drop the notation for the maps $\msA \rightarrow M(A_c)$ etc.) Multiplying with central local units $e_{N}\otimes e_{N'}$ with $N,N'\geq M'$, we see from Lemma \ref{LemTietDens} that $(A_0\otimes 1)\Delta_0(A_0)$ contains all $(A_{0} \underset{C_0(\Theta)}{\otimes} A_{0})(b\otimes 1)\Delta_0(a)$. Taking now $a,b$ local units converging strictly to 1, we see that the left hand side of \eqref{EqIncl} is dense in the right hand side.
\end{proof}

To end this section, we will study a particular situation allowing us to deduce continuity of an $\msR$-algebra. We first introduce the following notion. 

\begin{Def} 
Let $Y$ be a central control for $\msR$. We call a type $I$ Hopf $\msR$-algebra $(\msU,\Delta)$ of \emph{compact type} if $Y$ is also a central control for $\msU$ and if there exists a $*$-preserving, unit-preserving, positive $k_*$-linear map $\Phi: \msU \rightarrow \msR$ such that 
\begin{equation}\label{EqInv}
(\Phi\otimes \id)\Delta(a) = \Phi(a)1 = (\id\otimes \Phi)\Delta(a),\qquad a\in \msU.
\end{equation}
\end{Def}

Here positivity means that $\Phi(a^*a) \in \msR \subseteq C(\Theta)$ is a positive function for all $a\in \msU$. Note that by $k_*$-linearity, we obtain positive linear maps $\Phi_x: \msU_x \rightarrow \C$ for $x\in \Theta$, where $\Phi_x(a_x) =(\Phi(a))(x)$.

Let $\msU$ be a type $I$ Hopf $\msR$-algebra of compact type. Then for each $x\in \Theta$ we have $U_{x,b} = U_{x,0}$ by the assumption on the central control, and we simply write this C$^*$-algebra as $U_x$ in the following. Moreover, we have a $*$-homomorphism $\msU \rightarrow \msU_x \rightarrow U_x$ with dense image. In particular, we can deduce in this case directly that $\Delta_x(U_x)(1\otimes U_x)$ and $(U_x\otimes 1)\Delta_x(U_x)$ are dense in $U_x\otimes U_x$. It follows that $U_x$ defines a compact quantum group in the sense of Woronowicz \cite{Wor98}. In particular, $(U_x,\Delta_x)$ has an invariant state. Since the existence of the counit prevents $U_x$ from being zero, this invariant state must restrict to $\Phi_x$ on $\msU_x$ by uniqueness of unit-preserving invariant functionals. We will write the invariant state on $U_x$ as $\Phi_x$ as well. Since $f\msU$ is dense in $fU$ for any $f\in C_c(\Theta)$, it follows that for each $a\in U_0$ the map $x \mapsto \Phi_x(a_x)$ lies in $C_0(\Theta)$. We hence obtain a contractive, positive map 
\[
\Phi_0: U_0 \rightarrow C_0(\Theta),\quad a \mapsto (x\mapsto \Phi_x(a_x)).
\]
Note that since $\Phi_x$ is a non-zero positive invariant functional on $\msU_x$, it is in fact faithful \cite[Proposition 3.4]{VDae98}. Hence each $\msU_x$ is C$^*$-faithful, and we may identify $\msU_x \subseteq U_x$. 

We will also need the notion of coaction. 

\begin{Def} Let $(\msA,\Delta)$ be a $k_*$-bialgebra and $\msB$ a unital $k_*$-algebra. We call \emph{coaction} of $\msA$ on $\msB$ any $k_*$-homomorphism 
\[
\alpha: \msB \rightarrow \msB\kstartimes \msA
\]
satisfying the usual coaction identities
\[
(\alpha\otimes \id)\alpha = (\id\otimes \Delta)\alpha,\qquad (\id\otimes \varepsilon)\alpha = \id_{\msB}.
\]

If $(\msA,\Delta)$ is a type $I$ $\msR$-bialgebra and $\msB$ an $\msR$-algebra, we further require that $\alpha$ is an $\msR$-morphism
\[
\alpha: \msB \rightarrow \msB\msRtimes \msA.
\]
\end{Def} 

Such a coaction lifts to a non-degenerate C$^*$-homomorphism
\[
\alpha_0: B_0 \rightarrow M(B_0 \underset{C_0(\Theta)}{\otimes} A_0).
\]

Fix in the following an $\msR$-algebra $\msA$ with an $\msR$-coaction $\alpha: \msA \rightarrow \msA \msRtimes \msU$ by a type $I$ Hopf $\msR$-algebra $\msU$ of compact type.

\begin{Def}
We say that $\msA$ has a \emph{coinvariant central control} if there exists a central control $\oplus C_i$ with $C_i \in \msZ(\msA)$ and
\[
\alpha(C_i) = C_i\otimes 1.
\]
\end{Def}

We may again replace $\oplus C_i$ by $C = \sum C_i^*C_i$. 

In the following, we assume that $\msA$ has a positive coinvariant central control $C$, which we fix. Then $\alpha(e_M) = e_M\otimes 1$ for the associated central local units. 

\begin{Lem} The $\msR$-morphism $\alpha$ is proper.
\end{Lem} 
\begin{proof}
If $Y$ is a central control for $\msR$, it is clear that $C\oplus Y$ is also a central control. Since however $C\otimes 1$ and $1\otimes Y = Y\otimes 1$ determine the C$^*$-filtration on $\msA\msRtimes \msU$, it follows that $C\otimes 1$ is a central control for $\msA\msRtimes \msU$. The properness is now immediate. 
\end{proof}

Let 
\[
\alpha_0: A_0 \rightarrow A_0 \underset{C_0(\Theta)}{\otimes} U_0
\] 
be the completion of $\alpha$. Then we can form the idempotent maps 
\[
E: \msA \rightarrow \msA,\quad a \mapsto (\id\otimes \Phi)\alpha(a),\qquad E_0: A_0 \rightarrow A_0,\quad a\mapsto  (\id\otimes \Phi_0)\alpha_0(a),
\] 
which we call the associated \emph{conditional expectations}. Note that these maps descend to maps $E_x: \msA_x \rightarrow \msA_x$ and $E_{0,x}: A_{0,x}\rightarrow A_{0,x}$. It is also easily seen that the image of $E$ is precisely the set of coinvariant elements $\msA^{\msU}$, defined as 
\[
\msA^{\msU} = \{a \in \msA \mid \alpha(a) = a\otimes 1\}.
\]
In the following, we drop again the notation for the map $\theta: \msA \rightarrow M(A_c)$. 
\begin{Lem}
Let $a\in \msA$, and let $e_M$ be a central local unit. Then $e_M E(a) =E_0(e_Ma)$. 
\end{Lem} 
\begin{proof} 
It is sufficient to check this at each fiber $x\in \Theta$. But since $\alpha_{0,x}(e_{M,x}a_x) = (e_{M,x}\otimes 1)\alpha_x(a) $ lies in the algebraic tensor product of $A_{0,x}$ and $\msU_x$, the result there follows simply by applying $(\id\otimes \Phi_x)$ to both sides.
\end{proof}

Write now $\msB = E(\msA)$, and let $B_{\leq M}$ be the completion of $\msB$ in $A_{\leq M}$. Let $B_b \subseteq A_b$ be the C$^*$-algebraic projective limit of the $B_{\leq M}$, and let $B_0 = B_b \cap A_0$. Note that clearly $e_M \in B_0$ for all $M$. In particular, 
the inclusion $B_0 \rightarrow A_0$ is non-degenerate.

\begin{Lem}
The identity $B_0 = E_0(A_0)$ holds.
\end{Lem}
\begin{proof}
To prove $B_0\subseteq E_0(A_0)$, it is sufficient to prove that $E_0(b) = b$ for $b\in B_0$, since $E_0$ is idempotent. However, choose $\pi \in P_{\leq M}$ and $e_M$ with $\pi_0(e_M) = 1$. Choose a sequence $b_m \in \msB$ with $e_Mb_m \rightarrow e_M b$ in norm inside $A_0$. Then using the previous lemma, we obtain
\[
\pi_0(b) = \lim_m \pi(b_m)  = \lim_m \pi(E(b_m))= \lim_m \pi_0(E_0(e_Mb_m)) = \pi_0(E_0(e_Mb)) = \pi_0(e_ME_0(b)) = \pi_0(E_0(b)).
\]
Hence $b = E_0(b)$.

Conversely, it is sufficient to show that the $E_0(e_M\msA) \subseteq B_0$. This follows from $e_ME(a) = E_0(e_Ma)$ for each $a\in \msA$, and the fact that $E(a) \in \msB$ and $e_M \in B_0$. 
\end{proof}

We will now assume in the following that moreover $\msB$ is central in $\msA$. Then clearly also $B_0$ is central in $A_0$. Consider the following set of characters $\chi$ on $\msB$: 
\[
\Omega = \{\chi \mid \exists \pi \in P^{\msA}, \forall b\in \msB: \pi(b)= \chi(b)\id_{\Hsp_\pi}\}.
\] 
We endow $\Omega$ with the topology of pointwise convergence on $\msB$. Since $B_0$ is the closure of the span of the $e_M\msB$, it follows that for each $\chi \in \Omega$ there exists a unique non-zero character $\chi_0 \in \Spec(B_0)$ such that $\chi_0(f(1+C)^{-1})b) = f((1+\chi(C))^{-1})\chi(b)$ for $b\in \msB$ and $f\in C_c((0,1])$. This leads to a natural embedding
\[
\Omega \rightarrow \Spec(B_0).
\]

\begin{Prop}\label{PropCharComm}
The topological space $\Omega$ is locally compact, and the natural embedding $\Omega \rightarrow \Spec(B_0)$ is a homeomorphism.
\end{Prop}
\begin{proof}
Note first that if $\chi$ is a non-zero character on $B_0$, we can make a Hilbert space $\Hsp_{\chi}$ with non-degenerate representation $\pi_{\chi}$ of $A_0$ by the GNS-construction applied to $\chi \circ E_0$. Since $B_0$ acts by $\chi$ on $\Hsp_{\chi}$, and since $\chi(e_M)\neq 0$ for some $M>0$, the representation $\pi_{\chi}$ factors over some $A_{<M}$ via $A_0 \rightarrow M(A_{<M})$. This leads to a unital $*$-representation of $\msA$ through $\msA \rightarrow M(A_{<M})$. Clearly this is an admissible representation of $\msA$. Then any $\pi \preceq \pi_{\chi}$ restricts to a character $\chi'$ on $\msB$ with $\chi = \chi'_0$. 

The above shows that $\Omega \rightarrow \Spec(B_0)$ is a settheoretic bijection. The remainder of the proof can now be treated similarly as in Example \ref{ExaComm}.
\end{proof}

Note that $B_0$ is a $C_0(\Theta)$-algebra. Hence with $\Omega = \Spec(B_0)$ as above, we obtain a continuous map $\Omega \rightarrow \Theta$. 

We can now state the following particular situation in which a type $I$ $\msR$-algebra is continuous. We will use the notation as above.

\begin{Theorem}\label{TheoFieldCont}
Let $\msA$ be a countably generated $\msR$-algebra with coaction $\alpha$ by a type $I$ Hopf $\msR$-algebra $\msU$ of compact type. Assume that $\msA$ has a coinvariant central control and that $\msB = E(\msA)$ is central in $\msA$. Assume moreover that the map $\Spec(B_0)  \rightarrow \Theta$ is open, and assume that the maps $E_{0,x}:A_{0,x}\rightarrow B_{0,x}$ are faithful. Then $A_0$ is a continuous field of C$^*$-algebras over $\Theta$.
\end{Theorem} 
\begin{proof}
This follows by combining Proposition \ref{PropCharComm} with the upcoming Theorem \ref{TheoBlaS}, noting that $A_0$ will be separable since $\msA$ is countably generated and using Lemma \ref{LemTietDens}.
\end{proof}

\begin{Theorem}\label{TheoBlaS}
Let $A_0$ be a separable, nuclear $C_0(\Theta)$-algebra, and $C_0(\Omega) \subseteq A_0$ a commutative $C_0(\Theta)$-subalgebra. Assume that the ensuing continuous map $\Omega \rightarrow \Theta$ is open, and assume there exists a \emph{conditional expectation}  
\[
E_0: A_0 \rightarrow C_0(\Omega),
\]
i.e. $E_0$ is a contractive, positive, non-degenerate $C_0(\Theta)$-linear map. Assume moreover that the specialisations $E_{0,x}: A_{0,x} \rightarrow C_0(\Omega)_x$ are faithful. Then $A_0$ is a continuous field of C$^*$-algebras over $\Theta$.
\end{Theorem}
\begin{proof}
From \cite[Proposition 3.14]{Bla96}, $C_0(\Omega)$ is a continuous field of C$^*$-algebras over $\Theta$. By \cite[Th\'{e}or\`{e}me 3.3.4) $\Rightarrow$ 1)]{Bla96} we can find a right Hilbert $C_0(\Theta)$-module $\mcF$ with a $*$-homomorphism $C_0(\Omega) \rightarrow \mcL(\mcF)$ such that the localisations $C_0(\Omega)_x \rightarrow \mcL(\mcF_x)$ are faithful. On the other hand, we can complete $A_0$ into a right Hilbert $C_0(\Omega)$-module $\mcE$ by $\langle a,b \rangle_{C_0(\Omega)} = E_0(a^*b)$, and the corresponding localisation is the completion of $A_{0,x}$ with respect to $E_{0,x}$. It is now easy to check that the representation of $A_0$ on $\mcE \underset{C_0(\Omega)}{\otimes} \mcF$ descends to faithful representations of the $A_{0,x}$ on the
\[
(\mcE \underset{C_0(\Omega)}{\otimes} \mcF)_x = \mcE_x \underset{C_0(\Omega)_x}{\otimes} \mcF_x,
\]
and hence $A_0$ is a continuous  field over $\Theta$ by \cite[Th\'{e}or\`{e}me 3.3.1) $\Rightarrow$ 4)]{Bla96}.
\end{proof}

\section{The Hopf $k_*$-algebras $\mcO_\mbq(U(N)),\mcO_\mbq(T(N))$ and $\mcO_\mbq(GL(N,\C))$} 

In this section, we consider our main examples, based on the FRT-formalism \cite{FRT89}. 

\subsection{Definition of  $\mcO_\mbq(U(N)),\mcO_\mbq(T(N))$ and $\mcO_\mbq(GL(N,\C))$}

Let $\mbq$ be a formal variable. We let $k$ be the subalgebra of $\C(\mbq)$ generated by $\mbq,\mbq^{-1}$ and the $[n]_\mbq^{-1}$ for $n\in \N\setminus \{0\}$ where 
\[
[n]_\mbq = \frac{\mbq^{n}-\mbq^{-n}}{\mbq -\mbq^{-1}} =\mbq^{n-1} + \mbq^{n-3}+ \ldots + \mbq^{-n+3} + \mbq^{-n+1}.
\]
In the following we write $\kc = \C(\mbq)$. We define $k_*$ to be $k$ endowed with the $*$-structure $\mbq^* = \mbq$ inherited from $\kc$. 

Consider the following matrix $\mbR \in M_N(\C)\otimes M_N(\C)\otimes k = M_N(k) \msktimes M_N(k)$, 
\[
\mbR =  \mbq^{-1} \sum_i e_{ii}\otimes e_{ii} + \sum_{i\neq j} e_{ii}\otimes e_{jj} + (\mbq^{-1} -\mbq)\sum_{i<j} e_{ij} \otimes e_{ji},
\]
so that $\mbR$ satisfies the $R$-matrix identity 
\[
\mbR_{12}\mbR_{13}\mbR_{23} = \mbR_{23}\mbR_{13}\mbR_{12}.
\]
Note that $\mbR$ is invertible with inverse 
\[
\mbR^{-1} = \mbq \sum_i e_{ii}\otimes e_{ii} + \sum_{i\neq j} e_{ii}\otimes e_{jj} + (\mbq -\mbq^{-1})\sum_{i<j} e_{ij} \otimes e_{ji}.
\]
When interpreting $\mbR\in M_N(k_*) \kstartimes M_N(k_*)$, we have $\mbR^* = \mbR_{21}$.

\begin{Def}[\cite{FRT89}]
We define the FRT $k$-algebra $\mcO_\mbq(M_N(\C))$ to be the universal $k$-algebra generated by the matrix entries of $X = \sum_{ij} e_{ij} \otimes X_{ij}$ with universal relations
\[
\mbR_{12}X_{13}X_{23} = X_{23}X_{13}\mbR_{12}.
\]
We endow $\mcO_\mbq(M_N(\C))$ with the unique $k$-bialgebra structure such that 
\[
(\id\otimes \Delta)X = X_{12}X_{13},\qquad (\id\otimes \varepsilon)X = I_N.
\]
\end{Def} 

\begin{Lem} The $k$-algebra $\mcO_\mbq(M_N(\C))$ is a (non-commutative) noetherian domain. Moreover, $\mcO_\mbq(M_N(\C))$ is free as a $k$-module with basis $\{X_{11}^{k_{11}}X_{12}^{k_{12}}\ldots X_{n-1n}^{k_{n-1n}}\mid k_{ij} \in \N\}$. This result still holds true if the $X_{ij}$ are permuted. 
\end{Lem}
\begin{proof} 
This follows from the fact that $\mcO_\mbq(M_N(\C))$ is an iterated skew polynomial algebra over the noetherian domain $k$, see e.g.~ \cite[Lemma I.1.12, Theorem I.1.13 and Theorem 1.2.7]{BrGo02}. The fact that any order of the $X_{ij}$ provides a basis can be deduced from a similar reasoning as in \cite[Theorem 3.1]{Koe91}.
\end{proof}

Let 
\begin{equation}\label{EqDet1}
\Det_\mbq(X) = \sum_{\sigma \in S_N} (-\mbq)^{l(\sigma)} \underset{\rightarrow}{\prod_{i=1}^N} X_{i,\sigma(i)} \in \mcO_\mbq(M_N(\C)),
\end{equation}
where the order in the product is ascending in the first index. We also have the alternative expression
\begin{equation}\label{EqDet2}
\Det_\mbq(X) = \sum_{\sigma \in S_N} (-\mbq)^{l(\sigma)} \underset{\rightarrow}{\prod_{i=1}^N} X_{\sigma(i),i} \in \mcO_\mbq(M_N(\C)),
\end{equation}
where now the order is ascending in the last index, see e.g.~ \cite[Lemma 4.1.1]{BW91} (where $\mbq \leftrightarrow \mbq^{-1}$). Finally, if we change the order of multiplication, the weight has to be inverted, 
\begin{equation}\label{EqDet3}
\Det_\mbq(X) = \sum_{\sigma \in S_N} (-\mbq)^{-l(\sigma)} \underset{\leftarrow}{\prod_{i=1}^N} X_{i,\sigma(i)}  =  \sum_{\sigma \in S_N} (-\mbq)^{-l(\sigma)} \underset{\leftarrow}{\prod_{i=1}^N} X_{\sigma(i),i}.
\end{equation}

By \cite[Theorem 3]{FRT89} $\Det_\mbq(X)$ is a central element with
\[
\Delta(\Det_\mbq(X)) = \Det_\mbq(X)\otimes \Det_\mbq(X).
\]

We can hence localize at $\Det_\mbq(X)$.

\begin{Def} 
We define the FRT Hopf $k$-algebra $\mcO_{\mbq}(GL(N,\C))$ to be the localisation $\mcO_\mbq(M_N(\C))[\Det_\mbq(X)^{-1}]$. 
\end{Def} 

As $\mcO_\mbq(M_N(\C))$ is a domain, we obtain an embedding
\[
\mcO_\mbq(M_N(\C)) \subseteq \mcO_\mbq(GL(N,\C)).
\]
In fact, $\mcO_\mbq(GL(N,\C))$ is still free over $k$, by an argument similar as \cite[Theorem 3.4]{Koe91}. 

It is also well-known that $X$ becomes invertible in $M_N(\mcO_\mbq(GL(N,\C)))$ \cite[Theorem 4]{FRT89}, and that $\mcO_\mbq(GL(N,\C))$  is a Hopf $k$-algebra by the coproduct induced from $\mcO_\mbq(M_N(\C))$.  

The above results hold as well when specializing $\mbq = q$ for $q\in \C\setminus\{0\}$, and in particular we can define the Hopf algebra $\mcO_q(GL(N,\C))$ over $\C$. On the other hand, we can also extend the ring of scalars to the field $\kc = \C(\mbq)$. We denote the resulting Hopf $\kc$-algebra 
\[
\overline{\mcO}_\mbq(GL(N,\C)) = \kc \msktimes \mcO_\mbq(GL(N,\C)).
\] 

We now `double' $\mcO_\mbq(M_N(\C))$ as to take into account anti-holomorphic functions. This process of realification is well-known, see \cite{Zak92,Pod92,DSWZ92} and \cite[Section 10.2]{K-S97}.

\begin{Def} 
We define the FRT $k$-bialgebra $\mcO_\mbq^{\R}(M_N(\C))$ to be generated by two copies of $\mcO_{\mbq}(M_N(\C))$ with respective generators $X,Y$ having the interchange relation
\[
\mbR_{12}X_{13}Y_{23} = Y_{23}X_{13}\mbR_{12}
\]
and coproduct
\[
(\id\otimes \Delta)(X) = X_{12}X_{13},\qquad (\id\otimes \Delta)(Y) = Y_{12}Y_{13}.
\] 
\end{Def} 

We recall from \cite[Section 10.2.5, Example 16]{K-S97} that $\mcO_\mbq^{\R}(M_N(\C))$ may be viewed as the Drinfeld quantum double of $\mcO_\mbq(M_N(\C))$ with itself through the coquasitriangular $\mbr$-functional uniquely determined by
\[
\mbr: \mcO_\mbq(M_N(\C))\msktimes \mcO_\mbq(M_N(\C)) \rightarrow k,\qquad (\id\otimes \id \otimes \mbr)(X_{13}X_{24})= \mbR.  
\]
In particular, the multiplication map 
\begin{equation}\label{EqMultMN}
\mcO_\mbq(M_N(\C)) \msktimes \mcO_\mbq(M_N(\C)) \rightarrow \mcO_\mbq^{\R}(M_N(\C))
\end{equation}
is an isomorphism of $k$-modules, and $\mcO_\mbq^{\R}(M_N(\C))$ is free as a $k$-module.

\begin{Lem} 
The elements $\Det_\mbq(X)$ and $\Det_\mbq(Y)$ are central.
\end{Lem} 
\begin{proof} 
We claim that 
\begin{equation}\label{EqDetQ}
\mbr(\Det_\mbq(X),X_{ij}) = \mbq^{-1}\delta_{ij},\qquad \mbr^{-1}(\Det_\mbq(X),X_{ij}) = \mbq\delta_{ij}.
\end{equation}
With \eqref{EqDetQ} in hand, the definition of the product in the quantum double gives directly that $\Det_\mbq(X)$ commutes with the $Y_{ij}$, and hence is central. For $\Det_\mbq(Y)$ this follows by symmetry (replacing $\mbR$ by $\mbR_{21}^{-1}$).

Up to a change of conventions, the claim in \eqref{EqDetQ} follows from \cite[Theorem 10.9]{K-S97}. For the convenience of the reader we provide a direct proof.

Write $M_N^{\leq}(k)$ and $M_N^{\geq}(k)$ for respectively upper and lower triangular matrices over $k$. Let $\mcO_\mbq(M_N^{\leq}(\C))$, resp.~ $\mcO_\mbq(M_N^{\geq}(\C))$ be the quotient bialgebra of $\mcO_\mbq(M_N(\C))$ by putting $X_{ij} = 0$ for $i>j$, resp. $i<j$. Write $\pi_{+}$, resp. $\pi_-$ for the resulting quotient map and $T^{\pm}$ for the image of $X$ under $\pi_{\pm}$. Then since $\mbR \in M_N^{\leq}(k) \msktimes M_N^{\geq}(k)$, it follows that $\mbr$ factors through a skew pairing between $\mcO_\mbq(M_N^{\leq}(\C))$ and $\mcO_\mbq(M_N^{\geq}(\C))$, which we continue to write as $\mbr$. Since $\pi_{+}(\Det_\mbq(X)) = T_{11}^+ \ldots T_{NN}^+$, it follows from the skew-pairing property of $\mbr$ and the specific form of $\mbR$ that 
\[
\mbr(\Det_\mbq(X),X_{ij}) = \delta_{i\geq j} \mbr(T_{11}^+\ldots T_{NN}^+,T_{ij}^-) = \delta_{i,j} \prod_{k=1}^N \mbr(T_{kk}^+,T_{ii}^-) = \mbq^{-1}\delta_{ij}.
\]
This shows the first equality in \eqref{EqDetQ}. Since $\Det_\mbq(X)$ is grouplike, also the second formula in \eqref{EqDetQ} follows. 
\end{proof}

\begin{Def}\label{DefMainObj} 
We define the FRT Hopf $k_*$-algebra $\mcO_\mbq^{\R}(GL(N,\C))$ to be the localisation of $\mcO_\mbq^{\R}(M_N(\C))$ with respect to $\Det_\mbq(X)$ and $\Det_\mbq(Y)$, endowed with the $*$-structure
\[
X^* = Y^{-1},\qquad Y^* = X^{-1}.
\] 

We define the FRT $k_*$-algebra $\mcO_\mbq(U(N))$ to be the quotient Hopf $k_*$-algebra of $\mcO_{\mbq}(GL_N^{\R}(\C))$ by the relation $X = Y$. We write in this case the generating matrix $X = U$, and the quotient map
\[
\pi_U: \mcO_\mbq^{\R}(GL(N,\C)) \rightarrow \mcO_\mbq(U(N)).
\]

We define the FRT $k_*$-algebra $\mcO_{\mbq}(T(N))$ to be the quotient Hopf $k_*$-algebra of $\mcO_\mbq^{\R}(GL(N,\C))$ by the relations 
\[
X_{ij}= 0 \textrm{ for }i>j,\quad  Y_{ij}=0 \textrm{ for }i<j, \qquad X_{ii}^* = X_{ii}\textrm{ and }Y_{ii}^* = Y_{ii} \textrm{ for all }i.
\]
We write in this case the generating matrices $X = T = T^+$ and $Y= T^-$, and the quotient map
\[
\pi_T: \mcO_\mbq^{\R}(GL(N,\C)) \rightarrow \mcO_\mbq(T(N)).
\]
\end{Def} 
\begin{Rems}
\begin{enumerate}
\item Note that $\mcO_\mbq(U(N))$ can also be seen as $\mcO_\mbq(GL(N,\C))$ with a specific $*$-structure. 
\item From the definition of $\mcO_\mbq(T(N))$ we have $T_{ii}^+ = (T_{ii}^-)^{-1}$. We will write this element simply as $T_i$. We further obtain from the defining relations the $\mbq$-commutation relations
\[
\left\{\begin{array}{lll} T_i T_{jl} = T_{jl}T_i,&& i\neq j, i\neq l, j<l,\\
T_i T_{il} = \mbq T_{il} T_i,&& i < l,\\
T_iT_{li} = \mbq^{-1} T_{li} T_i, && i > l.\end{array}\right.
\]
\item Centrality of $\Det_\mbq(X)$ and $\Det_\mbq(Y)$ and freeness of $\mcO_\mbq(GL(N,\C))$ as a $k$-module show that $\Det_\mbq(X)$ and $\Det_\mbq(Y)$ are not zero-divisors, so that we still get an embedding 
\[
\mcO_\mbq^{\R}(M_N(\C)) \rightarrow \mcO_\mbq^{\R}(GL(N,\C)).
\]
\item The $k$-module isomorphism \eqref{EqMultMN} can be extended to the level of $GL(N,\C)$, showing in particular that $\mcO_\mbq^{\R}(GL(N,\C))$ is free as a $k$-module. 
\end{enumerate}
\end{Rems} 

\begin{Lem} 
We have $ \Det_{\mbq}(Y) = (\Det_\mbq(X)^{-1})^*$ inside $\mcO_\mbq^{\R}(GL(N,\C))$. 
\end{Lem} 
\begin{proof}
With $S$ the antipode on $\mcO_\mbq(GL_N^{\R})$, we have $Y^* = X^{-1} = (S(X_{ij}))_{ij}$. On the other hand, since $\Det_\mbq(X)$ is grouplike, we have $S(\Det_\mbq(X)) = \Det_\mbq(X)^{-1}$. Since $S$ and $*$ are both anti-multiplicative, it follows that 
\[
(\Det_\mbq(X)^{-1})^* =  \sum_{\sigma \in S_N} (-\mbq)^{l(\sigma)} \prod_{i=1}^N S(X_{i,\sigma(i)})^* =   \sum_{\sigma \in S_N} (-\mbq)^{l(\sigma)} \prod_{i=1}^N Y_{\sigma(i),i} = \Det_\mbq(Y). 
\]
\end{proof}

It follows that $\mcO_\mbq^{\R}(GL(N,\C))$ is generated by $\{X_{ij},X_{kl}^*\}$ and the central elements $\Det_\mbq(X)^{-1},(\Det_\mbq(X)^*)^{-1}$  as a $k_*$-algebra. Moreover, in obvious notation, $\Det_\mbq(U)$ is unitary in $\mcO_\mbq(U(N))$. On the other hand, it follows straight from the definition that
\[
\Det_\mbq(T) = T_{1}\ldots T_{N},
\]
with the $T_{i}$ selfadjoint. 

Let us write $\mcO_\mbq(D(N))$ for the $k$-algebra generated by the $T_{i}^{\pm 1}$, and $\mcO_\mbq(M_N^{<}(\C))$, resp.~ $\mcO_\mbq(M_N^{>}(\C)$ for the $k$-algebra generated by the $T_{ij}^+$ with $i<j$, resp.~ $T_{ij}^-$ with $i>j$. Then also the following PBW-decomposition is standard.

\begin{Lem}\label{LemPBWUq} 
The $k$-module $\mcO_\mbq(T(N))$ is free. More precisely, the $k$-modules  $\mcO_\mbq(D(N))$, $\mcO_\mbq(M_N^{<}(\C))$ and $\mcO_\mbq(M_N^{>}(\C))$ are free $k$-modules, and the multiplication map 
\[
\mcO_\mbq(M_N^{<}(\C)) \msktimes \mcO_\mbq(D(N)) \msktimes \mcO_\mbq(M_N^{>}(\C)) \rightarrow \mcO_\mbq(T(N))
\]
is an isomorphism of $k$-modules. 
\end{Lem}

For $0<q$, the above $k_*$-algebras specialize to Hopf $*$-algebras 
\[
\mcO_q^{\R}(GL(N,\C)),\quad \mcO_q(U(N)),\quad \mcO_q(T(N))
\]
defined by the same relations upon putting $\mbq = q$. For $q= 1$, we obtain the commutative Hopf $*$-algebras consisting of regular complex-valued functions on the respective real algebraic groups, where $T(N)$ consists of all upper triangular matrices in $GL(N,\C)$ with real-valued diagonal entries. We can also view the above Hopf $*$-algebras over $\kc$, in which case we write them as $\overline{\mcO}_\mbq^{\R}(GL(N,\C)),\overline{\mcO}_\mbq(U(N)), \overline{\mcO}_{\mbq}(T(N))$.

\subsection{Quantized enveloping algebra of $\mfgl(N)$ and $\mfu(N)$}

\begin{Def}
We define $U_\mbq(\mfn)$, resp.~ $U_\mbq(\mfn^-)$, to be the universal $\kc$-algebra generated by elements $E_1,\ldots, E_{N-1}$, resp. $F_1,\ldots, F_{N-1}$, such that 
\[
E_{i\pm 1}^2E_i - (\mbq+\mbq^{-1})E_{i\pm 1}E_iE_{i\pm 1} + E_{i}E_{i\pm 1}^2 =0,\quad 
F_{i\pm 1}^2F_i - (\mbq+\mbq^{-1})F_{i\pm 1}F_iF_{i\pm 1} + F_{i}F_{i\pm 1}^2 =0.
\]
We define $U_\mbq(\mfh)$ to be the $\kc$-algebra of Laurent polynomials $\kc[K_1^{\pm 1},\ldots, K_N^{\pm 1}]$, and write $\hat{K}_i = K_iK_{i+1}^{-1}$. 

We define $U_\mbq(\mfgl_N(\C))$ to be generated by $U_\mbq(\mfn),U_\mbq(\mfh),U_\mbq(\mfn^-)$ with interchange relations 
\[
K_iE_j =\mbq^{\delta_{ij}-\delta_{i,j+1}}E_j K_i, \qquad K_iF_j =\mbq^{\delta_{i,j+1}-\delta_{ij}}F_j K_i,\qquad 
E_i F_j - F_jE_i= \delta_{ij} \frac{\hat{K}_i - \hat{K}_i^{-1}}{\mbq-\mbq^{-1}}.
\]
We define $U_\mbq(\mfu(N))$ to be $U_\mbq(\mfgl_N(\C))$ endowed with the $*$-structure
\[
K_i^* = K_i,\qquad E_i^* = \mbq^{-1}F_i \hat{K}_i,\qquad F_i^* = \mbq \hat{K}_i^{-1} E_i.
\]
\end{Def} 
One can endow $U_\mbq(\mfu(N))$ with the Hopf $^*$-algebra structure 
\[
\Delta(K_i) = K_i \otimes K_i,\qquad \Delta(E_i) = E_i \otimes 1 +   \hat{K}_i\otimes E_i,\qquad \Delta(F_i) = F_i \otimes  \hat{K}_i^{-1} + 1 \otimes F_i.
\]
We can divide out by the relation $K_1\ldots K_N= 1$, in which case the resulting Hopf $\kc_*$-algebra is written $U_\mbq^{\ext}(\mfsu(N))$, with quotient map $\pi_\mfs$. On the other hand, we can also consider the Hopf $\kc_*$-subalgebra $U_\mbq(\mfsu(N))$ generated by the $\hat{K}_i,E_i,F_i$, which is defined by universal relations as above. Then $U_\mbq(\mfu(N))$ can be seen as a crossed product of $U_\mbq(\mfsu(N))$ and $\Z$, with extra generator $K_1$. The restriction of $\pi_\mfs$ to $U_\mbq(\mfsu(N))$ is an embedding into $U_\mbq^{\ext}(\mfsu(N))$. When forgetting the $*$-structure, we will write the associated $\kc$-algebras as $U_\mbq(\mfsl_N(\C))$ and $U_\mbq^{\ext}(\mfsl_N(\C))$.

These same definitions can be made  if we replace $\kc_*$ by $\C$ and $\mbq$ by $q>0$ with $q\neq 1$. We obtain a Hopf $*$-algebra $U_q(\mfu(N))$ (over $\C$).

\begin{Lem}\label{LemIsoUT}
There is an isomorphism of Hopf $\overline{k}_*$-algebras 
\[
U_\mbq(\mfu(N)) \rightarrow \overline{\mcO}_\mbq(T(N))^{\cop}
\]
such that 
\[
K_i \mapsto T_{i}^{-1},\qquad E_i \mapsto (\mbq-\mbq^{-1})^{-1} T_{i+1} T_{i+1,i}^- , \qquad F_i \mapsto - (\mbq-\mbq^{-1})^{-1}T_{i,i+1}^+ T_{i+1}^{-1}.\]
\end{Lem} 
\begin{proof}
Up to a different choice of convention, and taking into account the $*$-structure, this is \cite[Theorem 8.33]{K-S97}. 
\end{proof} 

The above result also holds over $\C$ when replacing $\mbq$ by $q>0$ distinct from $1$.

Let us recall some well-known elements from the representation theory of $U_\mbq(\mfgl_N(\C))$.

\begin{Def}\label{DefRepGL}
We call a $U_\mbq(\mfgl_N(\C))$-module $V$ \emph{type $1$} if it has a $\kc$-basis consisting of $\mbq$-integral weight vectors, i.e.~ joint $K_i$-eigenvectors whose eigenvalues are powers of $\mbq$. We call a type $1$ module \emph{highest weight} if there exists a generating $\mbq$-integral weight vector $\xi$ which vanishes under all $E_i$. 
\end{Def}

If $\xi$ is a $\mbq$-integral weight vector,  there must exist $\lambda_i \in \Z$ such that 
\[
K_{i} \xi = \mbq^{\lambda_i}\xi
\]
for all $i$. We call $\mbl = (\lambda_1,\ldots,\lambda_N)$ the \emph{associated integral $K$-weight}, where we write $P = \Z^N$ for the set of all integral $K$-weights. As usual we call a $K$-weight \emph{dominant} if  $\lambda_1\geq \lambda_2\geq \ldots \geq \lambda_N$, and we write $P^+$ for the set of dominant integral $K$-weights.

Similar considerations apply to $U_\mbq(\mfsl_N(\C))$. In this case, we consider $\mbq$-integral weight vectors for the $\hat{K}_i$. The set of integral $\hat{K}$-weights is $\hat{P} = \Z^{N-1}$, with dominant integral $\hat{K}$-weights $\hat{P}^+ = \N^{N-1}$. We write 
\[
P \rightarrow \hat{P},\quad \mbl\mapsto \hat{\mbl},\quad \hat{\mbl}_i = \lambda_i - \lambda_{i+1},
\]
which induces a surjective map $P^+ \rightarrow \hat{P}^+$. 

\begin{Theorem}\cite[Theorem 7.23]{K-S97}\label{TheoRepT}
The abelian category of finite dimensional type $1$  $U_\mbq(\mfgl_N(\C))$-modules (resp.~ $U_\mbq(\mfsl(N,\C))$-modules) is semisimple. Moreover, any irreducible type $1$ module is finite dimensional and highest weight. The highest weight vector is unique up to a scalar in $\kc$, and the weight of the highest weight vector uniquely determines the module up to equivalence. Moreover, any highest weight is dominant, and any dominant integral $K$-weight (resp.~ dominant integral $\hat{K}$-weight) arises in this way.
\end{Theorem} 

For $\mbn \in \hat{P}^+$  we choose a fixed representative $(\overline{V}_{\mbq}(\mbn),\pi_\mbn)$ in the class of irreducible highest weight $U_\mbq(\mfsl_N(\C))$-modules with highest weight $\mbn$. We write the highest weight vector as $\xi_{\mbn}$. For $\mbl \in P^+$, we can extend $\pi_{\hat{\mbl}}$ uniquely to an irreducible highest weight $U_\mbq(\mfgl_N(\C))$-module structure $\pi_{\mbl}$ on $\overline{V}_\mbq(\hat{\mbl})$ with highest weight $\mbl$. Such a representation factors over $U_\mbq^{\ext}(\mfsl(N))$ if and only if $\mbl \in P_0^+$, the set of weights $\mbl$ with $\sum\lambda_i= 0$.

In the case of $U_q(\mfgl_N(\C))$ and $U_q(\mfsl_N(\C))$ for a concrete $q>0$, we slightly change the definition of the type $1$-condition.

\begin{Def}\label{DefRepGL2}
We call a $U_q(\mfgl_N(\C))$-module $V$ \emph{type $1$} if it has a basis over $\C$ consisting of positive weight vectors, i.e.~ joint $K_i$-eigenvectors whose eigenvalues are positive. We call a type $1$ module \emph{highest weight} if there exists a generating positive integral weight vector $\xi$ which vanishes under all $E_i$. 
\end{Def}
 
Let us call $\mbl \in \R^N$ a \emph{weakly integral $K$-weight} if $\lambda_i - \lambda_j \in \Z$ for all $i,j$. We have the following analogue of Theorem \ref{TheoRepT}. If $\xi$ is a positive integral weight vector, let $\mbl = (\lambda_1,\ldots,\lambda_N) \in \R^N$ be the \emph{associated $K$-weight},
\[
K_{i} \xi = \mbq^{\lambda_i}\xi.
\]

\begin{Theorem}\cite[Theorem 7.23]{K-S97}\label{TheoRepTConc}
The abelian category of finite dimensional type $1$  $U_q(\mfgl_N(\C))$-modules (resp.~ $U_q(\mfsl(N,\C))$-modules) is semisimple. Moreover, any irreducible type $1$ module is finite dimensional and highest weight, with weakly integral associated $K$-weights (resp.~ integral associated $\hat{K}$-weights). The highest weight vector is unique up to a scalar in $\C$, and the weight of the highest weight vector uniquely determines the module up to equivalence. Moreover, any highest weight is dominant, and any dominant weakly integral $K$-weight (resp.~ dominant integral $\hat{K}$-weight) arises in this way.
\end{Theorem} 

We will continue to use the same notation as above. We note that in the setting of Theorem \ref{TheoRepTConc}, the $V_q(\hat{\mbl})$ can be endowed with a unique-up-to-scalar Hilbert space structure such that $\pi_{\mbl}$ is a $*$-representation of $U_q(\mfu)$. We will write the set of weakly integral $K$-weights as $\widetilde{P}$, and the dominant weakly integral $K$-weights as $\widetilde{P}^+$.

Let us return now to $U_\mbq(\mfgl_N(\C))$. For $\mbd_1 = (1,0,\ldots,0)$, we write $\kc^N = \overline{V} = \overline{V}(\mbd_1)$ and $\pi_{\overline{V}} = \pi_{\mbd_1}$ for the \emph{vector representation}
\[
\pi_{\overline{V}}(E_i) = e_{i,i+1},\quad \pi_{\overline{V}}(F_i) = e_{i+1,i},\quad K_i e_n = \mbq^{\delta_{i,n}}e_n.  
\]

\begin{Lem}\label{LemNonDegForm}
There is a unique non-degenerate $\kc_*$-valued pairing $(-,-)$ of the Hopf $\kc_*$-algebras $U_\mbq(\mfu(N))$ and $\overline{\mcO}_\mbq(U(N))$ such that 
\[
(\id\otimes (a,-))X = \pi_{\overline{V}}(a)_{ij},\qquad a\in U_\mbq(\mfu(N)).
\]
\end{Lem}

Here a pairing between Hopf $\kc_*$-algebras is in addition assumed to satisfy 
\[
(a^*,b) = (a,S(b)^*)^*.
\]
\begin{proof}
Since $\pi_{\overline{V}}: U_\mbq(\mfu(N)) \rightarrow M_N(k_*)$ is a $*$-preserving homomorphism, the existence of the pairing follows immediately from the fact that $U$ is a unitary corepresentation and $\Sigma \mbR$ an intertwiner of the tensor product representation $(\pi_{\overline{V}} \otimes \pi_{\overline{V}})\circ \Delta$. The non-degeneracy can be proven along the lines of \cite[Corollary 11.23]{K-S97}, which uses the result stated in the next lemma.
\end{proof}

\begin{Lem}\label{LemCosemi} For $\lambda \in P^+$, there exists a unique $\overline{\mcO}_\mbq(GL(N,\C))$-comodule structure $\nabla_{\mbq,\mbl}$ on $\overline{V}_\mbq(\hat{\mbl})$ such that, under the above pairing, we obtain the $U_\mbq(\mfgl(N))$-representation on $\overline{V}_\mbq(\hat{\mbl})$. Moreover, the induced map 
\[
\oplus_{\mbl\in P^+}\End(\overline{V}_\mbq(\hat{\lambda}))^{\circ} \rightarrow \overline{\mcO}_\mbq(GL(N,\C))
\]
is an isomorphism of $\kc$-coalgebras. Here $V^{\circ}$ denotes the $\kc$-linear dual of a $\kc$-vector space $V$.
\end{Lem} 

Lemma \ref{LemNonDegForm} can be lifted to the level of $\mcO_\mbq(T(N))^{\cop}$ and $\mcO_\mbq(U(N))$.

\begin{Lem}\label{LemNonDegk}
There exists a unique non-degenerate  Hopf $k_*$-algebra pairing between $\mcO_\mbq(T(N))^{\cop}$ and
$\mcO_\mbq(U(N))$, uniquely determined by the property 
\begin{equation}\label{EqPairUT}
p:  \mcO_\mbq(T(N))^{\cop} \kstartimes \mcO_\mbq(U(N))\rightarrow k_*,\qquad (\id\otimes p)(T^+_{13}U_{24}) = \mbR,\quad (\id\otimes p)(T^-_{13}U_{24}) = \mbR_{21}^{-1}.
\end{equation}
\end{Lem} 
\begin{proof}
Again, the existence of the pairing is immediate from the defining universal relations and the fact that $T^{\pm}$ and $U$ are corepresentations. As the algebras involved are free over $k$, it is sufficient to check non-degeneracy after extension of scalars from $k$ to $\kc$. However, it is easily checked that $p$ turns into the non-degenerate pairing of $U_\mbq(\mfu(N))$ and $\overline{\mcO}_\mbq(U(N))$ by the isomorphism of Lemma \ref{LemIsoUT}. 
\end{proof}

On the other hand, the same proof of Lemma \ref{LemCosemi} can be  lifted to a stronger statement as follows. 

\begin{Lem}\label{LemCosemik}
There exist (up to isomorphism) unique free $k$-modules $V_\mbq(\hat{\mbl})$ with a right $\mcO_\mbq(GL(N,\C))$-comodule structure such that, upon extension of scalars from $k$ to $\kc$, we obtain the $\overline{\mcO}_\mbq(GL(N,\C))$-comodules $\overline{V}_\mbq(\hat{\mbl})$. Moreover, the induced map 
\begin{equation}\label{EqPW}
\oplus_{\mbl \in P^+}\End(V_\mbq(\hat{\mbl}))^{\circ} \rightarrow \mcO_\mbq(GL(N,\C))
\end{equation}
is an isomorphism of $k$-coalgebras. Here $V^{\circ}$ now denotes the $k$-linear dual of a $k$-module $V$.
\end{Lem} 
We can also endow $V_\mbq(\hat{\mbl})$ with a $U_\mbq(T(N))$-module structure by means of the pairing in Lemma \ref{LemNonDegk}.
\begin{proof}
We can endow $V = k^N$ with its natural right $\mcO_\mbq(GL(N,\C))$-comodule structure by 
\[
\nabla_V(e_i) = \sum_j e_j \otimes X_{ji}.
\]
Let $(k_d,\nabla_d)$ be the one-dimensional comodule $\nabla_d(1) = 1\otimes \Det_\mbq(X)$, and consider for $m\in \Z$ and $n\in \N$ the tensor product comodule $\nabla_{m,n}:= \nabla_d^{\otimes m} \otimes \nabla_V^{\otimes n}$ on $V^{\otimes n}$. Let $H(n)$ be the Hecke algebra of type $A_{n-1}$ over $k$ with generators $\sigma_i$, so 
\[
\sigma_i^2 + (\mbq^{-1} - \mbq) \sigma_i -1 = 0,\quad \sigma_i \sigma_{i+1} \sigma_i = \sigma_{i+1}\sigma_i \sigma_{i+1},\quad \sigma_i\sigma_j = \sigma_j\sigma_i \textrm{ for }|i-j|\geq 2.
\]
Then it is well-known that $H(n)$ is absolutely semisimple over $k$ \cite[Theorem 1.1]{GS97}. We can represent $H(n)$ on $V^{\otimes n}$ by $\sigma_i \mapsto \hat{\mbR}_{i,i+1}$, where $\hat{\mbR} = \Sigma \mbR$ with $\Sigma$ the flip map. Since the $\hat{\mbR}_{i,i+1}$ commute with the comodule structure $\nabla_{m,n}$ on $V^{\otimes n}$, it follows that for each simple idempotent $p\in H(n)$ we obtain a $\mcO_\mbq(GL(N,\C))$-comodule $(V_p,\nabla_p)  \subseteq (V^{\otimes n},\nabla_{m,n})$, with $V_p$ a free $k$-module. By the quantum Brauer-Schur-Weyl duality \cite[Theorem 4.3]{Hay92} we know that $\kc \msktimes V_p$ is a simple $\overline{\mcO}_\mbq(GL(N,\C))$-comodule, and hence of the form $\overline{V}_\mbq(\hat{\mbl})$. If we define $V_\mbq(\hat{\mbl}) := V_p$ and $\nabla_{q,\mbl} = \nabla_p$, then $\nabla_{q,\mbl}$ does not depend on $p$ up to isomorphism, for example since as an $\mcO_\mbq(T(N))$-module the associated space of highest weight vectors (in the obvious sense) is a free one-dimensional $k$-module, completely determining the $\mcO_\mbq(T(N))$-module. Again by the quantum Brauer-Schur-Weyl duality, we obtain in this way $(V_\mbq(\hat{\mbl}),\pi_{\mbl})$ for all $\mbl \in P^+$.  

The map \eqref{EqPW} is obviously surjective since the $X_{ij}$ and $\Det_\mbq(X)^{\pm 1}$ generate $\mcO_\mbq(GL(N,\C))$. It is injective as this is true after extension of scalars to $\kc$ by Lemma \ref{LemCosemi}.
\end{proof}

\begin{Cor}\label{CorExInv}
There exists a unique unit-preserving $k$-linear map
\[
\Phi: \mcO_\mbq(U(N)) \rightarrow k
\]
which is invariant, i.e.~ satisfies \eqref{EqInv}. Moreover, $\Phi$ commutes with $*$ and is positive. 
\end{Cor}
\begin{proof}
Existence follows by taking $\Phi$ to be the component of $\mcO_\mbq(GL(N))$ corresponding to the trivial comodule under \eqref{EqPW}. Uniqueness is clear. The map $\Phi$ must necessarily be $*$-compatible as 
\[
a\mapsto \frac{1}{2}(\Phi(a) + \Phi(a^*)^*)
\]
satisfies the same conditions. Positivity follows since each $\mcO_q(U(N))$ is a Hopf $*$-algebra generated by a unitary corepresentation, and hence has a unique positive invariant functional which must coincide with $\Phi_q$.  
\end{proof}

\subsection{Reflection equation algebra and quantum Cholesky map}

\begin{Def}
We define the RE $k$-algebra $\mcO_{\mbq}^{\ad}(M_N(\C))$ to be the universal $k$-algebra generated by the matrix entries of $Z = \sum_{ij} e_{ij} \otimes Z_{ij}$ with universal relations\footnote{If we were to have used the inverse braiding $(\Sigma \mbR)^{-1} = \Sigma \mbR_{21}^{-1}$, this causes a change $\mbq \leftrightarrow \mbq^{-1}$.}
\[
\mbR_{21}Z_{13}\mbR_{12}Z_{23} = Z_{23}\mbR_{21}Z_{13}\mbR_{12}.
\]
We define the RE $k_*$-algebra $\mcO_{\mbq}(H(N))$ to be $\mcO_{\mbq}^{\ad}(M_N(\C))$ endowed with the $*$-structure $Z^* = Z$. 
\end{Def}

We can view $\mcO_\mbq^{\ad}(M_N(\C))$ as the covariantization of $\mcO_\mbq(M_N(\C))$ with respect to $\mbr$, see \cite{Maj91} and \cite[Example 10.18]{K-S97}; see also \cite[Appendix A]{Mud07b}. In particular, $\mcO_\mbq^{\ad}(M_N(\C))$ is again free over $k$. For $q$ non-zero we have the specialisation $\mcO_q^{\ad}(M_N(\C))$, and for $0<q$ the specialisation $\mcO_q(H(N))$. For $q=1$ these become the algebra of regular functions on $M_N(\C)$ (as a complex affine variety) and the $*$-algebra of regular functions on $H(N)$, the space of hermitian matrices as a real affine variety.

\begin{Lem}
There is a $k_*$-morphism $\chi_X^\mbq: \mcO_\mbq(H(N)) \rightarrow \mcO_\mbq^{\R}(GL(N,\C))$ such that $\chi_X^\mbq(Z) = X^*X$.
\end{Lem} 
\begin{proof}
It follows from an easy computation that $X^*X = Y^{-1}X$ satisfies the defining relation of the reflection equation algebra.
\end{proof}

\begin{Def}
Let $\chi_T^{\mbq}$ be the composition of $\chi_X^{\mbq}$ and $\pi_T$, 
\[
\chi_T^\mbq: \mcO_\mbq(H(N)) \rightarrow \mcO_\mbq(T(N)),\quad Z \mapsto T^*T.
\]
We call $\chi_T^{\mbq}$ the \emph{quantum Cholesky map}.
\end{Def}

Note now that we have available the unitary conjugation of `quantum Hermitian matrices'.

\begin{Lem}\label{LemDefCoad}
There is a $k_*$-coaction 
\[
\Ad_\mbq: \mcO_\mbq(H(N)) \rightarrow \mcO_\mbq(H(N)) \kstartimes \mcO_\mbq(U(N)),\quad Z \mapsto U_{13}^*Z_{12}U_{13}.
\]
\end{Lem} 
\begin{proof}
This follows immediately from the universal relations. 
\end{proof}

By the pairing of $\mcO_\mbq(U(N))$ and $\mcO_\mbq(T(N))^{\cop}$, we obtain a left $\mcO_\mbq(T(N))^{\cop}$-module $k_*$-algebra structure $\rhd$ on $\mcO_\mbq(H(N))$. On the other hand, we can consider on $\mcO_\mbq(T(N))$ the left adjoint $\mcO_\mbq(T(N))^{\cop}$-module $k_*$-algebra structure 
\begin{equation}\label{EqAdjAct}
a \rhd b = a_{(2)}bS^{-1}(a_{(1)}),\qquad a,b\in \mcO_\mbq(T(N)).
\end{equation}

\begin{Lem}\label{LemCholEqui}
The quantum Cholesky map is $\mcO_\mbq(T(N))$-equivariant.
\end{Lem} 
\begin{proof}
Since we are dealing with module $k_*$-algebras, it is sufficient to check equivariance on generators. Now an easy computation using the definition \eqref{EqPairUT} of the pairing shows
\[
T_{13}^{-1} \rhd Z_{23} = \mbR_{12}Z_{23}\mbR_{12}^{-1}.
\]
On the other hand, the definition of the adjoint action gives
\[
T^{-1}\rhd a =T^{-1}(1\otimes a)T,\qquad a \in \mcO_\mbq(T(N)).
\]
The equivariance of the quantum Cholesky map is now equivalent to the identity
\[
\mbR_{12}T_{23}^*T_{23}\mbR_{12}^{-1} =  T_{13}^{-1}T_{23}^*T_{23}T_{13},
\]
which follows straightforwardly from the defining $RTT$-commutation relations of $\mcO_\mbq(T(N))$.  
\end{proof}

\begin{Lem}\label{LemCholInj}
The quantum Cholesky map is injective.
\end{Lem} 
\begin{proof}
This is again a slight variation on known results \cite[Theorem 3]{Bau98},\cite[Proposition 1.7]{K-S09}. First of all, since $\mcO_\mbq^{\ad}(M_N(\C))$ is the covariantization of $\mcO_\mbq(M_N(\C))$ by $\mbr$, we can view $\chi_T^{\mbq}$ as a map on $\mcO_\mbq(M_N(\C))$. If we write $l^+(a)  = \mbr(a,-)$ and $l^-(a)=\mbr(-,S^{-1}(a))$ for the maps of $\mcO_\mbq(GL(N,\C))$ into the co-opposite of its dual, then under the natural non-degenerate pairing of $\mcO_\mbq(T(N))^{\cop}$ and $\mcO_\mbq(GL(N,\C))$ these turn into the Hopf algebra homomorphisms
\[
l^{\pm}: \mcO_{\mbq}(GL(N,\C)) \rightarrow \mcO_\mbq(T(N)),\qquad (\id\otimes l^{\pm})(X)= T^{\pm}.
\]
The covariantized version of $\chi_T^{\mbq}$ is given by restriction to $\mcO_\mbq(M_N(\C))$ of the map 
\[
I:\mcO_\mbq(GL(N,\C)) \rightarrow \mcO_\mbq(T(N)),\qquad a \mapsto l^-(S(a_{(1)}))l^+(a_{(2)}).
\]
We are to show that $I$ is injective. As the above are free $k$-modules, we may extend scalars from $k$ to $\kc$, and by Lemma \ref{LemIsoUT} view $I$ as a map 
\[
I:\overline{\mcO}_\mbq(GL(N,\C)) \rightarrow U_\mbq(\mfu(N)).
\]
Let now $\overline{\mcO}_\mbq(SL_N(\C))$ be the quotient of $\overline{\mcO}_\mbq(GL(N,\C))$ by $\Det_\mbq(X) = 1$. Then $l^{\pm}$ factorize to Hopf algebra homomorphisms $\overline{\mcO}_\mbq(SL_N(\C)) \rightarrow U_\mbq^{\ext}(\mfsl(N))$, leading to the map 
\[
\widetilde{I}: \overline{\mcO}_\mbq(SL_N(\C)) \rightarrow U_\mbq^{\ext}(\mfsl(N)),\qquad a \mapsto l^-(S(a_{(1)}))l^+(a_{(2)})
\]
which is injective by \cite[Corollary 8.4]{JL92}. On the other hand, we also have projections
\[
\overline{\mcO}_\mbq(GL(N,\C)) \rightarrow \kc[D,D^{-1}],\quad U_{ij} \mapsto \delta_{ij} D^2,\qquad U_\mbq(\mfgl(N)) \rightarrow \kc[D,D^{-1}],\quad T^{\pm}_{ij} \mapsto \delta_{ij}D^{\pm 1}.
\]
It is now clear that the following diagram is commutative, where we omit the obvious projection maps:
\[
\xymatrix{  \overline{\mcO}_\mbq(GL(N,\C)) \ar[d]^{I}\ar[r]^{\Delta\hspace{0.8cm}}  & \overline{\mcO}_\mbq(SL_N(\C)) \mskctimes \kc[D,D^{-1}] \ar[d]^{\widetilde{I}\otimes \id}    \\    U_{\mbq}(\mfu(N)) \ar[r]^{\Delta\hspace{0.9cm}} & U_\mbq^{\ext}(\mfsl(N)) \mskctimes \kc[D,D^{-1}].}
\]
Since the top arrows and the right arrow are injective, also the left arrow is injective.
\end{proof}

We will in the following identify $\mcO_\mbq(H(N))$ as a $k_*$-subalgebra of $\mcO_\mbq(T(N))$. By specialisation at $\mbq = q$, we obtain an embedding $\chi_T^q$ of $\mcO_q(H(N))$ into $\mcO_q(T(N))$. 

We will now `complete' the quantum Cholesky map, following the discussion in \cite[Section 1.1 and Section 1.4]{JW17}.\footnote{By our difference of convention, we will work later on with principal blocks  in stead of the right-lower blocks in \cite{JW17}. One can easily change to the conventions of \cite{JW17} upon changing $\mbq \leftrightarrow \mbq^{-1}$ and $e_{i} \leftrightarrow e_{N+1-i}$. Note however that  \cite[Equation (1.9)]{JW17} differs with a factor from our result in Lemma \ref{LemElB}, but it is not clear which identification map $\mcO_q(T(N)) \cong U_q(\mfgl(N))$ is used precisely in \cite[Section 1.4]{JW17}.}

\begin{Def}
We define $B_k \in \mcO_\mbq(H(N))$ to be the element
\[
B_k =  \sum_{\sigma \in S_k} (-\mbq)^{-l(\sigma)} \mbq^{-e(\sigma)} Z_{k,\sigma(k)}\ldots Z_{1,\sigma(1)}, 
\]
where $S_k$ is the symmetric group of $\{1,\ldots,k\}$, $l(\sigma)$ is the length of $\sigma$, and $e(\sigma)$ is the number of $1\leq i\leq k$ with $\sigma(i)<i$. 
\end{Def}

\begin{Lem}\label{LemElB}
Under the quantum Cholesky map, we have $\chi_{T}^{\mbq}(B_k) = T_{1}^2\ldots T_{k}^{2}$. 
\end{Lem}
\begin{proof}
For $K\leq N$, we can embed $\mcO_q(H(K))$ inside $\mcO_q(H(N))$ by considering the principal $K\times K$-block of $Z$. As this inclusion is compatible with the inclusion $U_q(T(K)) \subseteq U_q(T(N))$ sending $T_{ij}^{\pm}$ to $T_{ij}^{\pm}$, it is sufficient to prove the lemma for $k = N$. However, by \cite[Corollary 6.4]{JW17}, under the vector isomorphism $\mcO_\mbq^{\ad}(M_N(\C)) \cong \mcO_\mbq(M_N(\C))$, we have that $B_N$ corresponds to $\Det_{\mbq}(X)$, using \eqref{EqDet3}. But it is easily seen that in this case the map $I$ from Lemma \ref{LemCholInj} sends $\Det_\mbq(X)$ to $B_N$, using \eqref{EqDetQ}.
\end{proof}

We easily see that 
\[
\left\{\begin{array}{lll} B_{i} Z_{kl} = Z_{kl}B_{i},&& i \notin [\min\{k,l\},\max\{k,l\}),\\
B_{i} Z_{kl} = q^2 Z_{kl} B_{i},&& k\leq i < l ,\\
B_{i}Z_{kl} = q^{-2} Z_{kl} B_{i}, && l\leq i < k.\end{array}\right.
\]
We can hence form the localisation $\mcO_\mbq(H(N))[B_1^{-1},\ldots, B_N^{-1}]$. We can add to this algebra unique self-adjoint invertible elements $A_i$ such that 
\[
\left\{\begin{array}{lll} A_{i} Z_{kl} = Z_{kl}A_{i},&& i \notin [\min\{k,l\},\max\{k,l\}),\\
A_{i} Z_{kl} = q Z_{kl} A_{i},&& k\leq i < l ,\\
A_{i}Z_{kl} = q^{-1} Z_{kl} A_{i}, && l\leq i < k\end{array}\right.
\]
and $A_i^2 = B_i$.

\begin{Prop}\cite[Corollary 1.11]{JW17}\label{PropIdHT} 
For $0<q$ and $q\neq 1$, the natural extension of $\chi_T^q$ by 
\[
\chi_T^q:\mcO_q(H(N))[A_1^{-1},\ldots, A_N^{-1}]\rightarrow \mcO_q(T(N)),\quad A_i \mapsto T_1\ldots T_i
\]
is an isomorphism of $\C$-algebras. 
\end{Prop}
We refrain from formulating a corresponding statement for $q= 1$ since it will not be needed later on, and requires a different proof.
\begin{proof}
For $q\neq 1$, we may identify $\mcO_q(T(N))$ with $U_q(\mfu)$. Applying $E_i\rhd$ or $F_i\rhd$ to $B_i$, we find that the image of $\mcO_q(H(N))[B_1^{-1},\ldots,B_N^{-1}]$ contains all $K_i^{\pm 2}$, $E_i$ and $F_i\hat{K}_i$ and is generated by these elements \cite[Theorem 6.4]{JL92}. Since the latter algebra clearly has  a PBW-decomposition with respect to these generators, there is up to isomorphism a unique way to add self-adjoint invertible roots of the $K_i^2$, the resulting extension being $U_q(\mfu)$. This shows that $\chi_T^q$ is both injective and surjective on $\mcO_q(H(N))[A_1^{-1},\ldots, A_N^{-1}]$.
\end{proof}

Consider now  the unit-preserving invariant $k_*$-linear functional $\Phi: \mcO_\mbq(U(N)) \rightarrow k_*$  from Corollary \ref{CorExInv}, and define $E$ as the associated conditional expectation
\[
E: \mcO_\mbq(H(N)) \rightarrow \mcO_\mbq(H(N)),\quad a\mapsto (\id\otimes \Phi)\Ad_\mbq(a).
\]

\begin{Lem}
We have $E(\mcO_\mbq(H(N))) = \msZ(\mcO_\mbq(H(N)))$, the center of $\mcO_\mbq(H(N))$. 
\end{Lem}
\begin{proof}
Let $a$ be in the range of $E$. Then $a$ is $\Ad_\mbq$-coinvariant. It follows by Lemma \ref{LemCholEqui} that for each $0<q<1$ we have $\chi_T^q(a_q) \in \mcO_q(T(N))$  invariant for the adjoint action \eqref{EqAdjAct}. It is however well-known and easy to see that this implies $\chi_T^q(a_q)$ is central in $\mcO_q(T(N))$, and a fortiori $a_q$ central in $\mcO_q(H(N))$. As $\mcO_\mbq(H(N))$ is free over $k$, this implies that $a$ itself lies in the center.

Conversely, if $a$ lies in the center of $\mcO_\mbq(H(N))$, each $a_q$ lies in the center of $\mcO_q(H(N))$ by Proposition \ref{PropIdHT}. Since this implies that $a_q$ is invariant with respect to the adjoint action for $q\neq 1$, we obtain $E_q(a_q) = a_q$ for all $q\neq 1$. By freeness, this means $E(a) = a$.
\end{proof}

We can also introduce a Harish-Chandra homomorphism in the $k_*$-setting, cf. \cite[Section 8]{JL92}. We will use the PBW-decomposition from Lemma \ref{LemPBWUq}.

\begin{Def}
Let 
\[
P: \mcO_\mbq(T(N)) \cong \mcO_\mbq(M_N^{<}(\C)) \msktimes \mcO_\mbq(D(N))\msktimes \mcO_\mbq(M_N^{>}(\C)) \rightarrow \mcO_\mbq(D_N),\quad x\otimes y \otimes z \mapsto \varepsilon(x)y\varepsilon(z).
\]
We call the map
\[
\chi_{HC}: \msZ(\mcO_\mbq^{\ad}(M_N(\C))) \rightarrow \mcO_\mbq(D(N)),\quad c \mapsto P(\chi_T^{\mbq}(c))
\]
the \emph{Harish-Chandra map}.
\end{Def}

\begin{Lem}
The map $\chi_{HC}$ is a homomorphism with range in $k[T_1^{\pm 2},\ldots,T_N^{\pm 2}]$.
\end{Lem} 
\begin{proof}
The homomorphism property of $\chi_{HC}$ is clear, as this holds already for the relative commutant of $\mcO_\mbq(D(N))$. Since for each $q>0$ and $q\neq 1$ the range lies in $\C[T_1^{\pm 2},\ldots,T_N^{\pm 2}]$ by the proof of Proposition \ref{PropIdHT}, the same holds over $k$ by freeness. 
\end{proof}

We want to compute the composition 
\[
E_{HC}: k[B_1,\ldots,B_N] \rightarrow k[T_1^{\pm2},\ldots,T_N^{\pm 2}],\quad x \mapsto \chi_{HC}(E(x)).
\]

For this, we simply repeat the argument in \cite[Theorem 8.6]{JL92}, paying a little more detail to coefficients. For $\mbn \in \N^{N}$, let us write 
\[\check{n}_i = \sum_{j=i}^{N} n_j,
\]
so $\check{\mbn} \in P^+$. Further write for $\mbl \in \widetilde{P}^+$ and $\nu \in \widetilde{P}$
\[
d_{\mbl,\nu} = \dim(V(\mbl)_{\nu})
\]
for the $\nu$-weight space of the $\mfgl(N)$-module $V(\mbl)$ with highest weight $\mbl$. Then $d_{\mbl,\nu}$ is also the dimension of the $K$-weight space at $\nu$ of $V_q(\mbl)$. Finally, write 
\[
d_{\mbl}(\mbq) = \sum_{\nu \in \widetilde{P}}  \mbq^{-2\sum_{i=1}^N (N-i+1)\nu_i} d_{\mbl,\nu},
\]
which we call the \emph{quantum dimension} of $V_q(\mbl)$. 

\begin{Prop}\label{PropHC}
Let $\mbl \in \widetilde{P}^+$ be a positive weakly integral and dominant $K$-weight. Then 
\[
E_{HC}(B_1^{n_1}\ldots B_N^{n_N}) = \frac{\mbq^{-2 \sum_i \check{n}_i}}{d_{\check{\mbn}}(\mbq)} \sum_{\nu \in \Z^N} d_{\check{\mbn},\nu} \prod_{i=1}^{N} (\mbq^{-(N-i)} T_i)^{2\nu_i}.
\]
\end{Prop}
Note that it is not clear a priori that $d_{\check{\mbn}}(\mbq)^{-1}$ lies in $k$, but this will be a consequence of the proof.
\begin{proof}
By freeness over $k$, it is sufficient to verify this formula for fixed $q>0$ with $q\neq 1$. 

Let $U_{\pi}$ be the unitary corepresentation of $\mcO_q(U(N))$ associated to $\pi = \pi_{\mbl}$, so if $e_i$ is an orthonormal basis of $V_q(\mbl)$ then $\nabla_{q,\mbl}(e_i) = \sum_j e_j\otimes U_{\pi,ji}$. Consider $\End(V_q(\mbl))$ with the right $\mcO_q(U(N))$-coaction
\[
\Ad_{\pi}: \End(V_q(\mbl))\rightarrow \End(V_q(\mbl))\otimes \mcO_q(U(N)),\quad x\mapsto U_\pi(x\otimes 1)U_\pi^*.
\]
Then the $\mcO_q(T(N))$-module dual to $\Ad_{\pi}$ is easily seen to be given by 
\[
a \rhd x = \pi(a_{(2)})x\pi(S^{-1}(a_{(1)})).
\]
It follows that 
\[
(\pi\otimes \id)(\Ad_q(a)) = \Ad_\pi(\pi(a)),\qquad a\in \mcO_q(T(N)).
\]
As $x\in \End(V_q(\mbl))$ is invariant under $\rhd$ if and only if $x$ commutes with the $\pi(a)$ for $a\in \mcO_q(T(N))$, it follows from irreducibility of $V_q(\mbl)$ that we obtain a functional
\[
\omega_\pi: \End(V_q(\mbl))\rightarrow \C,\quad (\id\otimes \Phi_q)\Ad_\pi(x) = \omega_\pi(x)\id.
\]
Since  
\[
(\omega_\pi\otimes \id)\Ad_\pi(x) = \omega_{\pi}(x)1,
\]
it follows that $\omega_\pi$ is invariant under the right $\mcO_q(T(N))$-module structure on $\End(V_q(\mbl))^*$ dual to $\rhd$. But as the space of $\rhd$-invariants is one-dimensional, the same holds for the dual module structure, hence $\omega_{\pi}$ is the unique $\mcO_q(T(N))$-invariant functional satisfying $\omega_\pi(1) = 1$. But it is easily seen from the isomorphism $\mcO_q(T(N)) \cong U_q(\mfu(N))$ that the antipode $S^2$ of $\mcO_q(T(N))$ is implemented by $B =B_1\ldots B_N$, 
\[
S^2(a) = B^{-1}aB,\qquad a\in \mcO_q(T(N)).
\]
Hence $\Tr(\pi(B)-)$ is $\rhd$-invariant, and $\frac{\Tr(\pi(B)-)}{\Tr(\pi(B))}$ is the unique $\rhd$-invariant unit-preserving functional.

Consider now $B_{\mbn} = B_1^{n_1}\ldots B_N^{n_N}$. For $z\in \msZ(\mcO_q(T(N))$, let $z(\mbl)$ be the value in $V_q(\mbl)$. 
Then it follows from the above that
\[
E(B_{\mbn})(\mbl) = \omega_{\pi}(E(B_{\mbn})) = \omega_{\pi}(B_{\mbn}).
\]
However, from the Weyl character formula 
\[
\sum_{\nu \in \widetilde{P}} e^{\langle \xi,\nu\rangle} d_{\mbl,\nu} = \frac{\det(e^{\xi_i(\lambda_j + N-j)})}{\det(e^{\xi_i(N-j)})} 
\]
it follows straightforwardly by symmetry, rearranging factors and bringing out appropriate powers, that
\[
 \omega_{\pi}(B_{\mbn}) = \frac{\mbq^{-2\sum_i \check{n}_i}}{d_{\check{\mbn}}(\mbq)} \sum_{\nu \in \Z^N} d_{\check{\mbn},\nu} \prod_{i=1}^{N} (\mbq^{-(\lambda_i + N-i)})^{2\nu_i}.
\]
Since also 
\[
\chi_{HC}(E(B_{\mbn})) \xi_{\hat{\mbl}} = E(B_{\mbn})\xi_{\hat{\mbl}} = E(B_{\mbn})(\mbl)\xi_{\hat{\mbl}},
\]
the proposition follows. 
\end{proof}
\begin{Cor}\label{CorHC}
For $1\leq k\leq N$ 
\begin{equation}\label{EqHCB}
E_{HC}(B_k) = \mbq^{-(N+1)k}\binom{N}{k}_{\mbq}^{-1} e_k(\mbq^2 T_{1}^2,\mbq^4T_2^2,\ldots,\mbq^{2N}T_N^2),
\end{equation}
where $e_k$ is  the $k$th elementary symmetric polynomial $e_k(x_1,\ldots,x_N) = \sum_{1\leq i_1<\ldots<i_k\leq N}x_{i_1}\ldots x_{i_k}$.
\end{Cor}

\subsection{Quantum $QR$-decomposition}

The following is a quantum analogue of the QR-decomposition (or $KAN$-decomposition) of $GL(N,\C)$. 

\begin{Lem}\label{LemMorQR}
The $k_*$-morphism 
\[
\pi_{UT}^\mbq: \mcO_\mbq^{\R}(GL(N,\C)) \rightarrow \mcO_\mbq(U(N))\kstartimes \mcO_\mbq(T(N)),\quad x\mapsto (\pi_U\otimes \pi_T)\Delta(x)
\]
is injective.
\end{Lem} 

We will call $\pi_{UT}^{\mbq}$ the \emph{quantum $QR$-decomposition}. 

\begin{proof}
Let us introduce the Hopf $k$-algebras $\mcO_\mbq(GL^{\leq}(N,\C))$ and $\mcO_\mbq(GL^{\geq}(N,\C))$ obtained by dividing out $\mcO_\mbq(GL(N,\C))$ with respect to all $X_{ij} = 0$ for $i>j$, resp.~ $X_{ij}= 0$ for all $i<j$, and write the respective generators as $T_{ij}^+$ and $T_{ij}^-$. Let $\mcO_\mbq(T'(N))$ be defined as $\mcO_\mbq(T(N))$, without however asking that the $T_{ii}^+$ or $T_{ii}^-$ are selfadjoint. Let $\pi_{\pm}$ be the respective quotient maps from $\mcO_\mbq(GL(N,\C))$ to $\mcO_\mbq(GL^{\leq}{N,\C})$ and $\mcO_\mbq(GL^{\geq}{N,\C})$. Then the following two facts are easily verified:
\begin{itemize}
\item The multiplication map 
\begin{equation}\label{EqMult1}
\mcO_\mbq(GL^{\leq}(N,\C)) \msktimes \mcO_\mbq(GL^{\geq}(N,\C)) \rightarrow \mcO_\mbq(T'(N))
\end{equation}
 is an isomorphism.
\item The homomorphism 
\begin{equation}\label{EqMult2}
(\pi_+\otimes \pi_-)\Delta: \mcO_\mbq(GL(N,\C)) \rightarrow  \mcO_\mbq(GL^{\leq}(N,\C)) \msktimes \mcO_\mbq(GL^{\geq}(N,\C))
\end{equation}
is injective. 
\end{itemize}

We now claim that the map 
\[
\mcO_\mbq(M_N(\C))\msktimes \mcO_\mbq(M_N(\C)) \rightarrow \mcO_\mbq(M_N(\C))\msktimes \mcO_\mbq(GL^{\leq}(N,\C)) \msktimes \mcO_\mbq(GL^{\geq}(N,\C)),
\]
\[
a\otimes b \mapsto a_{(1)}b_{(1)}\otimes \pi_+(a_{(2)})\otimes \pi_-(b_{(2)}) 
\]
is injective. Indeed, if this were not injective, by expanding the first leg, applying the antipode to the second leg and multiplying with the third leg, we obtain that also 
\[
a\otimes b \mapsto ab_{(1)}\otimes \pi_+(b_{(2)})\otimes \pi_-(b_{(3)})
\]
is not injective. By \eqref{EqMult2}, this implies the Galois map 
\[
a \otimes b \mapsto ab_{(1)}\otimes b_{(2)}
\]
on $\mcO_\mbq(M_N(\C))$ is not injective, which is a contradiction. 

From \eqref{EqMultMN} and \eqref{EqMult1}, it now follows that, with 
\[
\pi_T': \mcO_\mbq^{\R}(M_N(\C)) \rightarrow \mcO_\mbq(T'(N)),\qquad X \mapsto T^+,\quad Y\mapsto T^-,
\]
the $*$-homomorphism
\[
\mcO_\mbq^{\R}(GL(N,\C)) \rightarrow \mcO_\mbq(U(N)) \kstartimes \mcO_\mbq(T'(N)),\quad x \mapsto (\pi_U\otimes \pi_T')\Delta(x) 
\]
is injective (where the $*$-structure on $\mcO_\mbq(T'(N))$ is given by $(T^-)^* = (T^+)^{-1}$). 

Finally, let us write $k_*[\Z^N] = k_*[L_i^{\pm 1}]$ for the Laurent polynomial algebra over $k_*$ with $*$-structure $L_i^* = L_i^{-1}$. Then we have $*$-homomorphisms 
\[
\pi_D^{T'}:\mcO_\mbq(T'(N)) \rightarrow k_*[\Z^N],\qquad T_{ij}^+ \mapsto \delta_{ij} L_i,\quad T_{ij}^- \mapsto \delta_{ij} L_i,
\]
\[
\pi_D^U:\mcO_\mbq(U(N)) \rightarrow k_*[\Z^N],\qquad U_{ij} \mapsto \delta_{ij} L_i,
\]
and put $\pi_D = \pi_D^{T'} \circ \pi_T'$. Let $\pi_T^{T'}$ be the natural projection map $\mcO_\mbq(T'(N)) \rightarrow  \mcO_\mbq(T(N))$. Then it is easily seen that the $*$-homomorphism
\[
\mcO_\mbq(T'(N)) \rightarrow k_*[\Z^N]\kstartimes \mcO_\mbq(T(N)),\quad x\mapsto (\pi_D^{T'}\otimes \pi_T^{T'})\Delta(x)
\]
is injective, and so the $*$-homomorphism
\[
\mcO_\mbq^{\R}(GL(N,\C)) \rightarrow \mcO_\mbq(U(N)) \kstartimes k_*[\Z^N]\kstartimes \mcO_\mbq(T(N)),\quad x \mapsto (\pi_U\otimes \pi_D\otimes \pi_T)((\id\otimes \Delta)\Delta(x).
\]
is injective. Clearly $(\pi_U\otimes \pi_D^U)\Delta$ is injective on $\mcO_\mbq(U(N))$, so by coassociativity we finally obtain injectivity of
\[
\mcO_\mbq^{\R}(GL(N,\C)) \rightarrow \mcO_\mbq(U(N)) \kstartimes \mcO_\mbq(T(N)),\quad x \mapsto (\pi_U\otimes \pi_T)\Delta(x). 
\]
\end{proof}

Note that both $\mcO_\mbq^{\R}(GL(N,\C))$ and $\mcO_\mbq(U(N))\kstartimes \mcO_\mbq(T(N))$ can be endowed with left coactions $\gamma_U$ by $\mcO_\mbq(U(N))$, resp.~ by 
\begin{equation}\label{EqLeftTrans}
(\id\otimes \gamma_U)X = U_{12}X_{13},\qquad \gamma_U(u\otimes x) = \Delta(u)\otimes x,
\end{equation}
and the map $\pi_{UT}^{\mbq}$ is $\gamma_U$-equivariant.

Let us write 
\[
E = (\Phi\otimes \id)\gamma_U,
\]
where $\Phi$ is the unique unit-preserving invariant functional on $\mcO_\mbq(U(N))$, see Corollary \ref{CorExInv}. Recall that the image of $E$ is precisely the $k_*$-algebra of $\gamma_U$-coinvariant elements. Since the $k_*$-algebra of coinvariants in $\mcO_\mbq(U(N))\kstartimes \mcO_\mbq(T(N))$ may be identified with $\mcO_\mbq(T(N))$, we hence see that $\chi_T^{\mbq}$ restricts to an embedding of 
\[
\msB = E(\mcO_\mbq^{\R}(GL(N,\C))) = \mcO_\mbq^{\R}(GL(N,\C))^{\mcO_\mbq(U(N))}
\]
into $\mcO_\mbq(T(N))$. This also applies for $\mbq$ specialized at $q>0$. 

\begin{Lem}\label{LemMSB}
For $0<q$ and $q\neq 1$, we have $\mcO_q(H(N))[B_N^{-1}] \subseteq \pi_{UT}^{\mbq}(\msB)\subseteq \mcO_q(H(N))[B_1^{-1},\ldots, B_N^{-1}]$. 
\end{Lem}
Again, the result for $q= 1$ will not be needed.
\begin{proof}
Clearly the entries of $X^*X$ are $\gamma_U$-coinvariant, hence $\mcO_\mbq(H(N))[B_N^{-1}] \subseteq \pi_{UT}^{\mbq}(\msB)$. On the other hand, let $C_2^N$ be the $N$-fold direct product of the cyclic group of order 2, and let $k_*[C_2^{N}]$ be the group $k_*$-algebra of $C_2^{N}$ with generators $\epsilon_i$. Consider on $\mcO_\mbq(U(N)) \kstartimes \mcO_\mbq(T(N))$ the $k_*[C_2^N]$-coaction 
\[
u_{ij}\otimes T^{\pm}_{kl} \mapsto u_{ij}\otimes T^{\pm}_{kl}\otimes \epsilon_j\epsilon_k. 
\]
Then $\pi_{UT}^{\mbq}(\mcO_\mbq^{\R}(GL(N,\C))$ and hence $\pi_{UT}^{\mbq}(\msB)$ lies in the $k_*$-algebra of $k_*[C_2^{N}]$-coinvariants. But $\mcO_q(T(N)) \cong U_q(\mfu(N))$ has generators 
\[
T_i^{-1}\cong K_i,\quad T_{i+1}T_{i+1,i}^-  \cong (q-q^{-1})E_i,\quad T_{i,i+1}^+ T_{i}^{-1} \cong (q^{-1}-q)F_i\hat{K}_i, 
\]
and by a PBW-argument we see that the $k_*[C_2^N]$-coinvariant elements of $\mcO_\mbq(T(N))$ must lie in the $k_*$-algebra generated by the $K_i^{\pm 2}, E_i, F_i\hat{K}_i$. By the proof of Proposition \ref{PropIdHT}, the latter $k_*$-algebra is precisely $\mcO_\mbq(H(N))[B_1^{-1},\ldots, B_N^{-1}]$.
\end{proof}

\subsection{Quantum Cayley-Hamilton identities}

Put
\[
\mbQ = \Diag(\mbq^{N-i})_{1\leq i \leq N} \in M_N(k_*),
\]
and write 
\[
\Tr_{\mbQ^2}: M_N(\msA) \rightarrow \msA,\quad W \mapsto (\Tr \otimes \id)(\mbQ^2 W) = (\Tr\otimes \id)(W\mbQ^2) = (\Tr\otimes \id)(\mbQ W \mbQ).
\]

From \cite{P-S95} we have the following theorem, see also \cite{JW17}.

\begin{Theorem}\label{TheoQCHRE}
The element $Z$ of $\mcO_\mbq(H(N))$ satisfies a CH-identity 
\[
F(Z)= Z^N - C_1Z^{N-1} + \ldots + (-1)^{N-1}C_{N-1}Z + (-1)^{N}C_N =  0,
\]
with all coefficients $C_i$ in the $k$-algebra generated by the $\Tr_{\mbQ^2}(Z^k)$, and $C_1 = \Tr_{\mbQ^2}(Z)$. 
\end{Theorem}
\begin{Rem}
The first part of the theorem holds in fact over $\C[\mbq,\mbq^{-1}]$ by the concrete formulas for the $C_i$ in \cite[Theorem 1.3]{JW17}. However,  the $[n]_\mbq^{-1}$ are needed in $k$ for the latter statement of the theorem to hold, since we need semisimplicity of the Hecke algebra (cf. the proof of Lemma \ref{LemCosemik}).
\end{Rem}

We again write $C = C_1$, and call it the leading Casimir. For $C_N$ we have the following identity from \cite[Theorem 1.3]{JW17}:
\begin{equation}\label{EqCNB}
C_N = \mbq^{N(N-1)}B_N.
\end{equation}

\begin{Rem}\label{RemMud}
Using the techniques of \cite[Proposition 4.1]{Mud07a}, we can deduce that the eigenvalues of $\pi_{\mbl}(T^*T)$ are $\mbq^{-2(\lambda_k-k-1)}$, hence $P(C_k) = \mbq^{2k} e_k(\mbq^2T_1^2,\mbq^4T_2^2,\ldots,\mbq^{2N}T_N^2)$ and
\[
C_k = \mbq^{(N-1)k} \binom{N}{k}_\mbq E(B_k).
\]
\end{Rem} 

We will also need the following expression for $C_N$ or, equivalently, $B_N$, which is \cite[Equation (2.7)]{P-S95} adapted to our conventions. We formulate the result for the case $\mbq = q$ with $0<q$. Recall that $\hat{R} = \Sigma R$, with $R$ the flip map.

\begin{Lem} \label{LemIdBN}
Up to a positive non-zero scalar, we have 
\[
B_N \sim (\omega_q\otimes \id)((Z_{N,N+1}\hat{R}_{N-1,N}\ldots\hat{R}_{23} \hat{R}_{12})^N),
\]
where $\omega_q$ is the vector state on $V^{\otimes N}$ associated to the normalized joint $-q$-eigenvector of the $\hat{R}_{i,i+1}$ for $1\leq i\leq N-1$. 
\end{Lem}

We now transport Theorem \ref{TheoQCHRE} to $\mcO_\mbq^{\R}(GL(N,\C))$ and $\mcO_\mbq(T(N))$. 

\begin{Lem}\label{LemNewt}
The elements $\chi_X^\mbq(\Tr_{\mbQ^2}(Z^k)) = \Tr_{\mbQ^2}((X^*X)^k)$ and $\Tr_{\mbQ^{-2}}((XX^*)^k)$  are central in $\mcO_\mbq^{\R}(GL(N,\C))$. 
\end{Lem} 

\begin{proof}
The proof technique for this is standard. Namely, note first that with $V = k^N$ we have the $*$-preserving \emph{vector representation}
\[
\pi_V: \mcO_\mbq(T(N)) \rightarrow  \End(V) = M_N(k_*),\qquad T^+ \mapsto \mbR,\quad T^- \mapsto \mbR_{21}^{-1}. 
\]
We denote the extension to $\mcO_\mbq^{\R}(GL(N,\C))$ as $\wpi_V = \pi_V \circ \pi_T$. Since the defining cross commutation between $X$ and $Y$ can be written as
\[
Y_{23}(\id\otimes (\wpi_V\otimes \id)\Delta^{\opp})(X)Y_{23}^{-1} = (\id\otimes (\wpi_V\otimes \id)\Delta)(X),
\]
and similarly with $X$ replaced by $Y$, it follows that 
\[
Y(\wpi_V\otimes \id)\Delta^{\opp}(a)Y^{-1} = (\wpi_V\otimes \id)\Delta(a),\qquad a\in \mcO_\mbq^{\R}(GL(N,\C)).
\]
Applying $*$, we see that 
\begin{equation}\label{EqCommDelt}
X^*X(\wpi_V\otimes \id)\Delta^{\opp}(a)= (\wpi_V\otimes \id)\Delta^{\opp}(a)X^*X,\qquad a\in \mcO_\mbq^{\R}(GL(N,\C)).
\end{equation}
Note now that the antipode $S$ is invertible by the general relation $S(a)^* = S^{-1}(a^*)$ in a Hopf $*$-algebra. Then since $Y^* = (S(X_{ij}))_{ij}$ and, with $\transp$ the transpose, $((Y^{\transp})^{-1})_{ij}= S^{-1}(Y_{ji})$, we find
\[
\wpi_V(S^2(X_{ij})) = \wpi_V(S(Y_{ji}^*)) = \wpi_V(S^{-1}(Y_{ji}))^* =  \wpi_V(((Y^{\transp})^{-1})_{ij})^*,
\]
and hence 
\[
(\id\otimes \wpi_V S^2)X = (((\mbR_{21}^{-1})^{\transp\otimes \id})^{-1})^{(\tr\otimes \id)*} =  (((\mbR^{-1})^{\transp\otimes \id})^{-1})^{\tr\otimes \id}.
\]
It is now easily calculated that 
\[
 (((\mbR^{-1})^{\transp\otimes \id})^{-1})^{\tr\otimes \id}= \mbq^{-1} \sum_i e_{ii}\otimes e_{ii} + \sum_{i\neq j} e_{ii}\otimes e_{jj} + (\mbq^{-1} -\mbq)\sum_{i<j} \mbq^{2(i-j)}e_{ij} \otimes e_{ji}.
\]
It follows that $\wpi_V(S^2(a)) = \mbQ^2\wpi_V(a)\mbQ^{-2}$ for $a\in \{X_{ij}\}$, and hence for all $a\in \mcO_\mbq^{\R}(GL(N,\C))$. Hence 
\begin{eqnarray*}
\Tr_{\mbQ^2}((X^*X)^k) a &=& (\Tr\otimes \id)((\mbQ^2\otimes 1)(X^*X)^k (\wpi_V(a_{(2)})\wpi_V(S(a_{(3)})) \otimes a_{(1)}))  \\ &=&  (\Tr\otimes \id)(\wpi_V(S(a_{(3)})\mbQ^2 \wpi_V(a_{(2)})\otimes a_{(1)})(X^*X)^k)  \\ &=&  a_{(1)} (\Tr\otimes \id)(\mbQ^2\wpi_V(S^{-1}(a_{(3)})\wpi_V(a_{(2)})\otimes 1)(X^*X)^k) \\ &=& a \Tr_{\mbQ^2}((X^*X)^k).
\end{eqnarray*}

To obtain the result for $XX^*$, we note that the equation \eqref{EqCommDelt} holds with $X^*X$ replaced with $XX^*$ and $\Delta^{\cop}$ replaced by $\Delta$. It follows that the rest of the proof holds with $S$ replaced by $S^{-1}$, i.e. with $\mbQ$ replaced by $\mbQ^{-1}$. 
\end{proof}

\begin{Rem}
It is not hard to verify that in fact 
\[
\mbq^{1-N} \Tr_{\mbQ^2}((XX^*)^k) = \mbQ^{N-1}\Tr_{\mbQ^{-2}}((X^*X)^k).
\]
\end{Rem}

Upon applying the surjective map $\pi_T$, we obtain the following corollary.

\begin{Cor}\label{CorNewt}
The elements $\Tr_{\mbQ^2}((T^*T)^k)$ and $\Tr_{\mbQ^{-2}}((TT^*)^k)$  are central in $\mcO_\mbq^{\R}(T(N))$. 
\end{Cor} 

\begin{Cor}\label{CorCHGLT}
The absolute value squared $X^*X$ of the generator $X$ of $\mcO_\mbq^{\R}(GL(N,\C))$ satisfies a CH-identity. Similarly, the absolute value squared $T^*T$ of the generator $T$ of $\mcO_\mbq(T(N))$ satisfies a CH-identity.
\end{Cor} 
\begin{proof}
If $F$ is the characteristic CH-polynomial for the generator $Z$ of the RE $k_*$-algebra, it follows from the last part of  Theorem \ref{TheoQCHRE} and Lemma \ref{LemNewt} that the $\chi_X^{\mbq}(C_i)$ are still in the center of $\mcO_\mbq^{\R}(GL(N,\C))$, and hence $\chi_X^\mbq(F)$ is a CH-polynomial for $X^*X$. The result for $\mcO_\mbq(T(N))$ follows by applying the surjective map $\pi_T$.
\end{proof}

\section{The continuous field of quantum $GL(N,\C)$, quantum $U(N)$ and quantum $T(N)$}

\subsection{Statement of the main theorem}

We keep the notation of the previous section.

We can put on $k_*$ the closed spectral condition $\mbq\geq 0$. Then it is clear that $[2]_{\mbq} = \mbq + \mbq^{-1}$ defines a central control with respect to this condition, since $[n]_\mbq^{-1} \leq 1$ in each admissible representation. We denote $\msR$ for the corresponding s$^*$-algebra. Then the associated bounded universal C$^*$-envelope of $\msR$ is $C_0(\Theta)$ where $\Theta = \R_0^+ = (0,+\infty)$. 

The $k_*$-algebras $\mcO_\mbq(U(N)),\mcO_\mbq(T(N))$ and $\mcO_\mbq^{\R}(GL(N,\C))$ can be given closed spectral conditions as follows. In each case, the closed spectral condition $0\leq \mbq$ is assumed by default. 

\begin{Def} 
We endow $\mcO_\mbq^{\R}(GL(N,\C))$ and $\mcO_\mbq(U(N))$ with no extra closed spectral conditions. We endow $\mcO_{\mbq}(T(N))$ with the extra closed spectral conditions $T_{ii}\geq 0$, in which case we write $\mcO_\mbq(T_+(N))$. 
\end{Def} 

We also need on the above $k_*$-algebras a quantum control.

\begin{Lem}
The element $[2]_\mbq \oplus X \oplus \Det_\mbq(X)^{-1}$ is a quantum control for $\mcO_\mbq^{\R}(GL(N,\C))$. 

Similarly, $[2]_\mbq \oplus T \oplus \Det_\mbq(T)^{-1}$ is a quantum control for $\mcO_\mbq(T_+(N))$, while $[2]_\mbq$ is a quantum control for $\mcO_\mbq(U(N))$. 
\end{Lem}
\begin{proof}
The first statement follows immediately from the fact that $\mcO_\mbq^{\R}(GL(N,\C))$ is generated as a unital $*$-algebra by the set
\[
\mbq,\mbq^{-1},[n]_\mbq^{-1},X_{ij}, \Det_\mbq(X)^{-1},
\] 
where the $[n]_\mbq^{-1} \leq 1$ in all admissible representations. 

The same argument holds for $\mcO_\mbq(T_+(N))$. For $\mcO_\mbq(U(N))$, it follows from the fact that, since $U$ and $\Det_\mbq(U)$ are unitary, also all $\|U_{ij}\|\leq 1$ and $\|\Det_\mbq(U)\|\leq 1$ in any admissible representation.  
\end{proof}

Note that the specialisations  at $\mbq = 1$ have as spectrum respectively the locally compact spaces $GL(N,\C)$, $U(N)$ and $T_+(N)$, where $T_+(N)$ consists of uppertriangular matrices in $M_N(\C)$ with strictly positive elements on the diagonal.

Our aim is to prove the following theorem.

\begin{Theorem}\label{TheoMainMainExtra} 
The weak $\msR$-bialgebras $\mcO_\mbq(U(N))$, $\mcO_\mbq(T_+(N))$ and $\mcO_\mbq^{\R}(GL(N,\C))$ are strongly C$^*$-faithful, continuous, type $I$ Hopf $\msR$-algebras.
\end{Theorem}

The proof will be split over the sections below. For now, we note the following corollary of Proposition \ref{PropFieldHopf}, and Lemma \ref{LemNewt} and Corollary \ref{CorNewt}.

\begin{Cor} 
The completions $C_{0,\mbq}(U(N)),C_{0,\mbq}(T_+(N))$ and $C_{0,\mbq}(GL(N,\C))$ form fields of Hopf C$^*$-algebras.
\end{Cor}

\subsection{Proof for $\mcO_\mbq(U(N))$}

\begin{Theorem}\label{TheoHopfRU}
Together with its natural Hopf $k_*$-algebra structure, $\mcO_\mbq(U(N))$ is a strongly C$^*$-faithful, continuous, type $I$ Hopf $\msR$-algebra.
\end{Theorem} 
\begin{proof}
It is immediate that $\mcO_\mbq(U(N))$ is an $\msR$-algebra. The strong C$^*$-faithfulness can be checked for each fiber $\mcO_q(U(N))$. For $q=1$ this is immediate, while for $q\neq 1$ this follows from \cite[Proposition 5.1]{Koe91}. Also the type $I$-property can be checked on fibers, and follows from \cite[Theorem 5.3]{Koe91}. It is further immediate that the Hopf $k_*$-algebra structures $\Delta$ and $\varepsilon$ are $\msR$-morphisms.

It remains to show continuity. This can easily be deduced with the techniques for quantum $SU(N)$ treated in \cite{Nag00}. Let us briefly spell out the argument. Note that by C$^*$-faithfulness each $\mcO_q(U(N))$ embeds in its C$^*$-envelope $C_q(U(N))$, and that the latter is hence a compact quantum group in the sense of Woronowicz \cite{Wor98}. Hence with $\Phi$ as in Corollary \ref{CorExInv}, the Haar state on $C_q(U(N))$ must restrict to the functional $\Phi_q$ on $\mcO_q(U(N))$, and in particular $\Phi$ is positive, and $(\mcO_\mbq(U(N)),\Delta)$ of compact type. Moreover, the $\Phi_q$ are faithful as each $C_q(U(N))$ defines a coamenable compact quantum group, see e.g. \cite{Nag93} or \cite[Theorem 2.7.14]{NT13}. It now follows from Theorem \ref{TheoFieldCont} that $\mcO_{\mbq}(U(N))$ is continuous. 
\end{proof}

\subsection{Proof for $\mcO_\mbq(T_+(N))$}

We first introduce a weak $\msR$-algebra-structure on $\mcO_\mbq(H(N))$. Recall the elements $B_i \in \mcO_\mbq(H(N))$ from Lemma \ref{LemElB}. Note that $B_N$ is central.

\begin{Def}
We define $\mcO_\mbq(P(N))$ to be the weak $\msR$-algebra defined by $\mcO_\mbq(H(N))$ with closed spectral condition $Z\geq 0$ and quantum control $[2]_\mbq \oplus Z$. 

We define $\mcO_\mbq(PD(N))$ to be the weak $\msR$-algebra defined by $\mcO_\mbq(H(N))[B_N^{-1}]$ with closed spectral condition $Z\geq 0$ and quantum control $[2]_\mbq \oplus Z \oplus B_N^{-1}$. 
\end{Def}

We can view the completion $C_{0,\mbq}(P(N))$, resp.~ $C_{0,\mbq}(PD(N))$ as a quantization of the algebra of $C_0$-functions on the locally compact space of positive (resp.~ positive-definite) matrices.

Most of the properties we seek for $\mcO_\mbq(T_+(N))$ will be developed in parallel with $\mcO_\mbq(PD(N))$. 

\begin{Prop}\label{PropMainR}
The weak $\msR$-algebras $\mcO_\mbq(PD(N))$ and $\mcO_\mbq(T_+(N))$ are $\msR$-algebras.
\end{Prop}
\begin{proof}
Recall that $\mbQ = \Diag(\mbq^{N-i})_{1\leq i \leq N}$. By Proposition \ref{PropDensr} and the last part of Theorem \ref{TheoQCHRE},
\[
[2]_\mbq \oplus \Tr(\mbQ Z \mbQ) \oplus B_N^{-2} 
\]
is a central control for $\mcO_\mbq(PD(N))$. Hence $\mcO_\mbq(PD(N))$ is an s$^*$-algebra. It is clearly also an $\msR$-algebra. 

Replacing $Z$ by $T^*T$ using Corollary \ref{CorNewt}, we obtain the result for $\mcO_\mbq(T_+(N))$. 
\end{proof}

We now show that the above $\msR$-algebras are strongly C$^*$-faithful. 

\begin{Lem}\label{LemAdmRepT}
For $0<q$ and $q\neq 1$, we have that, under the $*$-isomorphism $\mcO_q(T(N)) \cong U_q(\mfu(N))$, any irreducible admissible $\mcO_q(T_+(N))$-representation corresponds to an irreducible type 1 representation of $U_q(\mfu(N))$. 
\end{Lem}
\begin{proof}
Let us transport the notion of admissibility from $\mcO_q(T_+(N))$ to $U_q(\mfu(N))$, so we consider (bounded) $U_q(\mfu(N))$-representations with the $K_i$ positive. By the commutation relation between the $K_i$ and $E_j,F_j$, it is clear that any irreducible admissible $U_q(\mfu(N))$-representation must have an orthonormal basis consisting of positive weight vectors. By the commutation relations between the $E_i,F_i$ it is moreover clear that any irreducible bounded $U_q(\mfu(N))$-representation must have a highest  and a lowest weight. It hence follows that the representation must be finite-dimensional. The lemma now follows immediately from Theorem \ref{TheoRepTConc}.
\end{proof}

\begin{Prop}\label{PropFaith}
The $\msR$-algebras $\mcO_\mbq(PD(N))$ and $\mcO_\mbq(T_+(N))$ are strongly C$^*$-faithful.
\end{Prop} 
\begin{proof}
We have to show that for each $q>0$ the s$^*$-algebras $\mcO_q(PD(N))$ and $\mcO_q(T_+(N))$ are C$^*$-faithful.

For $q=1$ this is clear, since $PD(N)$ is Zariski dense in the complex variety $M_N(\C)$, and $T_+(N)$ is Zariski dense in $T(N)$.

For $q\neq 1$ it is by Lemma \ref{LemCholInj} enough to show C$^*$-faithfulness for $\mcO_q(T_+(N))$. This follows\footnote{The result in  \cite[Theorem 7.13]{K-S97} can be immediately extended from $U_q(\mfsu(N))$ to $U_q(\mfu(N))$ by the embedding $U_q(\mfu(N)) \subseteq U_q(\mfsu(N))\otimes \C[D^{\pm 1}]$ used in Lemma \ref{LemCholInj}.} from \cite[Theorem 7.13]{K-S97}, Lemma \ref{LemAdmRepT} and the remark following Lemma \ref{LemIsoUT}.
\end{proof}

To deal with the type $I$-property for $\mcO_\mbq(PD(N))$, we will need uniqueness of the quantum Cholesky decomposition. We make some preparations.

\begin{Lem}\label{LemL}
For $0<q$ and $1\leq k \leq N$, write 
\[
L_k = F_{k}^*Z_{N,N+1} F_{k}\in (\otimes_{i=1}^N M_N(\C))\otimes \mcO_q(PD(N)),
\]
where 
\[
F_k =\hat{R}_{N-1,N} \ldots \hat{R}_{k+1,k+2} \hat{R}_{k,k+1},\qquad F_N = 1,\qquad F_{k}\in (\otimes_{i=1}^N M_N(\C)).
\]
Then the $L_k$ commute, and for $1\leq k \leq N$
\[
L_NL_{N-1}\ldots L_{k} = (Z_{N,N+1}F_k)^{N-k+1}.
\]
\end{Lem} 
\begin{proof}
Note first that the $\hat{R}$ are selfadjoint, so 
\[
F_k^* = \hat{R}_{k,k+1}\hat{R}_{k+1,k+2}  \ldots \hat{R}_{N-1,N}.
\]
Write  $M_k = Z_{N,N+1}F_k$. Then for $k\leq l\leq  N-2$, we have by the braid relations for $\hat{R}$ that 
\begin{eqnarray*} \hat{R}_{l,l+1} M_k &=& Z_{N,N+1} \hat{R}_{N,N+1} \ldots \hat{R}_{l,l+1}\hat{R}_{l+1,l+2}\hat{R}_{l,l+1}\hat{R}_{l-1,l} \ldots \hat{R}_{k,k+1} \\ &=&  Z_{N,N+1} \hat{R}_{N,N+1} \ldots \hat{R}_{l+1,l+2}\hat{R}_{l,l+1}\hat{R}_{l+1,l+2}\hat{R}_{l-1,l}\ldots \hat{R}_{k,k+1}\\ &=& M_k \hat{R}_{l+1,l+2}.
\end{eqnarray*} 
Hence for $k \leq N-1$ 
\begin{eqnarray*} 
M_k^{N-k+1} &=&  M_{k+1} \hat{R}_{k,k+1} M_k^{N-k-1} M_{k+1} \hat{R}_{k,k+1} \\ &=& M_{k+1} M_k^{N-k-1} \hat{R}_{N-1,N} M_{k+1}\hat{R}_{k,k+1} \\ &=& 
M_{k+1}^2 \hat{R}_{k,k+1} M_k^{N-k-2} \hat{R}_{N-1,N} M_{k+1}\hat{R}_{k,k+1} \\ &=&M_{k+1}^2  M_k^{N-k-2} \hat{R}_{N-2,N-1}\hat{R}_{N-1,N} M_{k+1}\hat{R}_{k,k+1}  \\ &=& \ldots \\ &=& M_{k+1}^{N-k} \hat{R}_{k,k+1} \ldots \hat{R}_{N-2,N-1}\hat{R}_{N-1,N} M_{k+1}\hat{R}_{k,k+1}\\ 
&=& M_{k+1}^{N-k}L_k.
\end{eqnarray*}
By induction, we find that $M_k^{N-k+1} = L_{N}\ldots L_k$. 

Note now that $L_N$ and $L_{N-1}$ commute by the reflection equation. Assume that the $L_l$ pairwise commute for all $l>k$, where $k\leq N-2$.  We are to show that $L_k$ commutes with the $L_l$ for $l>k$. For $l\geq k+2$, we have by induction that
\[
L_k L_l = \hat{R}_{k,k+1} L_{k+1} \hat{R}_{k,k+1} L_l  =   \hat{R}_{k,k+1} L_{k+1} L_l \hat{R}_{k,k+1} =  L_l \hat{R}_{k,k+1} L_{k+1}\hat{R}_{k,k+1}  = L_l L_k.
\]
For $l=k+1$, we note that by the braid relations
\begin{eqnarray*}
L_k\hat{R}_{k+1,k+2}  &=& \hat{R}_{k,k+1} \ldots \hat{R}_{N,N+1} Z_{N,N+1} \hat{R}_{N,N+1} \ldots  \hat{R}_{k+1,k+2} \hat{R}_{k,k+1} \hat{R}_{k+1,k+2} \\ &=& \hat{R}_{k,k+1} \ldots \hat{R}_{N,N+1} Z_{N,N+1} \hat{R}_{N,N+1} \ldots  \hat{R}_{k,k+1} \hat{R}_{k+1,k+2} \hat{R}_{k,k+1} \\ &=& \hat{R}_{k,k+1}\hat{R}_{k+1,k+2}\hat{R}_{k,k+1}\ldots  \hat{R}_{N,N+1} Z_{N,N+1} \hat{R}_{N,N+1} \ldots  \hat{R}_{k,k+1} \\ &=& \hat{R}_{k+1,k+2} L_k.
\end{eqnarray*}
Hence
\[
L_kL_{k+1} = L_k \hat{R}_{k+1,k+2}L_{k+2}\hat{R}_{k+1,k+2} = \hat{R}_{k+1} L_k L_{k+2} \hat{R}_{k+1,k+2} = \hat{R}_{k+1}L_{k+2}L_k\hat{R}_{k+1,k+2} = L_{k+1}L_k.
\]
\end{proof}

\begin{Lem}\label{LemPosZB}
Let $0<q$, and assume that $\pi$ is an admissible representation of $\mcO_q(P(N))$. Then $\pi(B_N) \geq 0$. Moreover, $\pi(Z)$ has a bounded inverse if and only if $\pi(B_N)$ has a bounded inverse. 
\end{Lem} 
\begin{proof}
Assume first that $\pi(B_N)$ is invertible. We have $\pi(Z)\geq 0$ by admissibility. By Theorem \ref{TheoQCHRE}, equation \eqref{EqCNB} and the hypothesis,  we know that $\pi(Z)$ satisfies a polynomial equation with commuting coefficients and invertible constant term. Hence $\pi(Z)$ is invertible. 

On the other hand, let $\pi$ be a general admissible representation of $\mcO_q(P(N))$. Using the notation of Lemma \ref{LemL}, we have that each $\pi(L_k) \geq 0$. Since the $L_k$ commute, also the product 
\[
\pi(L_N\ldots L_1) = (\pi(Z)_{N,N+1}\hat{R}_{N-1,N}\ldots \hat{R}_{12})^N \geq 0.
\]
However, by Lemma  \ref{LemIdBN} we know that $B_N$ is a slice of $\pi(L_N\ldots L_1)$ by a non-zero positive multiple of a vector state. It follows that $\pi(B_N)\geq 0$. If moreover $\pi(Z)$ is invertible, it follows that each of the $\pi(L_k)$ has a bounded inverse, and so there exists $\delta>0$ with $\delta \leq \pi(L_N\ldots L_1)$. Being a slice, it follows that for some $\delta'>0$ we have $\delta' \leq \pi(B_N)$, hence $\pi(B_N)$ has a bounded inverse.
\end{proof}

\begin{Prop}\label{PropAdmExt}
Any admissible representation of $\mcO_\mbq(PD(N))$ extends uniquely through the quantum Cholesky map $\chi_T^\mbq$ to an admissible representation of $\mcO_\mbq(T_+(N))$, preserving the set of intertwiners. Moreover, this induces a C$^*$-isomorphism between the C$^*$-completions
\[
\chi_{T,0}^\mbq: C_{0,\mbq}(PD(N)) \overset{\cong}{\rightarrow} C_{0,\mbq}(T_+(N)).
\] 
\end{Prop} 
\begin{proof} 
Let $\pi$ be an admissible representation, which we may assume factors through some $\mcO_q(PD(N))$ for $0<q$. As $\pi$ is defined on $B_N^{-1}$, it follows that $\pi(B_N)$ has a bounded inverse, and the easy direction in Lemma \ref{LemPosZB} gives that $\pi(Z)$ is a positive invertible operator. In particular, for all $\xi \in \C^n\otimes \Hsp_{\pi} = \oplus_{i=1}^N \Hsp_\pi$ we have
\[
\langle \xi, \pi(Z)\xi \rangle \geq \|\pi(Z)^{-1}\|^{-1} \|\xi\|^2.
\] 
Recall now again that the $Z_{ij}$ with $1\leq i,j\leq K\leq M$ generate a copy of $\mcO_\mbq(H(K))$. Let 
\[
Z_{\leq K} \in M_K(\mcO_\mbq(H(K))) \subseteq M_K(\mcO_\mbq(H(N)))
\]
be the principal $K\times K$-block of $Z$. Then we clearly have for each $\xi \in \oplus_{i=1}^K \Hsp_\pi$ that
\[
\langle \xi, \pi(Z_{\leq K})\xi \rangle \geq \|\pi(Z)^{-1}\|^{-1} \|\xi\|^2.
\] 
Let $\pi_K$ be the restriction of $\pi$ to $\mcO_\mbq(H(K))$. Then we have $\pi_K(Z_{\leq K}) \geq \|\pi(Z)^{-1}\|^{-1}$, and by the other direction of Lemma \ref{LemPosZB} we have that $\pi_K(B_K) = \pi(B_K)$ is positive with bounded inverse. By Proposition \ref{PropIdHT} this leads to an admissible representation of $\mcO_\mbq(T_+(N))$. Clearly this construction preserves sets of intertwiners. 

Finally, since $[2]_\mbq \oplus \Tr_{\mbQ^2}(Z) \oplus B_N^{-1}$, resp.~ its image $[2]_\mbq \oplus \Tr_{\mbQ^2}(T^*T) \oplus \Det_{\mbq}(T)^{-2}$ under the quantum Cholesky map, is a central control for $\mcO_\mbq(PD(N))$, resp.~ $\mcO_\mbq(T_+(N))$, it is  by functional calculus clear that the resulting map $C_{b,\mbq}(PD(N)) \rightarrow C_b^{\mbq}(T_+(N))$, and hence also $C_{0,\mbq}(PD(N)) \rightarrow C_{0,\mbq}(T_+(N))$, is an isomorphism.
\end{proof}

\begin{Prop}\label{PropMainTypeI}
The $\msR$-algebras $\mcO_\mbq(PD(N))$ and $\mcO_\mbq(T_+(N))$ are type $I$.
\end{Prop} 

\begin{proof}
The $\msR$-algebra $\mcO_\mbq(T_+(N))$ is type $I$ since any irreducible admissible representation, which necessarily factors through $\mbq = q$ for some $q>0$, is finite-dimensional by Lemma \ref{LemAdmRepT}. The same holds for $\mcO_\mbq(PD(N))$ by Proposition \ref{PropAdmExt}.
\end{proof}


\begin{Theorem}
The type $I$ $\msR$-algebras $\mcO_\mbq(T_+(N))$ and $\mcO_\mbq(PD(N))$ are continuous. 
\end{Theorem} 
\begin{proof}
By Proposition \ref{PropAdmExt} it is sufficient to prove this for $\mcO_\mbq(PD(N))$. 

Consider the adjoint coaction $\Ad_\mbq$ from Lemma \ref{LemDefCoad}, and let $E$ be the associated conditional expectation. With $C=  \Tr_{\mbQ^2}(Z)$, note that $[2]_\mbq \oplus C\oplus B_N^{-1}$ is a coinvariant central control. Moreover, for each $q>0$ the conditional expectation $E_q$ is faithful, as it consist of integrating out a continuous action by a coamenable compact quantum group. Note that we also use here that the completion of $\Ad_\mbq$ will still be injective, by the existence of a counit.  

Let now $C_0(\Omega) := E_0(C_{0,\mbq}(PD(N)))$. It follows from Theorem \ref{TheoFieldCont} that $\mcO_\mbq(PD(N))$ will be continuous once we can show that the natural projection map $\Omega \rightarrow \R_0^+$ is open.

However, consider in $\msZ(\mcO_\mbq(PD(N)))$ the elements 
\[
C_k' = \mbq^{(N+1)k} \binom{N}{k}_{\mbq}E(B_k),
\]
which by Remark \ref{RemMud} are just rescalings of the $C_k$. We claim that $\mbq$ and the $C_k'$ separate the representations of $\mcO_\mbq(T(N))$. Indeed, it is sufficient for this that the $C_{k,q}'$ separate the representations of $\mcO_q(T(N))$. For $q\neq 1$ this follows immediately from Corollary \ref{CorHC}. For $q= 1$  this follows from Corollary \ref{CorHC} and the fact that a unitary conjugacy class of a positive-definite matrix is completely determined by its (unordered) list of eigenvalues (with multiplicities).

Let now again $(0,\infty)_q = q^{\Z}$ for $q>0$ and $q\neq 1$, and let $(0,\infty)_1 = (0,\infty)$. Let 
\[
\Omega' = \{(q,x_1,\ldots,x_N)\mid q>0,x_i >0 \textrm{ and }x_i/x_j \in (0,+\infty)_{q^2}\setminus\{1\} \textrm{ for }q\neq 1\textrm{ and }i\neq j\} \subseteq   \R_0^+\times (\R_0^+)^N.
\]
and consider the continuous map 
\[
f: \Omega' \rightarrow \R_0^+\times (\R_0^+)^N,\quad (q,x) \mapsto (q,e_1(x),\ldots,e_N(x)).
\]
Then clearly $\Omega = f(\Omega')$ by Corollary \ref{CorHC}. Since the projection map 
\[
\Omega' \rightarrow \R_0^+,\quad (q,x)\mapsto q
\]
is open, the same holds for the projection map $\Omega \rightarrow \R_0^+$. 
\end{proof}

\begin{Theorem}\label{TheoHopfR}
Together with its natural Hopf $k_*$-algebra structure, $\mcO_\mbq(T_+(N))$ is a strongly C$^*$-faithful, continuous, type $I$ Hopf $\msR$-algebra.
\end{Theorem} 
\begin{proof}
We have already shown that $\mcO_\mbq(T_+(N))$ is a strongly C$^*$-faithful, continuous type $I$ $\msR$-algebra. We are only to show the existence of the $\msR$-compatible Hopf algebra structure. However, it is immediate that $\Delta$ and $\varepsilon$ are admissible $*$-homomorphisms for $\mcO_\mbq(T_+(N))$. 
\end{proof}

\begin{Rems}
\begin{itemize}
\item For $N=2$, a similar theorem had been proven by a more concrete approach in \cite{DCF17}.
\item Together with Theorem \ref{TheoHopfRU}, Theorem \ref{TheoHopfR} realizes in the C$^*$-algebraic setting the Poisson-Drinfeld duality between the quantizations of the dual Poisson-Lie groups $U(N)$ and $T_+(N)$. 
\end{itemize}
\end{Rems}

\subsection{Proof for $\mcO_\mbq^{\R}(GL(N,\C))$}

\begin{Prop}\label{PropMainR}
The weak $\msR$-algebra $\mcO_\mbq^{\R}(GL(N,\C))$ is an $\msR$-algebra.
\end{Prop}
\begin{proof}
Using Lemma \ref{LemNewt}, this follows as in Proposition \ref{PropMainR}.
\end{proof}

We have the following strong version of the quantum $QR$-decomposition.  

\begin{Prop}\label{PropULDecomp} Under the quantum $QR$-decomposition map $\pi_{UT}^{\mbq}$, any admissible representation of  $\mcO_\mbq^{\R}(GL(N,\C))$ lifts uniquely to an admissible representation of  $\mcO_\mbq(U(N))\msRtimes \mcO_\mbq(T_+(N))$. Moreover, the completion 
\[
\pi_{UT,0}^\mbq: C_{0,\mbq}(GL(N,\C)) \rightarrow C_{0,\mbq}(U(N)) \underset{C_0(\Theta)}{\otimes} C_{0,\mbq}(T_+(N))
\]
becomes an isomorphism of $C_0(\Theta)$-algebras.
\end{Prop} 
\begin{proof}
Let us write $\msA = \mcO_\mbq^{\R}(GL(N,\C))$, $\msU = \mcO_\mbq(U(N))$ and $\msC = \mcO_\mbq(T_+(N))$. 

Recall from \eqref{EqLeftTrans} that we have on $\msA$ the left translation coaction $\gamma$ by $\msU$. Since $\msA$ has the coinvariant central control $[2]_\mbq \oplus C \oplus \Det_\mbq(X)$, it follows that $\gamma$ completes to a coaction 
\[
\gamma: A_0 \rightarrow U_0 \underset{C_0(\Theta)}{\otimes} A_0
\]
(note that this is meaningful by Remark \ref{RemTypI}). Let $\msB$ be the set of coinvariants for $\gamma$, and let 
\[
E = (\Phi\otimes \id)\gamma: \msA \rightarrow \msB
\]
be the associated conditional expectation with completion $E_0$. Since the completion $\gamma_0$ of $\gamma$ is still injective by the existence of a counit, and since $\Phi_0$ is pointwise faithful on $U_0$, it follows that $E_0$ is a pointwise faithful conditional expectation on $A_0$. 

Let now $\pi$ be an irreducible admissible representation of $\msA$, and $\wpi$ the corresponding representation of $A_0$. Since $E_0$ is faithful, we can find an irreducible representation $\wpi'$ of $B_0$ such that $\wpi \preceq \Ind_{B_0}^{A_0}(\wpi')$ where we induce from $E_0$. Since $\msA$ has a coinvariant central control, it follows that $\wpi'$ factors through some $B_{<M}$ and hence is the completion of some (bounded) irreducible representation $\pi'$ of $\msB$. However, by Lemma \ref{LemMSB} and Proposition \ref{PropAdmExt} such a representation can be uniquely extended to an irreducible admissible representation $\pi''$ of $\msC$. Since $\Ind_{B_0}^{A_0}(\wpi')$ factors through $A_{\leq M}$, it is easily seen that $\Ind_{B_0}^{A_0}(\wpi')$ is a subrepresentation of the restriction to $\msA$ of the induction of $\pi''$ to $\msU \msRtimes \msC$ via its natural conditional expectation onto $\msC$. However, since invariance of a subspace under $\msB$ implies invariance under $\msC$, it follows that in fact $\Ind_{B_0}^{A_0}(\wpi')$ is the restriction to $A_0$ of a representation of $U_0 \underset{C_0(\Theta)}{\otimes} C_0$, and that the images in this representation are the same.

It is now clear that $\pi$ extends in a unique way to an admissible irreducible representation of $\msU \msRtimes \msC$, and that this sets up a one-to-one correspondence between admissible irreducible representations. By compatibility of the central controls, this in fact identifies $A_0 \cong U_0 \underset{C_0(\Theta)}{\otimes} C_0$.
\end{proof}

\begin{Theorem}\label{TheoHopfRU}
Together with its natural Hopf $k_*$-algebra structure, $\mcO_\mbq^{\R}(GL(N,\C))$ is a strongly C$^*$-faithful, continuous, type $I$ Hopf $\msR$-algebra.
\end{Theorem} 
\begin{proof}
The strong C$^*$-faithfulness already follows from Lemma \ref{LemMorQR}, while the type $I$-condition and continuity follow from Proposition \ref{PropULDecomp}. The existence of a compatible coproduct and counit is again trivial.
\end{proof}

\section*{Appendix}

\appendix
\renewcommand{\thesection}{A.\arabic{section}}
\setcounter{section}{0}

\section{ Quantum generating families in the C$^*$-algebraic setting}\label{Ap1}

Note that although the proof of Theorem \ref{TheoMainMainExtra} required to work over $k$, the conclusion is still valid over $\C[\mbq,\mbq^{-1}]$ since the C$^*$-completions are the same. The benefit is that one can find quantum controls whose entries generate the $*$-algebra. In this appendix, we want to see how this generating property on the algebraic level can be compared to Woronowicz's generating property on the C$^*$-level \cite{Wor95}. 

Recall from \cite{BJ83} the notion of \emph{element $T$ affiliated to a C$^*$-algebra $A_0$}, $T\eta A_0$, and from \cite{Wor95} the \emph{z-transform} $z_T = T(1+T^*T)^{-1/2} \in M(A_0)$. Recall further that any non-degenerate $*$-homomorphism between C$^*$-algebras $\alpha: A_0 \rightarrow M(B_0)$ can be extended to a map $\alpha^{\eta}: A_0^{\eta}\rightarrow B_0^{\eta}$. 

\begin{Def} 
Let $A_0$ be a separable C$^*$-algebra, and let $\mbX \eta M_N(A_0)$ be an element affiliated with $M_N(A_0)$. We say that $\mbX$ is a \emph{quantum generating family} for $A_0$ if the following holds: for every representation $(\pi,\Hsp_{\pi})$ of $A_0$ and every separable C$^*$-algebra $B_0$, the condition $\pi(\mbX) \eta M_N(B_0)$ implies that $\pi(A_0) \subseteq M(B_0)$ with  
\[
\pi: A_0 \rightarrow M(B_0)
\] 
non-degenerate.

If $\msA$ is just a unital $*$-algebra, we will say that $X \in M_N(\msA)$ is a \emph{quantum generator} if $\msA$ is generated by the $X_{ij}$ as a unital $*$-algebra.
\end{Def} 

It is in practice quite hard to show that an affiliated element $\mbX$ is a generating quantum family. We will in this appendix consider this question in the setting of s$^*$-algebras which are quantum generated by a quantum control.

We will fix in the following a countably generated s$^*$-algebra $\msA$, so that $A_0$ is separable. We first present some general results on realizing elements of $\msA$ as affiliated elements of $A_0$.

\begin{Lem}\label{LemIdAff}
The collection of maps $\wtheta_M^{\eta}:A_0^{\eta} \rightarrow A_{\leq M}$ provide an identification of vector spaces 
\[
A_0^{\eta} \cong \underset{\longleftarrow}{\lim}\, A_{\leq M},
\] 
where we take the \emph{algebraic} projective limit.
\end{Lem} 
\begin{proof}
Straightforward using Lemma \ref{LemEqMulti} and the characterization of affiliated elements through the $z$-transform, see \cite[Section 1]{Wor95}.    
\end{proof}

We can hence identify $A_0^{\eta}$ as a subspace 
\[
A_0^{\eta} \subseteq \prod_{\pi \in P} B(\Hsp_{\pi}). 
\]
Remark that this gives a natural unital $*$-algebra structure on $A_0^{\eta}$. More precisely, we obtain an isomorphism 
\[
A_0^{\eta} \cong M(A_c).
\]

If now $\mbS,\mbT$ are two affiliated operators, denote by $\mbS\dotplus \mbT$ the closure of their sum, and by $\overline{\mbS\mbT}$ the closure of their product. In general, it is not clear if these are again affiliated elements, but the existence of local units $e_M \in \msZ(A_c)$ immediately gives the following lemma. 

\begin{Lem}\label{LemCompOK}
With respect to $\mbS\dotplus \mbT$, $\overline{\mbS\mbT}$ and the adjoint operation, the space $A^{\eta}_0$ becomes a unital $*$-algebra. Moreover, the natural isomorphism $A^{\eta}_0\rightarrow M(A_c)$ is an isomorphism of $*$-algebras, and $A_c$ is a joint core for all $\mbT\in A^{\eta}_0$. 
\end{Lem} 

This result also holds of course for elements in $M_N(\msA)$. 

Let now $\msA$ be an s$^*$-algebra, and let $X = (X_{ij}) \in M_N(\msA)$ be a quantum control whose entries generate $\msA$ as a unital $*$-algebra. Then from the above, we can in particular make sense of 
\[
\mbX \eta M_N(A_0).
\] 
However, it is not clear unfortunately with this information alone if $\mbX$ generates $A_0$ in the sense of Woronowicz. The problem is that it is not clear to what extent functional calculus on $\mbX \eta A_0$ can be used to produce elements in $A_c$.

\begin{Lem}\label{LemConvAbs} 
Let $\msA$ be an s$^*$-algebra, and assume that $X \in M_N(\msA)$ with associated affiliated element $\mbX \eta M_N(A_0)$. Put 
\[
z = (\Tr\otimes \id)(X^*X) \in \msA,
\] 
and let $\mbz \eta A_0$ be the associated affiliated element. Let $(f_k)_k$ be any sequence of bounded continuous functions on $\R^+$ which converges to the identity function on $\R^+$ uniformly on compacts, and put 
\[
z_k = (\Tr\otimes \id)(f_k(\mbX^*\mbX)) \in A_0.
\] 
Then $z_k \rightarrow \mbz$ in the almost uniform topology
\end{Lem}

\begin{proof}
Represent $A_0$ faithfully and non-degenerately on a Hilbert space $\Ksp$, and write $\Ksp_0 = A_c \Ksp$. Then $A_c\Ksp$ is a core for $\mbz$, and clearly for each $\xi\in A_c\Ksp$ one has 
\[
z_k\xi \rightarrow \mbz\xi,
\] 
as one can factorize through some fixed $A_{\leq M}$ (with $M$ depending only on $\xi$). It follows that $z_k \rightarrow \mbz$ in the strong resolvent sense, that is $(i-z_k)^{-1} \rightarrow (i-\mbz)^{-1}$ in norm, by \cite[Theorem A.6]{Tak03}. In particular, $f(z_k) \rightarrow f(\mbz)$ in norm for each function $f \in C_0(\R^+)$, in particular for $x \mapsto  1-x(1+x^2)^{-1/2}$. It follows that the $z$-transform of $z_k$ converges (in norm) to the $z$-transform of $\mbz$, and hence $z_k$ converges in the almost uniform topology to $\mbz$. 
\end{proof}

\begin{Theorem}\label{TheoMain}  
Let $\msA$ be an s$^*$-algebra. Assume that $Y_1,\ldots, Y_n \in \msZ(\msA)$, $X\in M_N(\msA)$, and assume that $Y_1\oplus \ldots\oplus Y_n \oplus X$ is at the same time a quantum control and quantum generator for $\msA$. Let $k_*$ be the unital algebra generated by the $Y_i$, and let $Q\in M_N(k_*)$ be invertible and selfadjoint. Assume further that $\Tr_{Q^2}(X^*X)$ is central in $\msA$. Then $\mbY_1\oplus \ldots \oplus \mbY_n\oplus \mbX$ is a quantum generating family for $A_0$.
\end{Theorem}
  
\begin{proof} 
Put $W = Y_1\oplus \ldots \oplus Y_n \oplus X$, and write $\mbW \eta M_{N+n}(A_0)$ for the associated affiliated element. Put $z = \Tr_{Q^2}(X^*X)$ and $r = (1+|\mbY|^2+\mbz)^{-1}$. Then $r\in A_0$ by Proposition \ref{PropDensr} and the remarks following Proposition \ref{PropCentTiet}. Further put $Z = QX$. Then $\mbZ$ is the closure of the product $\mbQ\mbX$ by Lemma \ref{LemCompOK}.

Let $\rho$ be a non-degenerate representation of $A_0$ on a (separable) Hilbert space $\Hsp$, and let $B_0 \subseteq B(\Hsp)$ be a separable C$^*$-algebra, represented non-degenerately on $\Hsp$. Assume that $\rho(\mbW) \eta M_{N+n}(B_0)$. By \cite[Theorem 4.2]{Wor95}, it is sufficient to prove the following:
\begin{itemize}
\item $\rho(r) \in M(B_0)$ with $\rho(r)B_0$ dense in $B_0$. 
\item If $\rho'$ is another non-degenerate representation of $A_0$ on $\Hsp$, and $\rho'(\mbW) = \rho(\mbW)$, then $\rho' = \rho$.
\end{itemize}

For the first point, note that if $\msR$ equals $k_*$ with central control $Y_1\oplus \ldots\oplus Y_n$, and $C_0(\Theta)$ the associated bounded universal C$^*$-envelope, the  $\mbY_i$ clearly generate $C_0(\Theta)$ in the sense of Woronowicz. It hence follows that $\rho(\mbQ) \eta M_N(B_0)$, and so $\rho(\mbZ) \eta M_N(B_0)$. The first point follows now immediately from Lemma \ref{LemConvAbs}. 

For the second point, let $f_n$ be a sequence of functions of compact support on $(0,1]$ tending uniformly to the identity function on $[0,1]$. Then $f_n(r)$ converges in norm to $r$, and $f_n(r) \in A_c$. It follows that $\rho(\mbW)$ is bounded when restricted to $\C^{N'}\otimes \overline{(\rho(f_n(r))\Hsp)}$, and hence the restriction of $\rho$ or $\rho'$ to this subspace factors through some $A_{\leq M}$. Since however the $X_{ij}$ and $Y_i$ generate $\msA$, it follows that $\rho_{\mid \overline{f_n(r)\Hsp}} = \rho_{\mid\overline{f_n(r)\Hsp}}'$. Since the span of all $f_n(r)\Hsp$ is dense in $\Hsp$, it follows that $\rho = \rho'$. 
\end{proof}

From Lemma \ref{LemNewt} we now obtain the following.

\begin{Cor}
Let $k_*$ be the $*$-algebra $\C[\mbq,\mbq^{-1}]$, and consider $\mcO_\mbq^{\R}(GL(N,\C))$ as in Definition \ref{DefMainObj} but defined over $k_*$. Let $C_{0,\mbq}(GL(N,\C))$ be the associated bounded universal C$^*$-envelope, and write $D = \Det_\mbq(X)$. Then $[2]_\mbq \oplus \mbX \oplus \mbD^{-1}$ is a generator for $\mcO_0^\mbq(GL(N,\C))$ in the sense of Woronowicz.
\end{Cor}

\section{Connection with Meyer's theory of C$^*$-hulls}\label{Ap2}

In \cite{Mey17}, R. Meyer introduced a general notion of C$^*$-hull for a $*$-algebra with respect to a given class of (possibly unbounded) representations on Hilbert modules. In \cite[Section 7]{Mey17}, the special case was treated of $*$-algebras with sufficiently many bounded representations. In this appendix, we clarify the connection between our notions and the ones introduced in \cite{Mey17}. We refer to \cite{Mey17} for all necessary definitions and background.

Let $(\msA,\msS,\msP)$ be a unital $*$-algebra with closed spectral conditions $\msS$ and C$^*$-filtration $\msP$ on the set $P$ of admissible irreducible representations. Let $\mcN(\msA)$ be the set of all C$^*$-seminorms on $\msA$. Then $\mcN(\msA)$ is directed, and we can form for $q\in \mcN(\msA)$ the separation-completion $\mcA_q$ of $\msA$ with respect to $q$. For $M>0$, we will denote $q_{M} = \|-\|_M$, so that $\mcA_{q_M} = A_{\leq M}$ in the notation of Section \ref{Sec1}. We let $\mcN_{\msS}(\msA)$ be the subset of $\mcN(\msA)$ consisting of all $q$ for which the projection map $\msA \rightarrow \mcA_q$ is admissible. Then $\mcN_{\msS}(\msA)$ is still directed. 

Let $\mcA$ be the (algebraic) projective limit of the $\mcA_q$ along $\mcN_{\msS}(\msA)$, so that $\mcA$ is a pro-C$^*$-algebra. On the other hand, let $A$ be the projective algebraic limit pro-C$^*$-algebra $A = \underset{\longleftarrow}{\lim}\, A_{\leq M}$. From the second condition in Definition \ref{DefCfiltr}, it follows immediately that the natural projection map
\[
j:\mcA \rightarrow A,\quad (a_q)_q \mapsto (a_{q_M})_M
\]
is a topological $*$-isomorphism, and that under this isomorphism the maximal C$^*$-subalgebra $\mcA_b$ of $\mcA$ is sent to $A_b$. 

In \cite[Definition 7.14]{Mey17}, the C$^*$-algebra $C_0(\mcA)$ is defined as the C$^*$-subalgebra of $\mcA_b$ consisting of all elements $a\in \mcA_b$ for which there exists $q \in \mcN_\msS(\msA)$ such that $\pi_b(a) = 0$ for all irreducible representations $\pi \in P$ which do not lift to $\mcA_q$. Letting the $q$ range over the $q_M$, it is easily seen that $A_0 \subseteq j(C_0(\mcA))$. On the other hand, since for each $q \in \mcN_{\msS}(\msA)$ there is $M>0$ with $q\leq q_{M}$, we have $j(C_0(\mcA)) = A_0$. In what follows, we will hence drop the isomorphism $j$. 

Let now $\Rep_{\msS}(\msA)$ be the class of all admissible representations of $\msA$ on Hilbert modules which are \emph{locally bounded} in the sense of \cite[Definition 7.4]{Mey17}. We have the following theorem.

\begin{Theorem}
Let $\msA$ be an s$^*$-algebra. Then the C$^*$-algebra $A_0$ is the $C^*$-hull of $\msA$ with respect to $\Rep_{\msS}(\msA)$.
\end{Theorem}
\begin{proof}
Let $\widetilde{\mcA}$ be the projective limit pro-C$^*$-algebra along $\mcN(\msA)$, and $\mathcal{J}$ be the ideal along the quotient map $\widetilde{\mcA}\rightarrow \mcA \cong A$. Then it is easily seen that $\mathcal{J}$ is a closed two-sided $*$-ideal, and that $\Rep_{\msS}(\msA)$ coincides with the class of representations $\Rep_b(\msA,A)$ as specified below Theorem 7.16 in \cite{Mey17}. By \cite[Theorem 7.17]{Mey17}, we are left with proving that $A_0$ is dense in $A$ with respect to the natural locally convex topology on the latter. However, this follows from the identification $A = A_0^{\eta}$ established in  Lemma \ref{LemIdAff}.
\end{proof}

\end{document}